\documentclass{article}
\usepackage{amsthm,amsmath,amssymb}
\usepackage[margin=1in,bottom=1in]{geometry} 
\usepackage{caption}
\usepackage{subcaption}
\usepackage{graphicx}
\usepackage{epstopdf}
\usepackage{color}
\usepackage{framed}
\usepackage{lscape}
\usepackage{rotating}
\usepackage{algorithm}
\usepackage[noend]{algpseudocode}
\usepackage{grffile}
\usepackage{multirow}
\usepackage{xcolor}
\usepackage{comment}
\usepackage{xr}
\usepackage{esint}
\usepackage{mathtools}
\usepackage{cancel}

\usepackage{tikz}
\usetikzlibrary{math}

\newtheorem{thm}{Theorem}[section]

\newtheorem{lem}[thm]{Lemma}

\newtheorem{cor}[thm]{Corollary}
\newtheorem{prop}[thm]{Proposition}

\theoremstyle{definition}
\newtheorem{rmk}{Remark}

\usepackage[colorlinks=true,linkcolor=black,citecolor=black,urlcolor=black]{hyperref}

\newcommand{\err}{\mathrm{err}}

\newcommand{\diam}{\mathrm{diam}}

\newcommand{\mean}{\mathrm{m}}

\newcommand{\meanvolt}{\mathrm{SV}}

\newcommand{\meanwass}{\mathrm{SW}}
\newcommand{\wass}{\mathrm W}
\newcommand{\monge}{\mathrm M}

\renewcommand{\bf}{\mathbf{f}}
\newcommand{\bg}{\mathbf{g}}

\newcommand{\Vtrap}{\boldsymbol{\mathrm V}}
\newcommand{\bV}{\boldsymbol{\mathrm V}}
\newcommand{\mV}{\mathrm V}

\newcommand{\cen}{\mathrm{cen}}
\newcommand{\bfend}{\mathbf{f}}

\newcommand{\fcen}{f_\mathrm{cen}}

\newcommand{\wcen}{\mathbf{w}_\mathrm{cen}}
\newcommand{\bfcen}{\mathbf{f}_\mathrm{cen}}

\newcommand{\EE}{\mathbb{E}}

\newcommand{\R}{\mathbb{R}}

\newcommand{\la}{\langle}

\newcommand{\ra}{\rangle}

\newcommand{\EMD}{\mathrm{EMD}}
\newcommand{\KM}{\mathrm{KM}}

\newcommand{\Prob}{\mathbb{P}}

\newcommand{\what}{\widehat}
\newcommand{\wtilde}{\widetilde}

\renewcommand{\a}{\mathbf{a}}
\renewcommand{\b}{\mathbf{b}}

\renewcommand{\c}{\mathbf{c}}
\renewcommand{\d}{\mathbf{d}}

\renewcommand{\u}{\mathbf{u}}
\renewcommand{\v}{\mathbf{v}}

\newcommand{\x}{\mathbf{x}}
\newcommand{\y}{\mathbf{y}}

\newcommand{\w}{\mathbf{w}}
\newcommand{\z}{\mathbf{z}}
\newcommand{\e}{\mathbf{e}}

\newcommand{\calR}{\mathcal{R}}

\newcommand{\calP}{\mathcal{P}}

\newcommand{\calU}{\mathcal{U}}
\renewcommand{\S}{\mathbb{S}}

\newcommand{\calM}{\mathcal{M}}
\newcommand{\calT}{\mathcal{T}}
\newcommand{\BB}{\mathbb{B}}

\newcommand{\V}{\mathcal{V}}
\renewcommand{\SS}{\mathbf{S}}

\newcommand{\sign}{\mathrm{sign}}

\newcommand{\AC}{\mathcal{A}}

\begin{document}
\title{
\textbf{On metrics robust to noise and deformations}
}
\author{\textbf{William Leeb} \\ \\
School of Mathematics \\
University of Minnesota, Twin Cities  \\
Minneapolis, MN}
\date{}
\maketitle

\begin{abstract}
We study the properties of a family of distances between functions of
a single variable. These distances are examples of 
integral probability metrics, and have been used previously for
comparing probability measures on the line; special cases include the
Earth Mover's Distance and the Kolmogorov Metric. We examine their
properties for general
signals, proving that they are robust to a broad class of deformations.
We also establish corresponding robustness results for the induced sliced
distances between multivariate functions.
Finally, we establish error bounds for approximating the 
univariate metrics
from finite samples, and prove that these approximations are robust to additive
Gaussian noise. The results are illustrated in numerical experiments,
which include comparisons with Wasserstein distances.
\end{abstract}

\section{Introduction}

Many tasks in statistics and machine learning require specification
of a metric that measures the similarity between data vectors. 
For example, the goal of clustering is to group 
points together that are close and separate points that are far,
where ``close'' and ``far'' are determined by a certain metric
or similarity measure that is designed to capture relevant
features
\cite{rokach2009survey, verde2007dynamic, aggarwal2014data}.
Such a method's effectiveness depends crucially on the choice of metric.
An appropriate metric will be robust to noise and to
irrelevant deformations of the input data, so that only the
``meaningful'' characteristics of each data vector inform the distance.

A widely-used class of metrics are the Wasserstein distances
for comparing probability distributions $f$ and $g$ defined over
a metric space $\mathcal{X}$.
Informally, the Wasserstein distance between $f$ and $g$ is equal to
the minimal cost of transforming $f$ into $g$ by rearranging the mass,
where the cost is determined by the metric on $\mathcal{X}$
\cite{villani2003topics, villani2008optimal};
we review the precise definition of Wasserstein distances
in Section \ref{sec:wasserstein}.
The Wasserstein distances
are popular metrics in a range of machine learning and statistical applications
\cite{panaretos2019statistical, peyre2019computational, santambrogio2015optimal,
bonneel2015sliced, rabin2011wasserstein, rubner2000earth,
levina2001earth, bernton2019parameter, rigollet2019uncoupled,
moosmuller2023linear, cloninger2019people}.

The goal of this work is to begin to address the question of
whether the favorable properties of Wasserstein distances
are shared by other families of metrics. More specifically,
we consider two known
``robustness'' properties exhibited by the Wasserstein distances.
The first result states that if there is a smooth bijection
$\Phi: \mathcal{X} \to \mathcal{X}$
and
\begin{align}
g(x) = f(\Phi(x)) \frac{d \Phi}{d x}(x),
\end{align}
where $\frac{d \Phi}{d x}(x)$ is the Radon-Nikodym derivative of $\Phi$, then
\begin{align}
\label{eq:wasserstein_deformation}
\wass_p(f,g) \le \sup_{x \in \mathcal{X}} d_{\mathcal{X}}(x,\Phi(x)),
\end{align}
where $\wass_p$ denotes the $p$-Wasserstein distance.
This bound is straightforward to prove (indeed, it is almost
tautological from the Monge formulation of Wasserstein distances),
and a self-contained
argument may be found in Section \ref{sec:wasserstein}; see also
\cite{leeb2016holder} for another argument when $p=1$.
Informally,
\eqref{eq:wasserstein_deformation} tells us that the Wasserstein distances
are robust to ``small'' deformations of a distribution,
where the ``size'' of a deformation is the maximum distance
that any point in the domain may be moved.

The second property of Wasserstein
distances that we will consider is from the recent paper \cite{rao2020wasserstein},
and
suggests that the $p$-Wasserstein distance $\wass_p$ is a good choice
for clustering tomographic projection images that arise
from cryo-electron microscopy (cryo-EM), a technique for molecular reconstruction
that is increasingly used in structural biology
\cite{singer2020computational, bendory2020single, doerr2016single}. Suppose
$f$ and $g$ are functions of two variables that are each projections
of a common three-dimensional volume $F$; that is,
\begin{align}
f(x,y) = \int_{\R} F(R_f(x,y,z)) dz,
\end{align}
where $R_f$ is an orthogonal transformation; and similarly for $g$.
Then \cite{rao2020wasserstein} proves that
\begin{align}
\label{eq:wasserstein_angle}
\min_{R \in SO(2)} \wass_p(f,g \circ R) \le \theta(R_f,R_g),
\end{align}
where $\theta(R_f,R_g)$ is the angle between the projection directions
of $R_f$ and $R_g$.
In fact, in Section \ref{sec:wasserstein} we observe that Wasserstein
distances exhibit a somewhat more general robustness property,
namely, the Wasserstein distance between projections
of two functions is robust to deformations of those
functions. Though we have not seen this result
stated in the literature, it follows rather trivially from
\eqref{eq:wasserstein_deformation} and the analysis in the
paper \cite{rao2020wasserstein}; see Theorem \ref{thm:wasserstein_projections}.

In many problems where Wasserstein distances are used,
the transportation problem solved
by the distance computation is not of interest in and of itself;
rather, one only requires a distance that is robust
to distortions of the data.
Furthermore, because Wasserstein distances are defined between positive measures,
it can be awkward to use or analyze them in settings where
the observed signals take on negative values, as is typically the case
when observations are corrupted by additive noise.
It is natural to ask whether there are other families
of metrics with similar robustness properties
as the Wasserstein metrics,
and which are robust to additive noise.
Furthermore, while the motivation in \cite{rao2020wasserstein}
for the bound \eqref{eq:wasserstein_angle}
is comparing two projection images in cryo-electron
microscopy, 
there are, in fact, numerous scientific problems for which
the measurement modality only permits
observing projections of an object, from which
the object itself must be reconstructed
\cite{deans2007radon, natterer1986mathematics,
helgason2010integral, herman2009fundamentals, singer2013two, coifman2008graph}.
It is therefore of interest to consider the properties of
distances used to compare tomographic projections.

In this paper, we approach these questions by
studying a family of simple metrics between
single-variable functions, including tomographic projections of multivariate functions.
These metrics are induced by norms, denoted by $\|f\|_{V^p}$,
which are the $p$-norms of the Volterra operator
(the indefinite integral operator) applied to $f$.
We call the norm $\| f \|_{V^p}$ the \emph{Volterra $p$-norm},
and its induced metric the \emph{Volterra $p$-distance},
or \emph{Volterra $p$-metric}.
The Volterra distances have been used previously for
comparing probability measures on the
real line \cite{muller1997integral, maejima1987ideal}; however,
unlike the Wasserstein distances, the Volterra distances
are defined between all integrable functions, not just between probability
measures; this makes it natural to consider
their robustness to additive noise
in addition to geometric distortions.

We show that the Volterra distances exhibit robustness properties
like \eqref{eq:wasserstein_deformation} and \eqref{eq:wasserstein_angle},
as well as a property generalizing \eqref{eq:wasserstein_deformation}
and \eqref{eq:wasserstein_angle}: namely, that when comparing
univariate tomographic projections of multivariate functions, the distance
is robust to deformations in the higher dimensional space.
In fact, the robustness bounds we prove for Volterra metrics
are stronger than those for Wasserstein distances;
more precisely, when $p>1$, the Volterra $p$-metric is
bounded by a concave, non-linear function of the deformation's size,
suggesting that the Volterra distances are more robust to large deformations.
Finally, we analyze discrete approximations to the Volterra distances,
showing that they
converge to their continuous counterparts for a broad class of
``well-behaved'' functions,
while also being robust to additive Gaussian noise.

Of course, the Volterra metrics are limited in their applicability,
as they are only defined between univariate functions;
the Wasserstein distances, by contrast, are naturally
defined between densities of any number of variables.
It is therefore of interest to ask whether there are
Volterra-type metrics for multivariate functions that exhibit similar
properties.
We leverage the theory
for Volterra distances to study distances between
multivariate functions induced
by ``slicing'', a technique that has been used for
Wasserstein distances \cite{rabin2011wasserstein, bonneel2015sliced,
kolouri2019generalized,
kolouri2018sliced, deshpande2018generative, nietert2022statistical}.
The robustness properties of the Volterra distances
immediately induce corresponding robustness properties of the sliced
Volterra distances.

Central to our analysis is the variational characterization of
Volterra distances, which is
well-known (and for which we provide a self-contained
proof in Section \ref{sec:variational}). This characterization shows that the Volterra distances
are a special case of the class of \emph{maximum mean discrepancies},
also known as \emph{integral probability metrics}, which
are typically used for comparing probability densities
\cite{gretton2012kernel, cherief2020mmd, arbel2019maximum,
muller1997integral, maejima1987ideal, agrawal2021optimal,
sriperumbudur2012empirical}. If $f$ and $g$ are two functions
on a measure space $\mathcal{X}$, a maximum mean discrepancy
is a metric of the form
\begin{align}
\label{eq:mmd}
d_{\mathcal{F}}(f,g) = \sup_{h \in \mathcal{F}} \int_{\mathcal{X}} h(x)(f(x) - g(x)) dx,
\end{align}
where $\mathcal{F}$ is a specified class of test functions.
It is a consequence of the Kantorovich-Rubinstein Theorem
\cite{kellerer1985duality, dudley2018real, kantorovich1958space,
kantorovich1982functional, edwards2011kantorovich}
that the $1$-Wasserstein distance, or Earth Mover's Distance,
between densities $f$ and $g$ is equal to the distance \eqref{eq:mmd}
when $\mathcal{F}$ is the set of $1$-Lipschitz functions
with respect to a metric on $\mathcal{X}$.
By contrast, when $\mathcal{X}$ is an interval,
the Volterra $p$-distance is obtained by taking $\mathcal{F}$
to be the space of functions
with derivative in $L^{p/(p-1)}$. The Volterra $1$-distance is
equal to the $1$-Wasserstein metric in one variable, but
the Volterra $p$-distance and the $p$-Wasserstein metric
are not equal for any $p > 1$.

The remainder of the paper is structured as follows:

\begin{enumerate}

\item
Section \ref{section:preliminaries} reviews basic definitions, notation, and properties,
including of
the Lebesgue norms, the Volterra operator, push-forwards and deformations,
Wasserstein and sliced Wasserstein distances, and
tomographic projections.
Theorem \ref{thm:wasserstein_projections} provides a general robustness result
on Wasserstein distances, which, though it follows easily
from existing work, does not appear to have been stated previously.

\item
Section \ref{section:volterra} defines the basic objects of study, namely
the Volterra distances
between univariate functions,
the sliced Volterra distances between multivariate
functions, and the trapezoidal rule approximations
to the Volterra distances. It also contains a self-contained proof
of the well-known variational characterization of
Volterra distances (Proposition \ref{prop:variational}).

\item
Section \ref{section:properties} contains statements
of the theorems on the robustness
of the Volterra and sliced Volterra distances.
Theorem \ref{thm:projpert} is a general robustness result
for comparing univariate projections; Theorems \ref{thm:rotations2D} and
\ref{thm:deformation} give stronger bounds for the case of rotations (i.e.\ changes
in projection angle) and monotonically increasing deformations,
respectively.
Theorems \ref{thm:sliced_deformations} and \ref{thm:sliced_rotations2D} show robustness of
the sliced Wasserstein distances.

\item
Section \ref{section:discrete} analyzes the trapezoidal rule approximation
to the Volterra distance. Theorems \ref{thm:convergence_lipschitz}
and \ref{thm:convergence_smooth}
show that when samples are taken from
sufficiently regular functions, the approximation
converges to the true Volterra distance. Theorem
\ref{thm:noise} and Corollary \ref{cor:convergence} show that the
approximations are robust to additive Gaussian noise.

\item
Section \ref{section:numerical} shows the results of numerical
experiments illustrating the theoretical results,
including comparisons between the Volterra,
Wasserstein, and Lebesgue distances.

\item
Sections \ref{section:proofs_properties} and
\ref{section:proofs_discrete} contain proofs
of the results from Sections \ref{section:properties}
and \ref{section:discrete}, respectively.
Theorem \ref{thm:general} in Section \ref{section:proofs_properties}
gives a more general statement about the Volterra distances,
from which the theorems in Section \ref{section:properties}
may all be derived; the variational characterization
in Proposition \ref{prop:variational} is central
to the analysis.

\item
Section \ref{section:conclusion} concludes the paper, providing
a summary and topics for future research.

\end{enumerate}

\section{Preliminaries}
\label{section:preliminaries}

This section introduces the basic definitions and notation
that will be used in the rest of the paper.
Familiarity with basic concepts of measure and integration,
e.g.\ at the level of \cite{folland1999real}, will be assumed.

\subsection{The Lebesgue $p$-norms}

We recall the standard definition of the Lebesgue $p$-norm.
Let $F : \R^d \to \R$ be a measurable function.
If $1 \le p < \infty$,
then the Lebesgue $p$-norm of $F$ is defined by
\begin{align}
\|F\|_{L^p} = \left( \int_{\R^d} |F(\x)|^p d\x\right)^{1/p}.
\end{align}
For $p = \infty$, we define
\begin{align}
\|F\|_{L^\infty} = \operatorname*{ess \, sup}_{\x \in \R^d} |F(\x)|.
\end{align}
We denote by $L^p$ the set of all functions $F : \R^d \to \R$ with
$\|F\|_{L^p} < \infty$, and by $L^p(A)$ the subset of $L^p$ containing only functions supported
on $A$.
As is well-known, if $F$ is supported in a bounded set $A \subset \R^d$, then
$\|F\|_{L^p} \le \|F\|_{L^q} |A|^{1/p-1/q}$
if $p \le q$, where $|A|$ denotes the Lebesgue measure of $A$;
in particular, $L^p(A) \subset L^q(A)$.
If $F$ is in $L^p$ and $G$ is in $L^q$, where $1/p + 1/q = 1$, we define
their inner product by
\begin{align}
\langle F, G\rangle = \int_{\R^d} F(\x) G(\x) d\x.
\end{align}

We also define the normalized Lebesgue $p$-norm $\|\x\|_{\ell_p}$ for vectors $\x$ in $\R^n$.
When $1 \le p < \infty$,
\begin{align}
\|\x\|_{\ell_p} = \left( \frac{1}{n}  \sum_{j=1}^{n} |x_j|^p \right)^{1/p},
\end{align}
and when $p=\infty$,
\begin{align}
\|\x\|_{\ell_\infty} = \max_{1 \le k \le n} |x_k|.
\end{align}
Note the normalization by $n$ when $p < \infty$. With this convention,
$\|\x\|_{\ell_p} \le \|\x\|_{\ell_q}$
whenever $p \le q$.

We will denote the unnormalized $2$-norm of a vector $\x$ in $\R^d$
by
\begin{align}
|\x| = \left(\sum_{j=1}^{d} x_j^2 \right)^{1/2}.
\end{align}
Note that we do not normalize by $1/d$ in this case.
The normalized norm $\|\x\|_{\ell_p}$ will be used when $\x$ is a vector
of samples of a function,
and the unnormalized norm $|\x|$ when $\x$ is a variable.
We define the inner product between two vectors $\x$ and $\y$ in $\R^d$ by
\begin{align}
\langle \x,\y\rangle = \sum_{j=1}^d x_j y_j.
\end{align}

\subsection{Trapezoidal rule approximation}
\label{sec:trap_rule}

If $f$ is a function on an interval $[a,b]$
and $n \ge 1$ is an integer,
the trapezoidal rule approximation to $\int_{a}^{b} f(x) dx$ is defined as
\begin{align}
T_n(f,a,b) = \frac{b-a}{2n} \sum_{k=0}^{n-1} (f(a_k) + f(a_{k+1})),
\end{align}
where
\begin{align}
\label{eq:subinterval_ends}
a_k = a + \frac{k}{n}(b-a), \quad 0 \le k \le n.
\end{align}
It is well-known that if $f$ is $C^3$, then
\begin{align}
\label{eq:trap_error}
\left| T_n(f,a,b) - \int_{a}^{b} f(x) dx \right|
\le \frac{1}{12} \|f''\|_{L^\infty} \frac{(b-a)^3}{n^2};
\end{align}
see, e.g., \cite{dahlquist1974}.

\subsection{Absolute continuity}

Suppose $a < b$.
Recall that a function $G$ on $[a,b]$ is said to be absolutely continuous
if it can be written as
\begin{align}
G(x) = G(a) + \int_{a}^{x} g(t) \,dt
\end{align}
for a function $g$ in $L^1([a,b])$. If $G$ is absolutely continuous,
then it is differentiable almost everywhere, and $G' = g$ where the
derivative exists.
We denote by $\AC_0$ the set of absolutely continuous functions $G$ satisfying
$G(b)=0$; these functions may be written as
\begin{align}
G(x) = -\int_{x}^{b} g(t) \, dt
\end{align}
where $g = G'$ almost everywhere.
For brevity, whenever $G$ is in $\AC_0$, $G'$ will denote any function such that
$G(x) = -\int_{x}^{b} G'(t) \, dt$.

The following result is standard (e.g., see Section 3.5 of \cite{folland1999real}):

\begin{thm}[Integration by parts]
If $F$ and $G$ are absolutely continuous functions on $[a,b]$, then
\begin{align}
\int_{a}^{b} (F'(x) G(x) + F(x) G'(x) ) dx = F(b)G(b) - F(a) G(a).
\end{align}

\end{thm}

\subsection{The Volterra operator}

The Volterra operator $\V$ is defined on $L^1([a,b])$ by
\begin{align}
(\V f)(x) = \int_{a}^{x} f(t) dt.
\end{align}
We note that this is only the simplest of a large family of related operators
that have been widely studied \cite{gohberg1970theory}.
Importantly, if $f$ is in $L^1([a,b])$, $\V f$ is in $L^\infty$, with
$\|\V f\|_{L^\infty} \le \|f\|_{L^1}$;
furthermore, $\V f$ is, by definition, absolutely continuous when $f$ is in $L^1([a,b])$.

The adjoint transform $\V^*$ is given by
\begin{align}
(\V^* f)(x) = \int_{x}^{b} f(t) dt.
\end{align}
This operator satisfies
\begin{align}
\langle \V f, g \rangle = \langle f, \V^* g \rangle
\end{align}
where $f$ and $g$ are two functions in $L^1([a,b])$.

\subsection{Push-forwards and $\epsilon$-deformations}
\label{sec:pushforwards}

Let $\Omega \subset \R^d$ be
a non-empty, bounded, open set.
Suppose that $\mu$ is a finite, signed measure on $\Omega$,
and that $\Psi : \Omega \to \R^d$ is a measurable function.
Then $\Psi$ induces a signed measure on $\Psi(\Omega)$, denoted $\Psi_\sharp \mu$,
defined by
\begin{align}
(\Psi_\sharp \mu)(E) = \mu(\Psi^{-1}(E)),
\end{align}
for measurable sets $E \subset \Psi(\Omega)$. The measure $\Psi_\sharp \mu$
is referred to as the \emph{push-forward} of $\mu$ induced from $\Psi$ \cite{peyre2019computational}.
Note that $(\Psi_\sharp \mu)(\Psi(\Omega)) = \mu(\Psi^{-1}(\Psi(\Omega))) = \mu(\Omega)$;
that is, the push-forward preserves the total measure.

Now suppose that $\mu$ is induced from a function $f$ supported on $\Omega$; that is,
$\mu(E) = \int_{E} f(\x) d\x$ for all measurable $E \subset \Omega$.
Suppose too that $\Psi$ is a diffeomorphism between $\Omega$ and $\Psi(\Omega)$;
that is, it is $C^1$, one-to-one, and $\det (\nabla \Psi(\x)) \ne 0$,
where $\nabla \Psi(\x)$ denotes the
Jacobian matrix of $\Psi$ at $\x$.
Let $\Phi = \Psi^{-1}$ denote the functional inverse of $\Psi$.
Then $\Psi_\sharp \mu$ has density
$f(\Phi(\x))|\det(\nabla \Phi (\x))|$.
Indeed, using the change of variables $\x = \Phi(\u)$ and $d\x = |\det(\nabla \Phi(\u))|d\u$,
\begin{align}
(\Psi_\sharp \mu)(E)
= \mu(\Psi^{-1}(E))
= \int_{\Psi^{-1}(E)} f(\x) d\x
= \int_{E} f(\Phi(\u))|\det(\nabla \Phi(\u))| d\u.
\end{align}
When convenient, we will write $(\Psi_\sharp f)(\x) = f(\Phi(\x))|\det(\nabla \Phi(\x))|$,
or $f_\Phi(\x) = f(\Phi(\x))|\det(\nabla \Phi(\x))|$.

When the diffeomorphism $\Psi$ does not move points by more than
a value $\epsilon > 0$ (that is, $|\x - \Psi(\x)| \le \epsilon$
for all $\x$ in $\Omega$, or equivalently,
$|\x - \Phi(\x)| \le \epsilon$
for all $\x$ in $\Psi(\Omega)$), we will refer to $\Psi_\sharp f$
as an \emph{$\epsilon$-deformation} of $f$; we will also refer to $\Psi$ itself
as an $\epsilon$-deformation of $\Omega$.
For example, if $\u$ is a fixed unit vector, then the function
$\Psi(\x) = \x + \epsilon \u$ is
an $\epsilon$-deformation.

\subsection{Tomographic projections and the Radon transform}

Let $\calU = (\u_1,\dots,\u_r) \in \S^{d-1} \times \cdots \times S^{d-1}$
(where $\S^{d-1} \subset \R^d$
is the $(d-1)$-dimensional unit sphere)
denote an ordered collection of $r$ unit vectors in $\R^d$.
Let $\u_{r+1},\dots,\u_{d}$ denote $d-r$ orthonormal vectors that are
orthogonal to $\u_1,\dots,\u_r$.
We define the operator $\calP_\calU$ by
\begin{align}
(\calP_\calU F)(t_1,\dots,t_r)
    = \int_{\R^{d-r}} F(t_1\u_1 + \dots + t_r \u_r  + s_1 \u_{r+1} + \dots + s_{d-r} \u_{d})
        ds_1 \cdots ds_{d-r}.
\end{align}
We refer to $\calP_\calU F$ as the \emph{tomographic projection} of $F$
onto the subspace spanned by $\u_1,\dots,\u_r$.
When $r=1$, we will denote the tomographic projection of $F$ onto
the span of a unit vector $\u$ by $\calP_\u F$.
Note that in this case, the \emph{Radon transform}
$\calR F : \R \times \S^{d-1}$ of the function $F$
is defined by $(\calR F)(t,\u) = (\calP_\u F)(t)$.
For more background on these transforms,
see, for example, the references \cite{natterer1986mathematics, helgason2010integral}.

\subsection{Wasserstein distances}
\label{sec:wasserstein}

If $F$ and $G$ are probability densities on a subset $\Omega \subset \R^d$,
their $p$-Wasserstein distance $\wass_p(F,G)$ (also known as the
Kantorovich distance) is defined as
\begin{align}
\wass_p(F,G) = \min_{\Pi \in \mathcal{M}(F,G)} 
    \left(\int_{\Omega} \int_{\Omega} |\x-\y|^p d\Pi(\x,\y)\right)^{1/p},
\end{align}
where $\mathcal{M}(F,G)$ denotes the space of all probability measures on
$\Omega \times \Omega$ with marginals equal to $F$ and $G$, respectively
\cite{villani2003topics, villani2008optimal}.
That is, $\Pi \in \mathcal{M}(F,G)$ if for all measurable $E \subset \Omega$,
\begin{align}
\Pi(E \times \Omega) = \int_E F(\x) d\x,
\end{align}
and
\begin{align}
\Pi(\Omega \times E) = \int_E G(\y) d\y.
\end{align}

Informally, $\wass_p(F,G)$ is the minimal cost of rearranging
a unit of mass with distribution $F$ into one with distribution
$G$, where the cost of moving mass between locations $\x$ and $\y$
is $|\x - \y|^p$.
The distance $\wass_1(F,G)$ is also known as the \emph{Earth Mover's Distance (EMD)}
between the probability measures $F$ and $G$ \cite{villani2003topics, villani2008optimal},
which we will also denote by $\EMD(F,G)$.
The Wasserstein distances and their variants have been widely used in statistics,
machine learning, image processing, and related areas
\cite{panaretos2019statistical, peyre2019computational, santambrogio2015optimal,
bonneel2015sliced, rabin2011wasserstein, rubner2000earth,
levina2001earth, bernton2019parameter, rigollet2019uncoupled,
moosmuller2023linear, cloninger2019people}.

The Wasserstein distance
is a relaxation of the Monge distance, defined by
\begin{align}
\monge_p(F,G)
= \min_{\Phi \in \mathcal{T}(F,G)} 
    \left(\int_{\Omega} |\x-\Phi(\x)|^p F(\x) \,d\x \right)^{1/p}
\end{align}
where $\mathcal{T}(F,G)$ is the set of all functions $\Phi : \Omega \to \Omega$
such that $\int_{E} G(\x) d\x = \int_{\Phi^{-1}(E)} F(\x) d\x$;
that is, all functions $\Phi$ which push $F$ onto $G$, in the sense described
in Section \ref{sec:pushforwards}.
Indeed, any function $\Phi$ in $\calT(F,G)$ induces a measure $\Pi_\Phi$
in $\calM(F,G)$, with
\begin{align}
\label{eq:wp_monge}
\int_{\Omega} \int_{\Omega} |\x-\y|^p d\Pi_\Phi(\x,\y)
= \int_{\Omega} |\x-\Phi(\x)|^p F(\x) \,d\x,
\end{align}
and hence $\wass_p(F,G) \le \monge_p(F,G)$. (In fact, when $\monge_p(F,G)$
is finite, equality holds; see \cite{santambrogio2015optimal}.)
This implies \eqref{eq:wasserstein_deformation};
indeed, if $\Phi$ is a smooth bijection on $\Omega$
satisfying $|\x - \Phi(\x)| \le \epsilon$
for all $\x$, and $F_\Phi(\x) = F(\Phi(\x))|\det(\nabla \Phi(\x))|$,
then $\Phi$ is contained in $\calT(F,F_\Phi)$, and so
\begin{align}
\label{eq:wasserstein_deformation2}
\wass_p(F,F_\Phi) \le \monge_p(F,F_\Phi) \le
\left(\int_{\Omega} |\x-\Phi(\x)|^p F(\x) \,d\x \right)^{1/p}
\le \epsilon \left(\int_{\Omega} F(\x) \,d\x \right)^{1/p}
= \epsilon.
\end{align}

In fact, a more general robustness result may be shown,
which we state now.
The proof is nearly identical to that found
in \cite{rao2020wasserstein}.
\begin{thm}
\label{thm:wasserstein_projections}
Suppose $F$ is a probability density supported on
a set $\Omega$, and let $\Phi : \Omega \to \Omega$ be an $\epsilon$-deformation.
For $\calU = (\u_1,\dots,\u_r)$, where $\u_1,\dots,\u_r$ are orthonormal,
let $\calP = \calP_\calU$ denote the tomographic projection operator.
Then for all $p \ge 1$,
\begin{align}
\wass_p(\calP F, \calP F_\Phi) \le \epsilon.
\end{align}

\begin{proof}
An identical proof to that of Lemma 1 in \cite{rao2020wasserstein} shows that
$\wass_p(\calP F, \calP F_\Phi) \le \wass_p(F,F_\Phi)$
(note that the left side refers to transportation in $\R^{d-1}$,
and the right side to $\R^d$).
The bound then follows from \eqref{eq:wasserstein_deformation2}.
\end{proof}

\end{thm}

When $d=1$, it is known \cite{santambrogio2015optimal}
that  $\wass_p(F,G)$ may be written as follows:
\begin{align}
\label{eq:wasserstein_1d}
\wass_p(F,G) = \| (\V F)^{-1} - (\V G)^{-1} \|_{L^p}.
\end{align}
Here, $(\V F)^{-1}$ denotes the functional inverse of
$\V F$, defined as
\begin{align}
(\V F)^{-1}(x) = \inf\{t \in [a,b] \ : \ (\V F)(t) \ge x\}.
\end{align}
When $p=1$, it is also well-known that
$\wass_1(F,G) = \| \V F - \V G \|_{L^1}$.

\subsection{Sliced Wasserstein distance}
\label{sec:sliced_wasserstein}

The \emph{sliced Wasserstein distance}
\cite{rabin2011wasserstein}
is defined between two probability
densities $F$ and $G$ in $\R^{d}$ as follows:
\begin{align}
\meanwass_{p,\eta}(F,G)
= \left(\int_{\S^{d-1}}
            \wass_p^p(\calP_\w F, \calP_\w G) d\eta(\w) \right)^{1/p}
\end{align}
where $\eta$ is a suitable probability measure over $\S^{d-1}$.
That is, $\meanwass_{p,\eta}(F,G)$ is obtained by averaging the
distances between the one-dimensional projections of $F$ and $G$
over all directions.

One advantage of these metrics is
that each one-dimensional
distance $\wass_p(\calP_\w F, \calP_\w G)$
may be computed rapidly by using the formula \eqref{eq:wasserstein_1d}.
Sliced Wasserstein distances have been the subject of
considerable research activity \cite{rabin2011wasserstein, bonneel2015sliced,
kolouri2019generalized,
kolouri2018sliced, deshpande2018generative, nietert2022statistical}.

\section{Volterra distances}
\label{section:volterra}

In this section, we introduce our basic objects
of study, the Volterra distances and their discrete approximatons,
and the sliced Volterra distances.
Section \ref{sec:volterra_defined} defines the Volterra norms and 
corresponding distances
for functions of a single variable; Section \ref{sec:variational}
reviews the variational characterization of these distances;
Section \ref{sec:sliced_defined} defines the sliced Volterra distances;
and Section \ref{sec:trapezoidal} defines the trapezoidal rule approximations
to the Volterra distances.

\subsection{The Volterra norms and distances}
\label{sec:volterra_defined}

Let $f$ be in $L^1([a,b])$. For any value $1 \le p  \le \infty$,
we define the following norm, which we will refer to as the \emph{Volterra $p$-norm}:
\begin{align}
\|f\|_{V^p} = \|\V f\|_{L^p}.
\end{align}
Concretely, when $1 \le p < \infty$,
\begin{align}
\|f\|_{V^p}
= \left( \int_{a}^{b} \left| \int_{a}^{x} f(t) dt \right|^p dx\right)^{1/p},
\end{align}
and when $p = \infty$,
\begin{align}
\|f\|_{V^\infty}
=\operatorname*{ess \, sup}_{a \le x \le b} \left|\int_{a}^{x} f(t) dt \right|.
\end{align}
Note that, because $\V f$ is in $L^\infty([a,b])$, the Volterra $p$-norm of $f$ is finite
for any function $f$ in $L^1([a,b])$.
If $f$ and $g$ are two functions in $L^1([a,b])$, we will refer to
$\|f-g\|_{V^p}$ as the \emph{Volterra $p$-distance},
or $\emph{Volterra $p$-metric}$,
between $f$ and $g$.

\begin{rmk}
When $p=\infty$ and $f$ and $g$ are two probability densities,
the Volterra $\infty$-distance is known as the Kolmogorov Metric between $f$ and $g$ 
\cite{gibbs2002choosing}:
$\KM(f,g) = \|f - g\|_{V^\infty}$.
The KM arises in the context of goodness-of-fit
testing in statistics \cite{massey1951kolmogorov}. 
\end{rmk}

\begin{rmk}
When $p=1$ and $f$ and $g$ are two probability densities,
the Volterra $1$-distance is equal to the Earth Mover's Distance between $f$ and $g$
described in Section \ref{sec:wasserstein}:
$\EMD(f,g) = \|f-g\|_{V^1}$.
When $p>1$, however, the $p$-Wasserstein distance $\wass_p(f,g)$
is equal to $\|(\V f)^{-1} - (\V g)^{-1}\|_{V^p}$,
and is generally \emph{not} 
equal to the Volterra $p$-distance $\|\V f - \V g\|_{L^p}$.
\end{rmk}

\subsection{Variational formulation of $\|f\|_{V^p}$}
\label{sec:variational}

The following result is an alternate formulation of the Volterra norm
that will be useful for analysis.
It essentially appears as Theorem 1 in \cite{maejima1987ideal};
we provide a self-contained proof for the reader's convenience.

\begin{prop}
\label{prop:variational}
Let $1 \le p  \le \infty$ and let $q$ be the conjugate exponent:
\begin{align}
\frac{1}{p} + \frac{1}{q} = 1.
\end{align}
Then for any  function $f$ in $L^1([a,b])$,
\begin{align}
\label{eq:variational}
\|f\|_{V^p}
= \sup_{G \in \AC_0 : \|G'\|_{L^q} \le 1} \langle f, G \rangle.
\end{align}
\end{prop}

\begin{proof}[Proof of Proposition \ref{prop:variational}]
By duality of $L^p$ and $L^q$, we have:
\begin{align}
\|f\|_{V^p} = \|\V f\|_{L^p}
= \sup_{g:\|g\|_{L^q} \le 1} \int_{a}^{b}  (\V f)(x) g(x)dx
= \sup_{g:\|g\|_{L^q} \le 1} \langle \V f ,g \rangle
= \sup_{g:\|g\|_{L^q} \le 1} \langle f , \V^* g \rangle.
\end{align}
Any function of the form $\V^* g$ is contained in $\AC_0$,
and any function $G$ in $\AC_0$ is of the form
$G = \V^* g$ where $g = G'$ almost everywhere. Consequently:
\begin{align}
\|f\|_{V^p}
= \sup_{g:\|g\|_{L^q} \le 1} \langle f , \V^* g \rangle
= \sup_{G \in \AC_0:\|G'\|_{L^q} \le 1} \langle f , G \rangle,
\end{align}
which completes the proof.
\end{proof}

\subsection{Sliced Volterra distances}
\label{sec:sliced_defined}

Analogous to the sliced Wasserstein distances reviewed
in Section \ref{sec:sliced_wasserstein},
we define the \emph{sliced Volterra distance} between functions $f$
and $g$ on $\R^d$:
\begin{align}
\meanvolt_{p,\eta}(f,g) = \left(\int_{\S^{d-1}}
            \|\calP_\w f - \calP_\w g\|_{V^p}^p d\eta(\w) \right)^{1/p},
\end{align}
where $\eta$ is a probability measure over $\S^{d-1}$.
That is, $\meanvolt_{p,\eta}(f,g)$ is obtained by averaging the
Volterra $p$-distances between the one-dimensional projections of $f$ and $g$
over all directions.

\subsection{Trapezoidal rule approximation to $\|f\|_{V^p}$}
\label{sec:trapezoidal}

Suppose $f$ is a function on $[a,b]$, and we are given samples of $f$
on an equispaced grid of points in $[a,b]$, from which we wish to approximate
$\|f\|_{V^p}$. That is, let
$a_0 < a_1 < \dots < a_n$ be equispaced points in $[a,b]$ defined by
\eqref{eq:subinterval_ends}, that is,
\begin{align}
a_k = a + \frac{k}{n}(b-a), \quad 0 \le k \le n.
\end{align}
Note that $a_0 = a$ and $a_n = b$. We suppose we are given the values of
$f(a_k)$, $0 \le k \le n$, or possibly noisy estimates of these,
and wish to approximate $\|f\|_{V^p}$.

To this end, we introduce some convenient notation.
If $\v$ is a vector in $\R^{n+1}$ and $1 \le p < \infty$, we define the norm
\begin{align}
\| \v \|_{\tau_p}
= \frac{b-a}{2n} \sum_{k=0}^{n-1} (\left|\v[k] \right|^p + \left|\v[k+1] \right|^p)
= \left( \frac{b-a}{n} \sum_{k=1}^{n-1} \left|\v[k] \right|^p
    + \frac{b-a}{2n} \left| \v[0] \right|^p
    + \frac{b-a}{2n} \left| \v[n] \right|^p \right)^{1/p}.
\end{align}
When $p = \infty$, we define $\| \v \|_{\tau_\infty} = \| \v \|_{\ell_\infty}$,
that is,
\begin{align}
\| \v \|_{\tau_\infty} = \max_{0 \le k \le n} \left| \v[k] \right|.
\end{align}
If $f$ is a function on $[a,b]$ and $\bf$ is a vector with entries
$\bf[k] = f(a_k)$, $0 \le k \le n$, then $\| \bf \|_{\tau_p}$ is the
trapezoidal rule approximation to $\|f\|_{L^p}$ when
$1 \le p < \infty$, and
$\| \bf \|_{\tau_\infty}$ is an approximation to $\|f\|_{L^\infty}$.

Suppose $n$ is a positive integer. We define the following discrete
Volterra operator $\Vtrap : \R^{n+1} \to \R^{n+1}$
on a vector $\x$ by $(\Vtrap \x)[0] = 0$, and, for $1 \le k \le n$,
\begin{align}
(\Vtrap \x)[k]
= \frac{b-a}{2n} \sum_{j=0}^{k-1} (\x[j] + \x[j+1]).
\end{align}
We then define the discrete Volterra $p$-norm of $\x$ as
\begin{align}
\| \x \|_{\nu_p} = \| \Vtrap \x \|_{\tau_p}.
\end{align}
The interpretation of this quantity may be understood as follows.
Suppose $f$ is a function on $[a,b]$, and let
$\bf$ be the vector in $\R^{n+1}$ with entries $\bf[k] = f(a_k)$,
for $0 \le k \le n$, where $a_k$ are defined in \eqref{eq:subinterval_ends}.
Then $(\Vtrap \bf)[k]$ is the trapezoidal rule approximation
to $(\V f)(a_k)$, and $\| \bf \|_{\nu_p}$ approximates $\|f\|_{V_p}$.
The approximation error will be bounded in Section \ref{section:discrete}.
We remark that one can also define an approximate Volterra
norm based on $n$ equispaced midpoint samples on $[a,b]$, which will
have the same theoretical guarantees as the trapezoidal rule
approximation considered here.

\section{Properties of Volterra distances}
\label{section:properties}

This section analyzes the use of Volterra distances
for comparing functions on an interval and univariate projections of functions
on $\R^d$. Theorem \ref{thm:projpert} establishes
a property similar to Theorem \ref{thm:wasserstein_projections}
for Wasserstein distances, namely, that the
distance between projections of a function and its
$\epsilon$-deformation is controlled by $\epsilon$.
Notably, however, the bound
for Volterra metrics is a concave, non-linear function of
$\epsilon$, indicating that the Volterra
metrics may be more robust than the Wasserstein distances
to large deformations. The bound in Theorem \ref{thm:projpert}
can be strengthened in certain special cases: Theorem \ref{thm:rotations2D}
analyzes the case of rotations in the plane, and
Theorem \ref{thm:deformation} analyzes the setting
of univariate functions related by a monotonically
increasing deformation.
Theorems \ref{thm:sliced_deformations} and \ref{thm:sliced_rotations2D}
state corresponding results for the sliced Volterra distances.

\begin{thm}
\label{thm:projpert}
Let $A$ and $B$ be  non-empty, bounded, open sets in $\R^d$,
with $r = \diam(A \cup B)$.
Let $F: \R^d \to \R$ be in $L^p$ and supported on $A$,
$\Phi : B \to A$ be an $\epsilon$-deformation,
and $F_\Phi = \Phi^{-1}_\sharp F$, i.e.
\begin{align}
F_\Phi(\x) = F(\Phi(\x)) |\det(\nabla \Phi(\x))|
\end{align}
on $B$, and $0$ elsewhere.
Then for any $\u \in \S^{d-1}$,
\begin{align}
\label{eq:projpert}
\|\calP_\u F - \calP_\u F_\Phi\|_{V^p}
\le \min\left\{ \epsilon   \cdot K_{p,\Phi}(F) \cdot C_{p,d}(r), \
               \epsilon^{1/p} \cdot \|F\|_{L^1}\right\},
\end{align}
where
\begin{align}
K_{p,\Phi}(F) = \min\{\|F\|_{L^p}, \|F_\Phi\|_{L^p}\},
\end{align}
and
\begin{align}
C_{p,d}(r) = 2^{(p-1)/p} \, r^{(d-1)(p-1)/p}.
\end{align}

\end{thm}

\begin{rmk}
Theorem \ref{thm:projpert} states that Volterra distances
between a projection of a function and a projection
of its $\epsilon$-deformation
may be bounded by an increasing function of $\epsilon$.
However, when $p > 1$ the bound on the Volterra
distances becomes either constant or
strictly concave for large $\epsilon$,
indicating that the Volterra distances are more robust
to large deformations.
\end{rmk}

Next, we will consider special cases of the deformation $\Phi$,
for which tighter bounds can be shown.

\subsection{Changes in projection angle}

In this section, $F : \R^2 \to \R$ will denote a function in $L^p(\BB_R)$,
where $\BB_R \subset \R^2$ is the disc of radius $R$ and center $(0,0)$.
For a given angle $\theta$, if $\u = (\cos(\theta),\sin(\theta))$,
we let $f_\theta = \calP_{\u} F$; that is,
\begin{align}
f_\theta(x)
= \int_{\R} F(\cos(\theta)x + \sin(\theta)y,
        \cos(\theta)y - \sin(\theta)x) dy
\end{align}
We then have the following result:
\begin{thm}
\label{thm:rotations2D}
Let $\theta$ and $\varphi$ be real numbers,
with $\delta \equiv |\theta - \varphi| < \pi$.
Then for all $1 \le p  \le \infty$,
\begin{align}
\label{eq:anglebound}
\|f_\theta - f_\varphi\|_{V^p}
&\le (2 \sin(\delta/2))^{1/p} \cdot
        \min\left\{ \delta^{1-1/p} \cdot \|F\|_{L^p} \cdot R^{2-1/p}, \
                   \|F\|_{L^1} \cdot R^{1/p} \right\}.
\end{align}
\end{thm}

The proof of Theorem \ref{thm:rotations2D} may be found in
Section \ref{proof:rotations2D}.

\begin{rmk}
When $p=1$, then scaling the problem
so that $R=1$ and $\|F\|_{L^1} = 1$, Theorem \ref{thm:rotations2D} states
\begin{align}
\|f_\theta - f_\varphi\|_{V^p}
&\le 2 \sin(\delta/2),
\end{align}
which matches the bound for Wasserstein distances from \cite{rao2020wasserstein}.
For all values of $p$, under the same scaling, the first term of the right side of
\eqref{eq:anglebound} yields the upper bound
\begin{align}
\|f_\theta - f_\varphi\|_{V^p} \le \delta \cdot \|F\|_{L^p},
\end{align}
since $2 \sin(|\theta - \varphi|/2) \le \delta$.

\end{rmk}

\subsection{Monotonically increasing deformations}

The bound in Theorem \ref{thm:projpert} can
be sharpened by a constant factor depending on $p$ for univariate functions
in the case where the deformation $\Phi$
is monotonically increasing.

\begin{thm}
\label{thm:deformation}
Suppose $f : \R \to \R$ is in $L^p(I)$, where $I$ is a closed interval.
Let $\Phi: J \to I$ be an $\epsilon$-deformation with $\Phi'(x) > 0$
for all $x$. Let
$f_{\Phi}(x) = (\Phi^{-1}_\sharp f)(x) = f(\Phi(x)) \Phi'(x)$ on $J$,
and $0$ elsewhere. Then
\begin{align}
\label{eq:bound1D}
\|f - f_{\Phi}\|_{V^p}
    \le \min\left\{\epsilon \cdot K_{p,\Phi}(f) ,\,
                   \epsilon^{1/p}\cdot \|f\|_{L^1} \right\},
\end{align}
where
\begin{align}
K_{p,\Phi}(f) = \min\{\|f\|_{L^p}, \|f_\Phi\|_{L^p}\}.
\end{align}
\end{thm}

\begin{rmk}
Note that the factor of $2^{(p-1)/p}$ from
Theorem \ref{thm:projpert} is not present in Theorem \ref{thm:deformation},
due to the fact that $\Phi$ is increasing.
To see that the monotonocity of $\Phi$ is required
for this sharper bound, consider the following example.
Fix $\eta > \delta > 0$, and let $f$ be defined by
\begin{align}
f(x)=
\begin{cases}
1 , &\text{ if } -\eta \le x < 0, \\
-1 , &\text{ if } 0 \le x \le \eta, \\
0 , &\text{ otherwise}.
\end{cases}
\end{align}
Let $\Phi : [-\delta,\delta] \to [-\eta,\eta]$
be defined by $\Phi(x) = -(\eta/\delta) x$.
Then
\begin{align}
f_\Phi(x)=
\begin{cases}
-\eta/\delta , &\text{ if } -\delta \le x < 0, \\
\eta/\delta , &\text{ if } 0 \le x \le \delta, \\
0 , &\text{ otherwise}.
\end{cases}
\end{align}
Then it is straightforward to verify that $\epsilon = \eta + \delta$,
$K_{p,\Phi}(f) = 1$, and
$\|f - f_\Phi\|_{V^\infty} = 2\eta$.
Hence, the  right side of \eqref{eq:bound1D},
with $p=\infty$, is $\epsilon = \eta + \delta$,
which is not bigger than
$\|f - f_\Phi\|_{V^\infty} = 2\eta$.
\end{rmk}

\subsection{Applications to sliced Volterra distances}

The bounds from Theorems \ref{thm:projpert} and \ref{thm:rotations2D}
immediately yield bounds for the sliced Volterra
distances.

\begin{thm}
\label{thm:sliced_deformations}

Let $A$ and $B$ be non-empty, bounded, open sets in $\R^d$,
with $r = \diam(A \cup B)$.
Let $F: \R^d \to \R$ be in $L^p(A)$,
$\Phi : B \to A$ be an $\epsilon$-deformation,
and $F_\Phi = \Phi^{-1}_\sharp F$, i.e.
\begin{align}
F_\Phi(\x) = F(\Phi(\x)) |\det(\nabla \Phi(\x))|
\end{align}
on $B$, and $0$ elsewhere.
Let $1 \le p \le \infty$ and
fix any probability distribution $\eta$ over $\S^{d-1}$.
Then
\begin{align}
\meanvolt_{p,\eta}(F,F_\Phi) \le
\min\left\{ \epsilon   \cdot K_{p,\Phi}(F) \cdot C_{p,d}(r), \
               \epsilon^{1/p} \cdot \|F\|_{L^1} \right\},
\end{align}
where
\begin{align}
K_{p,\Phi}(F) = \min\{\|F\|_{L^p}, \|F_\Phi\|_{L^p}\},
\end{align}
and
\begin{align}
C_{p,d}(r) = 2^{(p-1)/p} \, r^{(d-1)(p-1)/p}.
\end{align}

\end{thm}

\begin{thm}
\label{thm:sliced_rotations2D}
Let $F: \R^2 \to \R$ be in $L^p(\BB_R)$.
Suppose $0 \le \delta < \pi$, and define $F_\delta$ by
\begin{align}
F_\delta(x,y) = F(x \cos(\delta) + y \sin(\delta), y \cos(\delta) - x \sin(\delta)).
\end{align}
Then for all $1 \le p \le \infty$ and
any probability distribution $\eta$ over $\S^{d-1}$,
\begin{align}
\meanvolt_{p,\eta}(F,F_\delta)
&\le (2 \sin(\delta/2))^{1/p} \cdot
        \min\left\{ \delta^{1-1/p} \cdot \|F\|_{L^p} \cdot R^{2-1/p}, \
                   \|F\|_{L^1}  \cdot R^{1/p} \right\}.
\end{align}
\end{thm}

\section{Asymptotic behavior of the discrete norms}
\label{section:discrete}

In this section, we consider the behavior of
the discrete Volterra norms, defined in Section \ref{sec:trapezoidal},
for vectors consisting of samples of a function $f$ on $[a,b]$
from an equispaced grid.
It will be convenient to introduce some notation.
Let $n$ be a positive integer, and define 
$a_0 < a_1 < \dots < a_n$ as in \eqref{eq:subinterval_ends},
namely
\begin{align}
a_k = a + \frac{k}{n}(b-a), \quad 0 \le k \le n.
\end{align}
Note that $a_0 = a$ and $a_n = b$. Let $\bfend$ be the vector
in $\R^{n+1}$ with entries $\bfend[k] = f(a_k)$, for $0 \le k \le n$.

Denote the mean of $f$ on $[a,b]$ by
\begin{align}
\mu(f) = \frac{1}{b-a} \int_{a}^{b} f(t) dt,
\end{align}
and let $\fcen(x) = f(x) - \mu(f)$.

If $\w$ is a vector in $\R^{n+1}$, denote its trapezoidal mean by
\begin{align}
\mean(\w)
= \frac{1}{2n} \sum_{k=0}^{n-1} (\w[k] + \w[k+1]),
\end{align}
and let $\wcen$ in $\R^{n+1}$ have entries $\wcen[k] = \w[k] - \mean(\w)$.

\subsection{Convergence rates for well-behaved functions}
\label{sec:sample_limit}

We first prove rates on the convergence of the discrete approximation
$\|\bf\|_{\nu_p}$ to the Volterra norm $\|f\|_{V^p}$,
where $f$ is a reasonably well-behaved function.
Theorem \ref{thm:convergence_lipschitz} establishes a covergence rate
of $O(1/n)$ for piecewise Lipschitz functions, whereas
Theorem \ref{thm:convergence_smooth} establishes the faster rate
of $O(1/n^2)$ for smoother functions.

\begin{thm}
\label{thm:convergence_lipschitz}
Suppose $a=c_0 < c_1 < \dots < c_r = b$, and
$f$ has Lipschitz constant bounded by $L > 0$ on each interval $(c_j,c_{j+1})$,
$0 \le j \le r-1$.
Then for all $1 \le p \le \infty$,
\begin{align}
\label{eq:discrete_lipschitz}
\left| \| \bf \|_{\nu_p} -  \|f\|_{V^p} \right|
\le C \frac{(b-a)^{1+1/p}}{n} \left( L (b-a) + r\|f\|_{L^\infty}\right),
\end{align}
where $C>0$ is a universal constant.
The same bound holds by replacing $f$ with $\fcen$
and $\bf$ with $\bfcen$.
\end{thm}

In other words, for piecewise Lipschitz functions,
the discrete Volterra norm based on $n$ subintervals
converges to the true Volterra
norm at a rate of $O(1/n)$.

The proof of Theorem \ref{thm:convergence_lipschitz}
may be found in Section \ref{proof:convergence_lipschitz}.
If instead of being merely piecewise Lipschitz,
the function $f$ is $C^2$ and not too oscillatory, then
the discrete Volterra norms give a higher order approximation
to the Volterra norms of $f$:

\begin{thm}
\label{thm:convergence_smooth}
Suppose $f$ is a two times continuously differentiable function on $[a,b]$
with only finitely many zero crossings.
Let $1 \le p \le \infty$.
Then for all $n$ sufficiently large,
\begin{align}
\left| \|\bf\|_{\nu_p} - \|f\|_{V^p}\right| \le C \frac{K(f,p,a,b)}{n^2},
\end{align}
where
\begin{align}
K(f,p,a,b)=
(b-a)^{3+1/p}\|f''\|_{L^\infty} +  (b-a)^{2+1/p} \|f'\|_{L^\infty}
            + (b-a)^2 |f(b)|\left(\frac{|\mu(f)|}{\|f\|_{V^p}} \right)^{p-1}
\end{align}
when $1 \le p < \infty$, and
\begin{align}
K(f,\infty,a,b) = (b-a)^3\|f''\|_{L^\infty} + (b-a)^2\|f'\|_{L^\infty},
\end{align}
and where $C > 0$ is a universal constant.

An analogous bound holds by replacing $f$ with $\fcen$
and $\bf$ with $\bfcen$:
\begin{align}
\left| \|\bfcen\|_{\nu_p} - \|\fcen\|_{V^p}\right|
    \le C \frac{(b-a)^{3+1/p}\|f''\|_{L^\infty} +  (b-a)^{2+1/p} \|f'\|_{L^\infty}}{n^2},
\end{align}
for all $1 \le p \le \infty$.

\end{thm}

In other words, for such functions $f$,
the discrete Volterra norm based on $n$ subintervals
converges to the true Volterra
norm at a rate of $O(1/n^2)$.
The proof of Theorem \ref{thm:convergence_smooth}
may be found in Section \ref{proof:convergence_smooth}.

\subsection{Robustness to Gaussian noise}
\label{sec:noise}

Next, we show that the discrete Volterra metrics are robust to additive noise.
More precisely, as the number $n$ of subintervals on which samples are taken grows,
the effects of additive Gaussian noise on the samples of $f$
vanish at a predictable rate.

\begin{thm}
\label{thm:noise}
Let $\sigma_0,\sigma_1,\dots,\sigma_n,\dots$
be a sequence of positive numbers, and let
$Z = (Z[0],\dots,Z[n])$, where $Z[0],Z[1],\dots,Z[n],\dots$
are independent with $Z[j] \sim N(0,\sigma_j^2)$.
Suppose too that $\sigma > 0$ satisfies
\begin{align}
\frac{1}{n} \sum_{j=1}^{n} \sigma_j^2
\le \sigma^2,
\end{align}
for all $n$.
Let  $t > 0$. Then for all $1 \le p \le \infty$,
\begin{align}
\label{eq:concentration_bound}
\Prob\left\{ \| Z \|_{\nu_p} \ge t \right\}
\le A e^{-B t^2 n / \sigma^2},
\end{align}
where $A > 0$ and $B > 0$ are universal constants;
\begin{align}
\label{eq:as_limit}
\lim_{n \to \infty} \| Z \|_{\nu_p} = 0
\end{align}
almost surely; and
\begin{align}
\label{eq:expectation}
\EE \| Z \|_{\nu_p}
\le C \frac{\sigma}{\sqrt{n}},
\end{align}
where $C > 0$ is a universal constant.
Furthermore, \eqref{eq:concentration_bound}, \eqref{eq:as_limit}
and \eqref{eq:expectation} hold with $Z$ replaced by $Z_\cen$.
\end{thm}

\begin{cor}
\label{cor:convergence}
Suppose $f$ satisfies the conditions of Theorem \ref{thm:convergence_lipschitz}, $Z$ satisfies the conditions of Theorem \ref{thm:noise}, and $Y = \bf + Z$.
Let  $t > 0$ and $1 \le p \le \infty$. Then for all $n$ sufficiently large,
\begin{align}
\label{eq:concentration2}
\Prob\left\{ \left| \| Y \|_{\nu_p} - \|f\|_{V^p}\right| \ge t \right\} \le A e^{-B t^2 n / \sigma^2},
\end{align}
where $A > 0$ and $B > 0$ are universal constants;
\begin{align}
\label{eq:as_limit2}
\lim_{n \to \infty} \| Y \|_{\nu_p} = \|f\|_{V^p}
\end{align}
almost surely; and
\begin{align}
\label{eq:expectation2}
\EE \left|\| Y \|_{\nu_p} - \| f \|_{V^p} \right| \le C \frac{\sigma}{\sqrt{n}},
\end{align}
where $C > 0$ is a universal constant.
Furthermore, \eqref{eq:concentration2}, \eqref{eq:as_limit2} and \eqref{eq:expectation2}
hold with $f$ replaced by $\fcen$ and $Y$ replaced by $Y_\cen$.
\end{cor}

The  proofs of Theorem \ref{thm:noise} and Corollary \ref{cor:convergence} are provided in Section \ref{proof:noise}.

\begin{rmk}
In the setting of Corollary \ref{cor:convergence},
both the signal vector $\bf$ and the noise vector $Z$ have comparable $p$-norms;
consequently, $\|Y\|_{\ell_p}$ does not approach $\|f\|_{L^p}$ as $n \to \infty$.
For example, if $\sigma_j = \sigma$ for all $j$, then almost surely
\begin{align}
\lim_{n \to \infty} \|Y\|_{\ell_2}^2 = \|f\|_{L^2}^2 + \sigma^2.
\end{align}
By contrast, \eqref{eq:as_limit2} states that $Z$ has a negligible effect
on the Volterra norm when $n$ is large.
\end{rmk}

\section{Numerical results}
\label{section:numerical}

In this section, we show the results of numerical experiments
comparing the Volterra distances to Wasserstein
and Lebesgue distances.

Define the approximate Wasserstein distance as follows.
Suppose $f$ is a probability density on $[a,b]$,
and let $\bf = (f(a_0),\dots,f(a_{n}))$,
where 
\begin{align}
a_k = a + \frac{k}{n}(b-a), \quad 0 \le k \le n.
\end{align}
We define, for $0 < t < 1$, the approximate
inverse CDF by
\begin{align}
\what{(\V f)^{-1}}(t) = \min\{a_k : (\Vtrap \bf)[k] \ge t\}.
\end{align}
For integer $m$, we define the midpoint
grid values $t_j = (j - 1/2)/m$, $1 \le j \le m$.
Let $\what{(\Vtrap \bf)^{-1}}[j] = \what{(\V f)^{-1}}(t_j)$.
Then the approximate $p$-Wasserstein distance between densities $f$ and $g$
is defined by
\begin{align}
\what\wass_p(f,g)
= \|\what{(\Vtrap \bf)^{-1}} - \what{(\Vtrap \bg)^{-1}}\|_{\ell_p}s.
\end{align}
In all the experiments reported below, we set $m = n$.

\subsection{Distances under translation}
\label{sec:transl}

We illustrate Theorem \ref{thm:deformation} on the functions
shown in Figure \ref{fig:transl_bumps};
these are translations $f_\epsilon(x) = f(x-\epsilon)$
of the function $f$ on $[0,1]$ defined by
\begin{align}
\label{eq:transl_bump}
f(x) = C  \cos(10x - 1) e^{-16 (10x - 3/2)^2},
\end{align}
where $C$ is chosen so that the integral of $f$ is $1$.
Figure \ref{fig:transl_all} shows the estimated Volterra $p$-distances (top row),
$p$-Wasserstein (middle row), and Lebesgue $p$-distances (bottom row)
based on samples from $n = 500$ subintervals,
plotted as functions of the shift size $\epsilon$.

The Volterra distances exhibit the behavior described by
the bound in Theorem \ref{thm:deformation},
namely, the distances grow as concave functions of the shift size.
When $p=1$ and $p=2$, the distances continue to grow with the translation,
whereas when $p=\infty$ the distances level off,
consistent with the upper bound from Theorem \ref{thm:deformation}.
The Wasserstein distances all increase linearly with the shift size
(which can be seen easily from the Monge formulation of the Wasserstein distance).
By contrast with these behaviors, all of the Lebesgue distances quickly saturate to a constant value, independent of $\epsilon$, as soon as the translation is big enough so that the numerical supports of the translated functions do not overlap.

\begin{figure}
\center
\includegraphics[scale=.5]{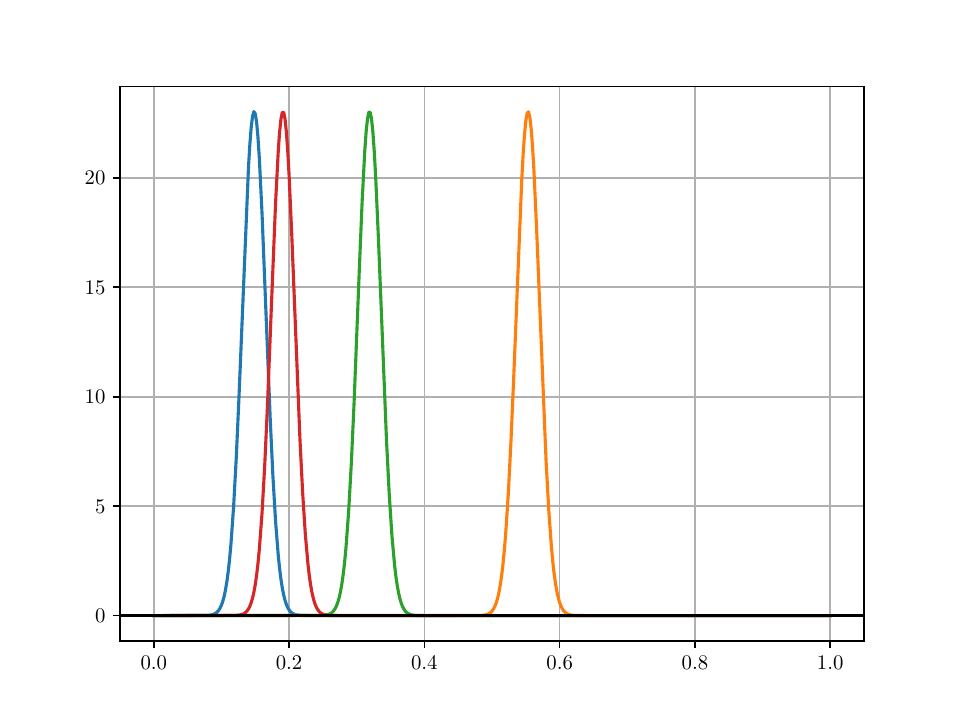}
\caption{The function \eqref{eq:transl_bump} (far left, in blue) and its translations,
used in the experiment from Section \ref{sec:transl}.}
\label{fig:transl_bumps}
\end{figure}

\begin{figure}
\center
\includegraphics[scale=.45]{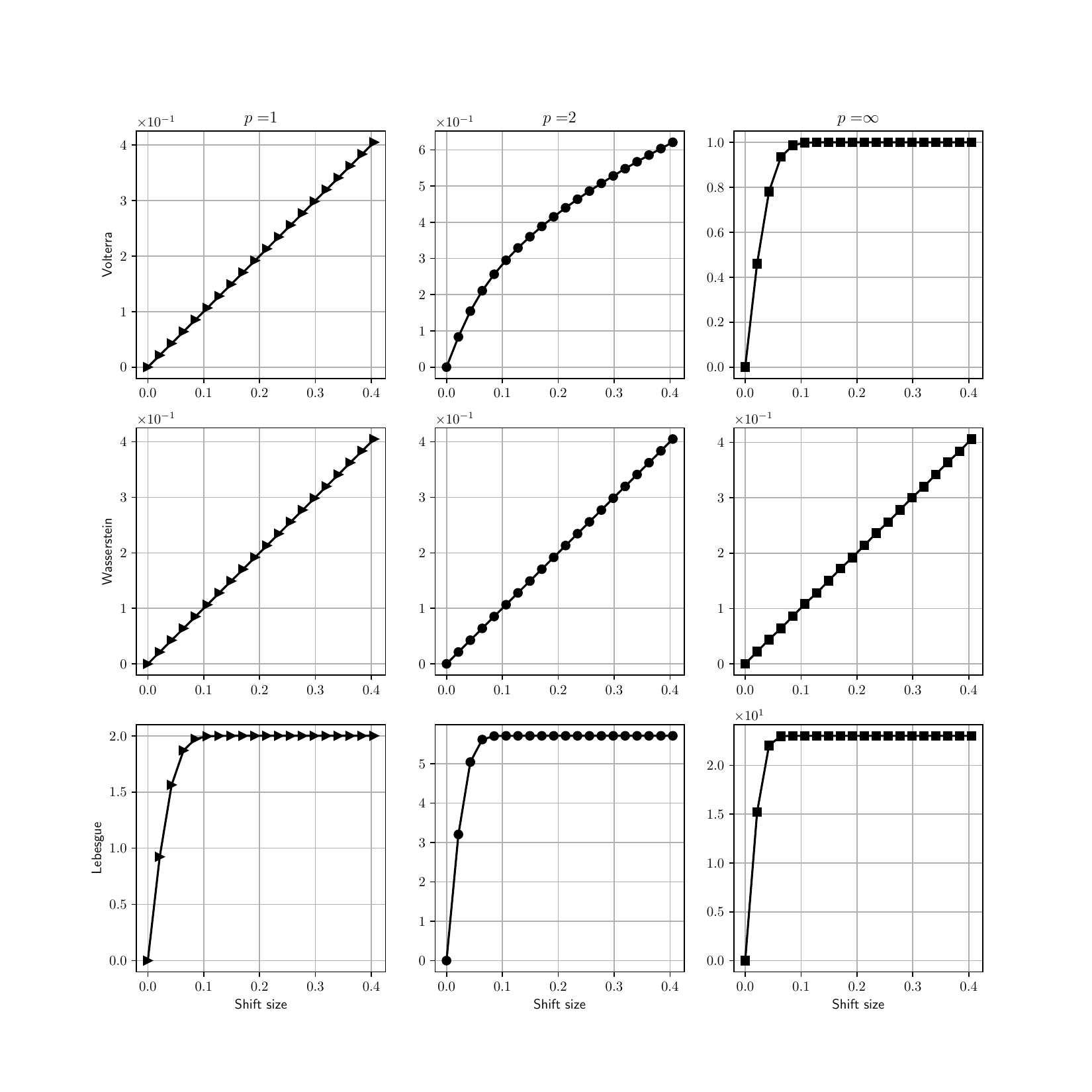}
\caption{The first row shows the approximated Volterra distances
(based on $n=500$ subintervals) between the function \eqref{eq:transl_bump} 
and its shifts, as a function of the shift size.
The second row shows the approximated Wasserstein distances,
and the third row shows the approximated Lebesgue distances.
The values of $p$ (from left to right) are $p=1,2, \infty$.
See Section \ref{sec:transl} for details.
}
\label{fig:transl_all}
\end{figure}

\subsection{Distances under dilation}
\label{sec:dilations}

Next, we consider the function $f$ defined on $[0,1]$ by
\begin{align}
\label{eq:dilation_func}
f(x) = C  x^6(1-x)
\end{align}
where $C$ is chosen so that the integral of $f$ is $1$.
We consider the family of dilations of $f$ parameterized by $\eta \ge 1$;
these are the functions $f_\eta$ defined by
$f_\eta(x) = f(\eta x) \eta$ on $[0,1/\eta]$, and
$f_\eta(x) = 0$ elsewhere.
The size $\epsilon$ of the dilation is 
\begin{align}
\epsilon = 1  - \frac{1}{\eta}.
\end{align}
Figure \ref{fig:dilates_plots} shows the function $f$ and some of its dilates.
Figure \ref{fig:dilates_plots} shows the estimated Volterra $p$-distances (top row),
$p$-Wasserstein (middle row), and Lebesgue $p$-distances (bottom row)
based on samples from $n = 500$ subintervals, plotted as functions of $\epsilon$.

The Volterra distances exhibit the behavior described by
the bound in Theorem \ref{thm:deformation},
namely, the distances grow as concave functions
of the deformation size $\epsilon$.
When $p=1$ and $p=2$,
the distances continue to grow as $\epsilon$ grows,
whereas when $p=\infty$ the distances level off,
consistent with the upper bound from Theorem \ref{thm:deformation}.
The Wasserstein distances all increase linearly with the deformation size
(which can be seen easily from the Monge formulation of the Wasserstein distance).
Because the transformation preserves the integral of $f$,
the $L^1$ distance levels off when the dilation size is big,
since the supports of the function and its dilate are almost disjoint.
By contrast, the $L^2$ and $L^\infty$ distances grow rapidly for large dilation sizes.
This is because $\|f_\eta\|_{L^p}$ diverges as $\eta$ grows,
and hence these distances reflect the size of the individual functions
and not the relationship between the functions.

\begin{figure}
\center
\includegraphics[scale=.5]{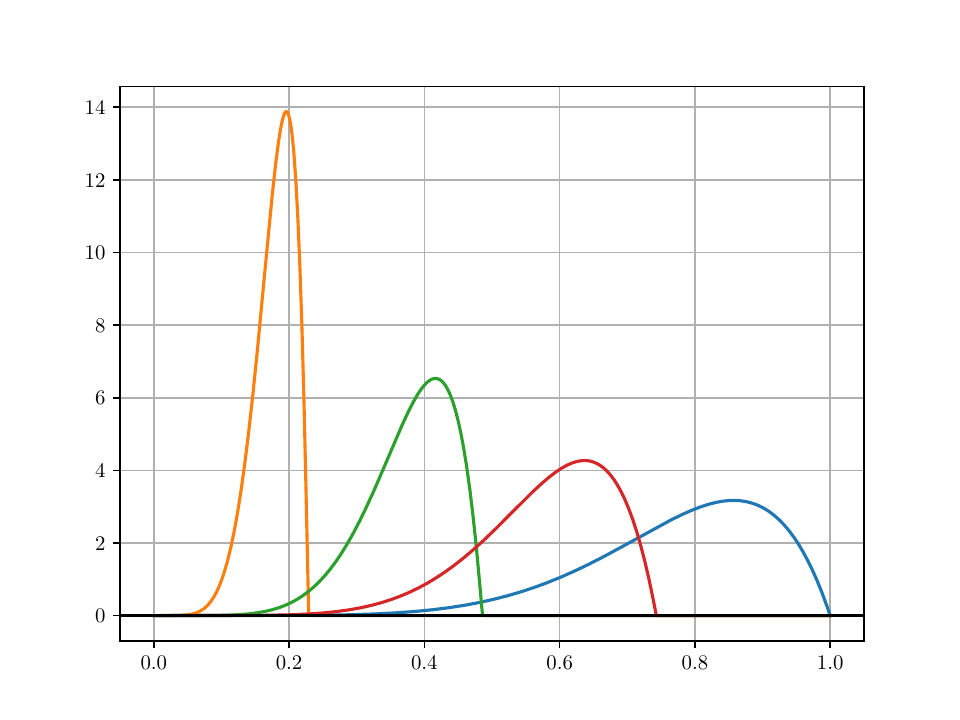}
\caption{The function \eqref{eq:dilation_func} (in blue) and its
dilations, used in the experiment from Section \ref{sec:dilations}.
}
\label{fig:dilates_plots}
\end{figure}

\begin{figure}
\center
\includegraphics[scale=.4]{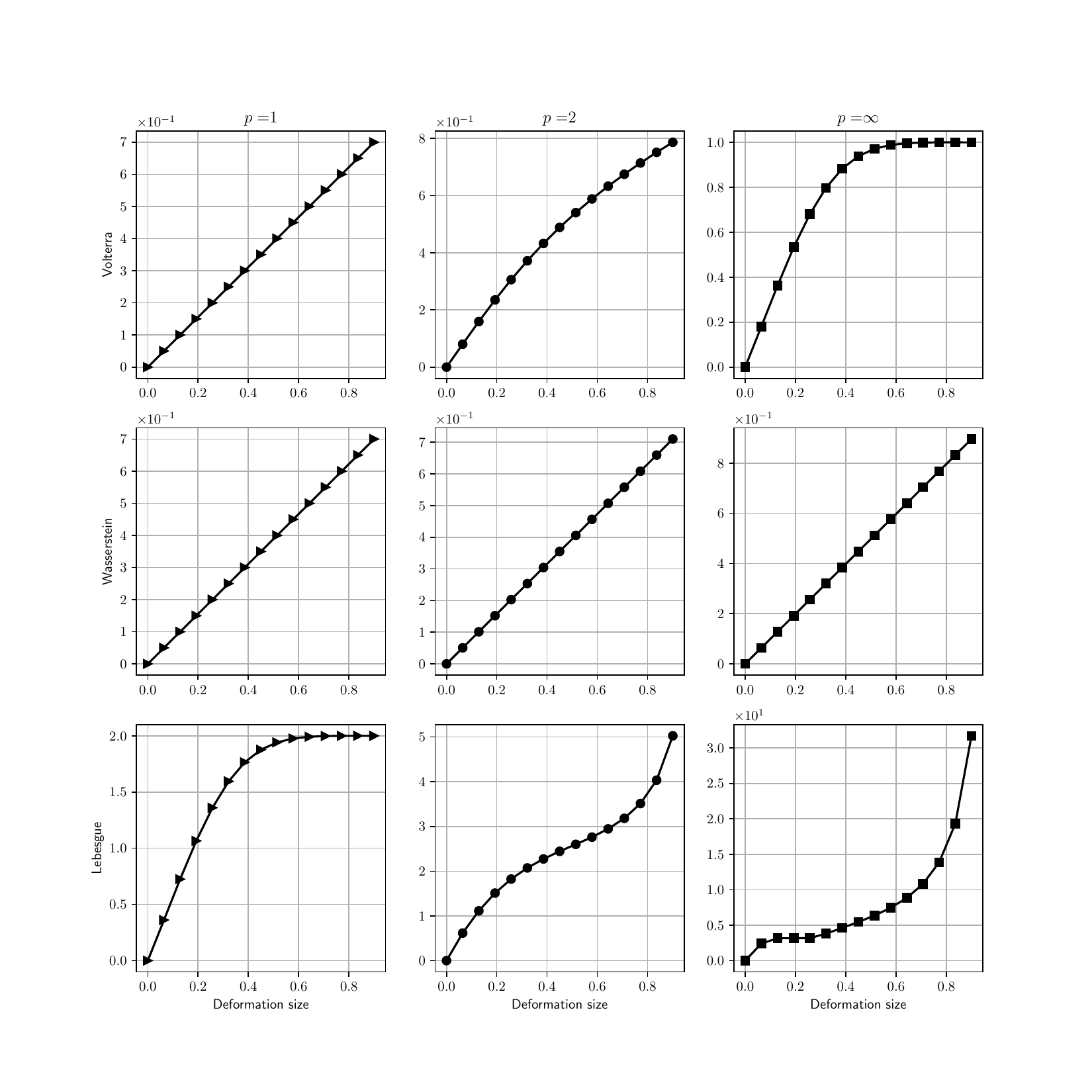}
\caption{The first row shows the approximated Volterra distances
(based on $n=500$ subintervals) between the function \eqref{eq:dilation_func} 
and  its dilates, as a function of the deformation size.
The second row shows the approximated Wasserstein distances,
and the third row shows the approximated Lebesgue distances.
The values of $p$ (from left to right) are $p=1,2, \infty$.
See Section \ref{sec:dilations} for details.
}
\label{fig:dilates_all}
\end{figure}

\subsection{Distances under powers}
\label{sec:powers}

Next, we consider the function $f$ defined on $[0,1]$ by
\begin{align}
\label{eq:powers_func}
f(x) = C  x(1-x)^4,
\end{align}
where $C$ is chosen so that the integral of $f$ is $1$.
We consider the family of deformations $\Phi(x) = x^{\alpha}$,
where $\alpha \ge 1$; the corresponding transformation of $f$
is the function $f_\alpha$ defined by
$f_\alpha(x) = f(x^\alpha) \alpha x^{\alpha-1}$ on $[0,1]$, and
$f_\alpha(x) = 0$ elsewhere.
The deformation size $\epsilon$ is given by
\begin{align}
\epsilon = \left(\frac{1}{\alpha} \right)^{\frac{1}{\alpha-1}}.
\end{align}
Figure \ref{fig:powers_plots} shows the function $f$ and its deformations.
Figure \ref{fig:powers_plots} shows the estimated Volterra $p$-distances (top row),
$p$-Wasserstein (middle row), and Lebesgue $p$-distances (bottom row)
based on samples from $n = 500$ subintervals, plotted as functions of $\epsilon$.

The Volterra distances exhibit the behavior described by
the bound in Theorem \ref{thm:deformation}, namely,
the distances grow as concave functions of the deformation size $\epsilon$.
When $p=1$ and $p=2$, the distances continue to grow as $\epsilon$ grows,
whereas when $p=\infty$ the distances level off,
consistent with the upper bound from Theorem \ref{thm:deformation}.
As in the case of dilations,
because the transformation preserves the integral of $f$,
the $L^1$ distance levels off when the deformation size is big,
since the supports of the function and its deformation are almost disjoint.
On the other hand, the $L^2$ and $L^\infty$ distances grow rapidly for
large powers.

\begin{figure}
\center
\includegraphics[scale=.5]{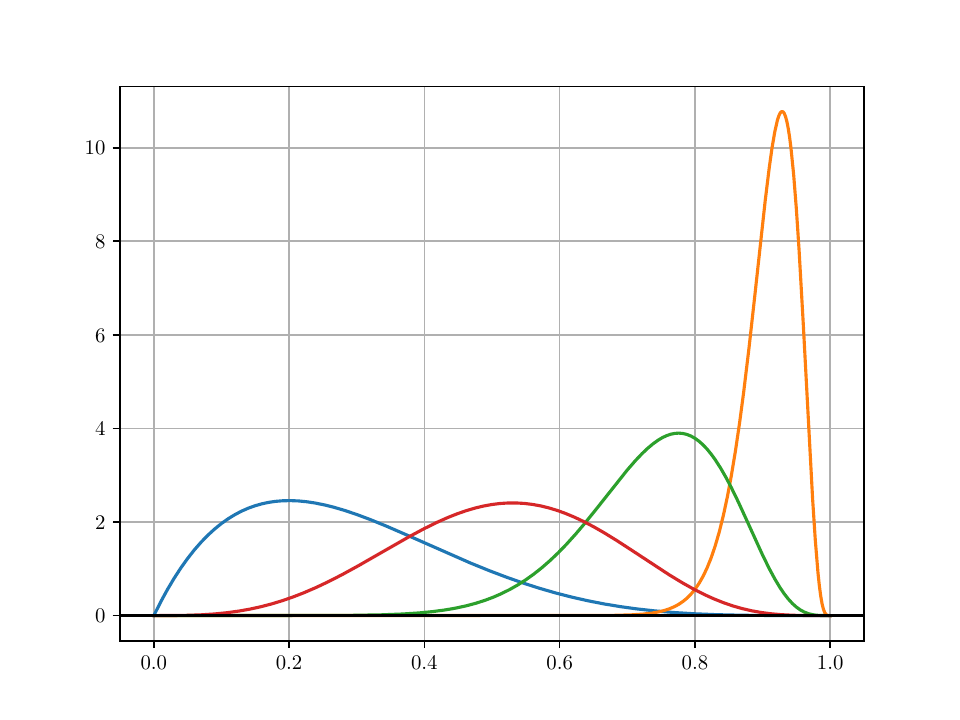}
\caption{The function \eqref{eq:powers_func} (in blue) and its
deformations, used in the experiment from Section \ref{sec:powers}.
}
\label{fig:powers_plots}
\end{figure}

\begin{figure}
\center
\includegraphics[scale=.4]{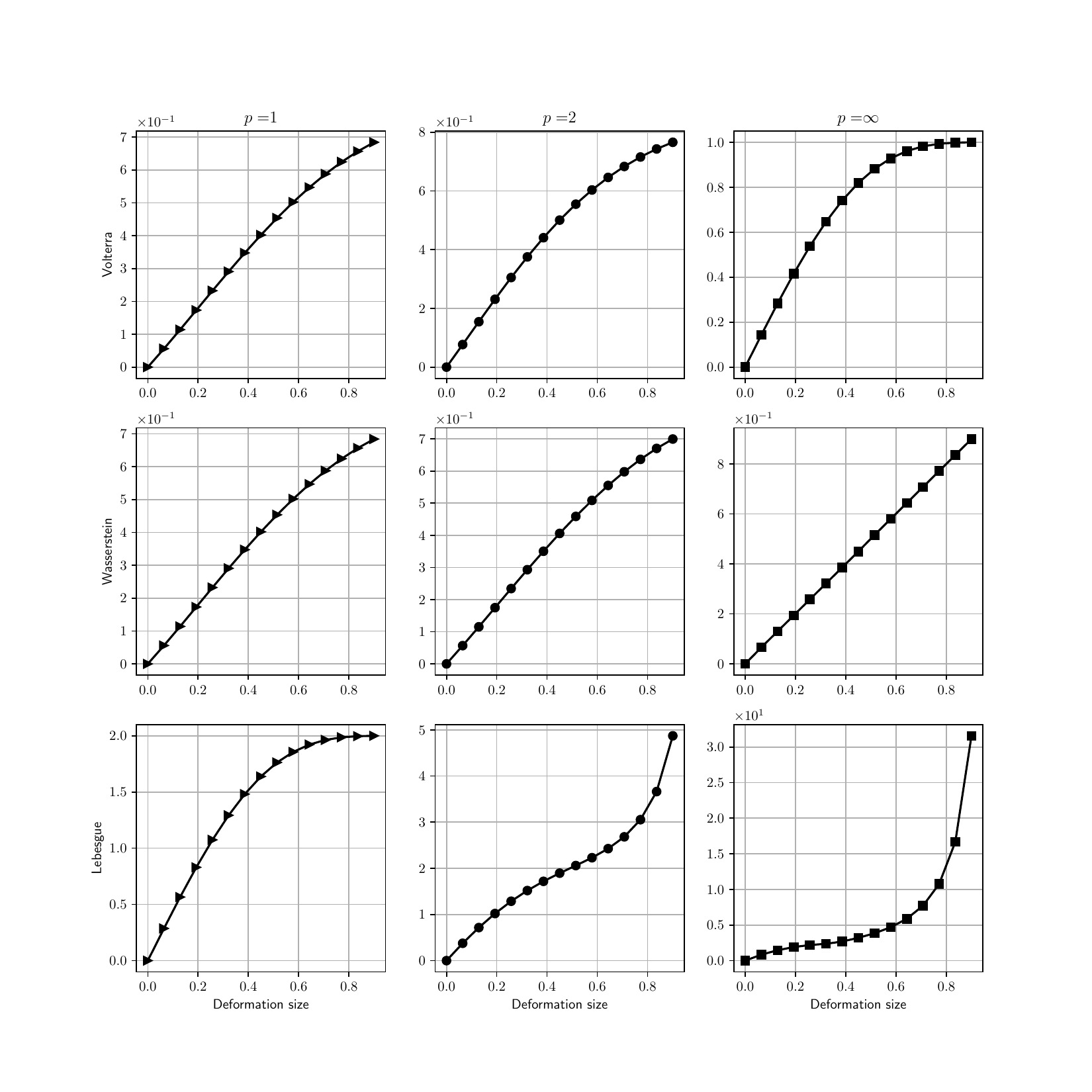}
\caption{The first row shows the approximated Volterra distances
(based on $n=500$ subintervals) between the function \eqref{eq:powers_func} 
and its deformations, as a function of the deformation size.
The second row shows the approximated Wasserstein distances,
and the third row shows the approximated Lebesgue distances.
The values of $p$ (from left to right) are $p=1,2, \infty$.
See Section \ref{sec:powers} for details.
}
\label{fig:powers_all}
\end{figure}

\subsection{Distances between rotated projections I}
\label{sec:rotline}

We illustrate the behavior described by Theorem
\ref{thm:rotations2D} on the function $F$ displayed in Figure \ref{fig:rotline_plots},
given by the formula
\begin{align}
F(\x) = \frac{1}{2 \pi \sigma} \sum_{k=0}^{6} h_k e^{-|\x - \c_k|^2/\sigma},
\end{align}
where $\sigma=1/5000$, $\c_k = (x_k,x_k^2)$ and $x_k = -1/3 + k/9$,
and $h_k = (k+1) / 24$, $0 \le k \le 6$.

We denote by $f$ the projection of $F$ onto the $x$-axis, and $f_\theta$
the projection of $F$ after rotation by $\theta$ radians.
Figure \ref{fig:rotline_plots} shows a heatmap of $F$ and a rotation,
along with their corresponding projections.
Figure \ref{fig:rotline_plots} shows the estimated Volterra $p$-distances (top row),
$p$-Wasserstein (middle row), and Lebesgue $p$-distances (bottom row)
based on samples from $n = 500$ subintervals, plotted as functions of the
rotation angle. When $p=2$, the Volterra and Wasserstein distances
have very similar behavior, and both vary smoothly with respect to the rotation angle
(when $p=1$, Volterra and Wasserstein are the same, as always).
The Lebesgue distances for all $p$, by contrast, are more irregular functions
of the rotation angle.

\begin{figure}
\center
\includegraphics[scale=.3]{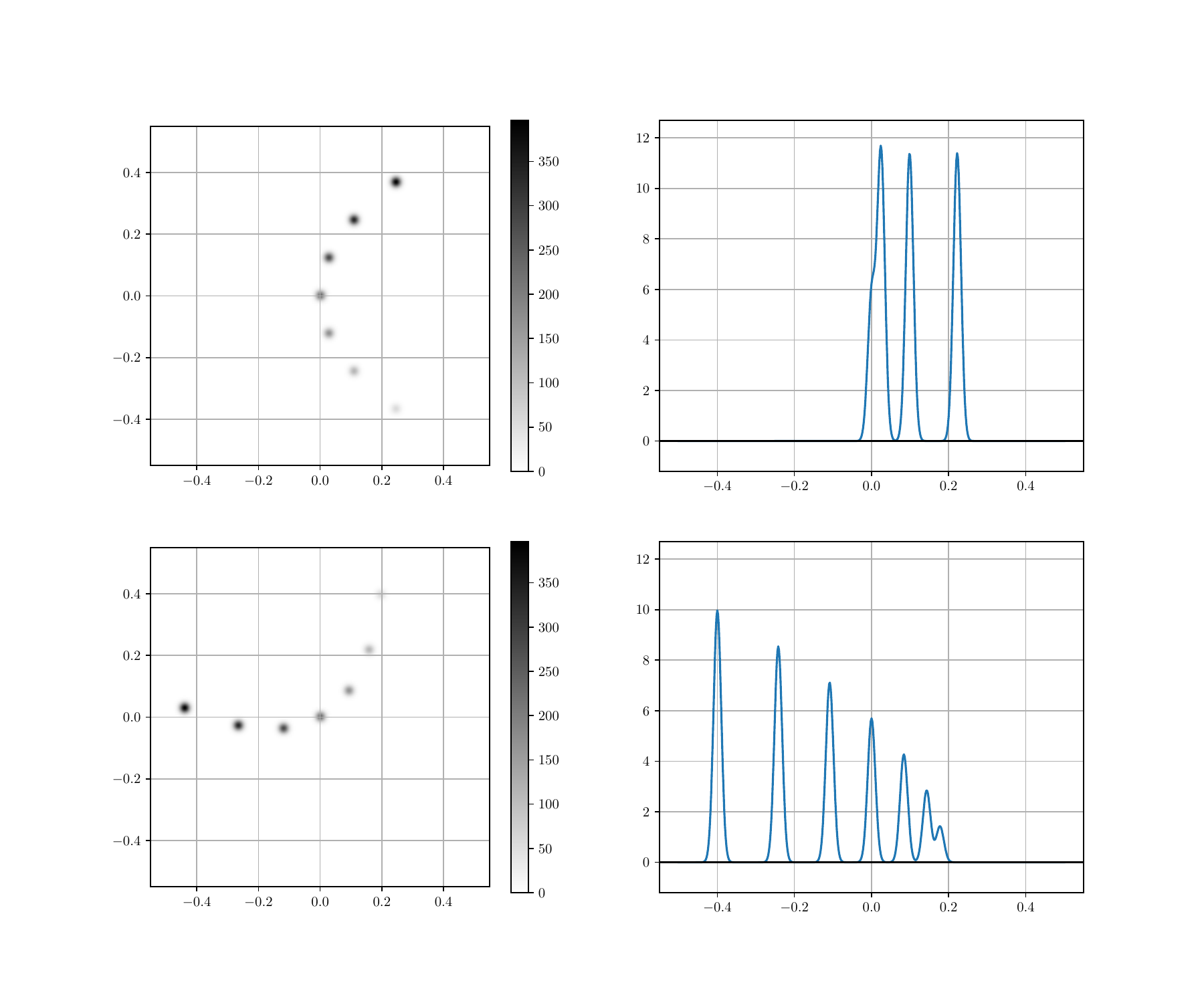}
\caption{The function $F$ from Section \ref{sec:rotline} (top left)
and a rotation  (bottom left),
with their respective projections onto the $x$-axis.
}
\label{fig:rotline_plots}
\end{figure}

\begin{figure}
\center
\includegraphics[scale=.4]{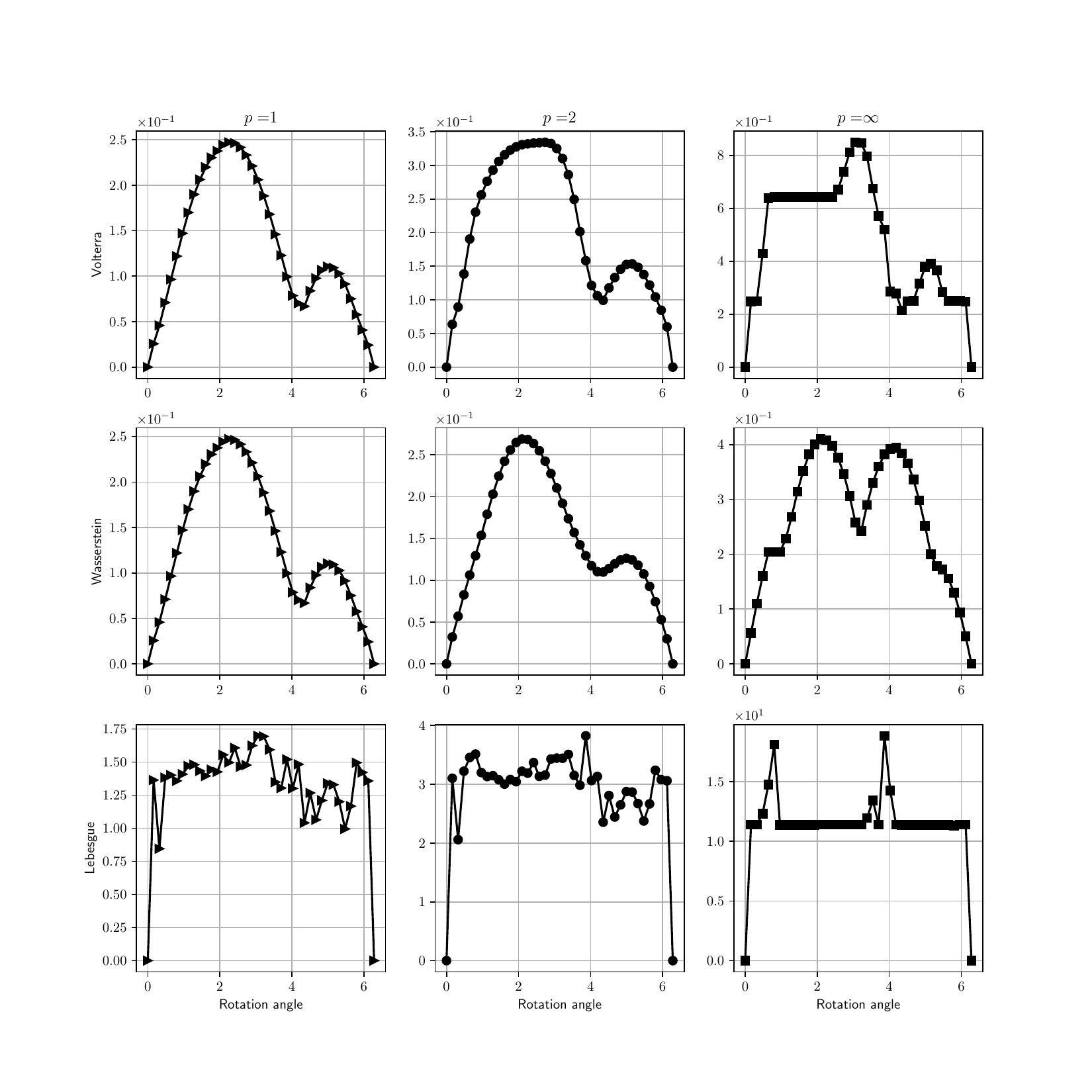}
\caption{The first row shows the approximated Volterra distances
(based on $n=500$ subintervals) between the projections of
$F$ and its rotations, as a function of the rotation angle.
The second row shows the approximated Wasserstein distances,
and the third row shows the approximated Lebesgue distances.
The values of $p$ (from left to right) are $p=1,2, \infty$.
See Section \ref{sec:rotline} for details.
}
\label{fig:rotline_all}
\end{figure}

\subsection{Distances between rotated projections II}
\label{sec:rotrings}

We next illustrate the behavior described by Theorem
\ref{thm:rotations2D} on the function $F$ displayed in Figure \ref{fig:rotrings_plots},
which consists of two nested rings of Gaussian bumps,
given by the formula
\begin{align}
F(\x) = \frac{h}{2 \pi \sigma} \sum_{k=0}^{4} e^{-|\x - \c_k|^2/\sigma}
        + \frac{h}{2 \pi \sigma} \sum_{k=0}^{6} e^{-|\x - \d_k|^2/\sigma}
\end{align}
where $\sigma=1/4000$, $h=1/12$, and $\c_k = (\cos(\theta_k),\sin(\theta_k))$
with $\theta_k = 2k\pi/5 + (\sqrt{2} + \sqrt{5} + \sqrt{3}) \pi$
when $0 \le k \le 4$, and $\d_k = (\cos(\varphi_k),\sin(\varphi_k))$
with $\varphi_k = 2k\pi/7 + (\sqrt{2} + \sqrt{5}) \pi$
when $0 \le k \le 6$.

We denote by $f$ the projection of $F$ onto the $x$-axis, and $f_\theta$
the projection of $F$ after rotation by $\theta$ radians.
Figure \ref{fig:rotrings_plots} shows a heatmap of $F$ and a rotation,
along with their corresponding projections.
Figure \ref{fig:rotrings_plots} shows the estimated Volterra $p$-distances (top row),
$p$-Wasserstein (middle row), and Lebesgue $p$-distances (bottom row)
based on samples from $n = 500$ subintervals, plotted as functions of the
rotation angle. As for the example
from Section \ref{sec:rotline}, when $p=2$, the Volterra and Wasserstein distances
have very similar behavior, and both vary smoothly with respect to the rotation angle,
whereas the Lebesgue distances for all $p$, by contrast, are irregular functions
of the rotation angle.

\begin{figure}
\center
\includegraphics[scale=.3]{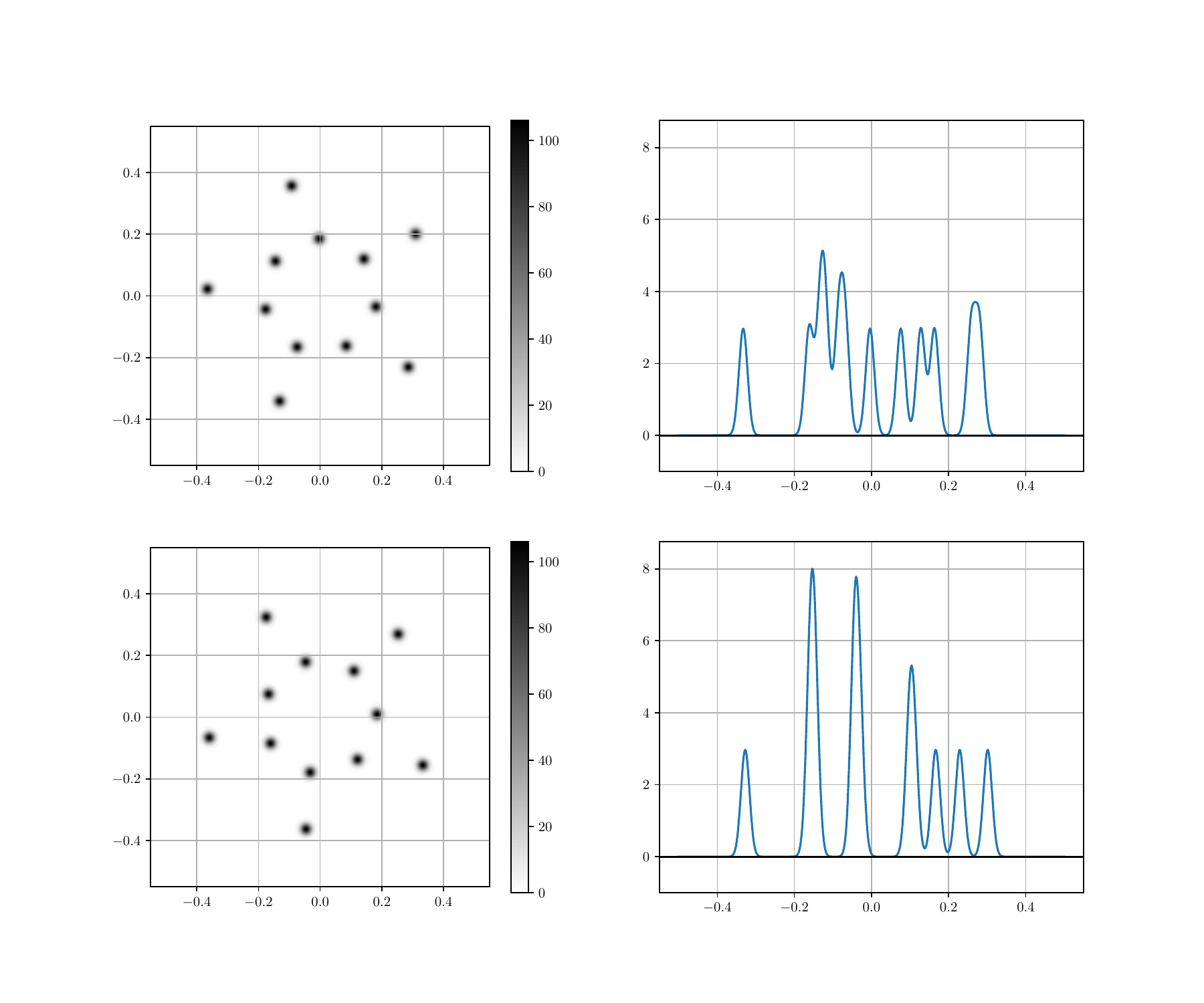}
\caption{The function $F$ from Section \ref{sec:rotrings} (top left) and a rotation (bottom left),
with their respective projections onto the $x$-axis.
}
\label{fig:rotrings_plots}
\end{figure}

\begin{figure}
\center
\includegraphics[scale=.4]{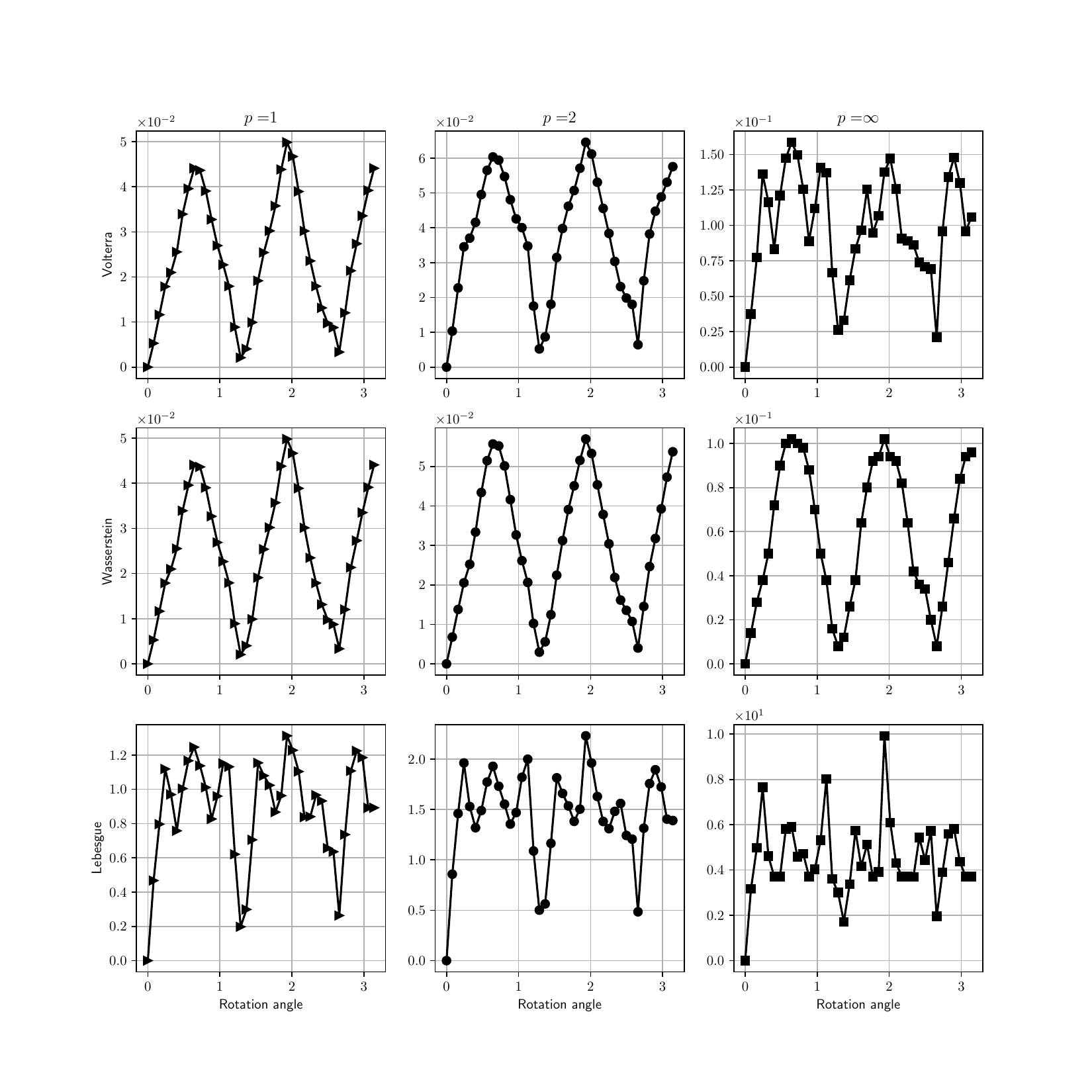}
\caption{The first row shows the approximated Volterra distances
(based on $n=500$ subintervals) between the projections of
$F$ and its rotations, as a function of the rotation angle.
The second row shows the approximated Wasserstein distances,
and the third row shows the approximated Lebesgue distances.
The values of $p$ (from left to right) are $p=1,2, \infty$.
See Section \ref{sec:rotrings} for details.
}
\label{fig:rotrings_all}
\end{figure}

\subsection{Distances under domain shrinking}
\label{sec:shrink}

We next illustrate the behavior of the distances
between projections on the function $F$ displayed in
Figure \ref{fig:shrink_plots},
given by the formula
\begin{align}
F(\x) = \frac{1}{2 \pi \sigma} \sum_{k=0}^{6} h_k e^{-|\x - \c_k|^2/\sigma}
\end{align}
where $\sigma=1/3000$, $h_k=(k+4) / 49$,
and $\c_k = (\cos(\theta_k),\sin(\theta_k))$
with $\theta_k = 2k\pi/7 + (\sqrt{2} + \sqrt{5} + \sqrt{3}) \pi$,
$0 \le k \le 7$.

The function $F$ is transformed by shrinking the center
of each ring towards the origin by an amount $\epsilon$.
Figure \ref{fig:shrink_plots} shows a heatmap of $F$ and a shrunken version,
along with their corresponding projections.
Figure \ref{fig:shrink_plots} shows the estimated Volterra $p$-distances (top row),
$p$-Wasserstein (middle row), and Lebesgue $p$-distances (bottom row)
based on samples from $n = 500$ subintervals, plotted as functions of the
shrinkage parameter (proportional to the distance between the Gaussian centers
of the original function and the shrunken function).
The Wasserstein distances are linear functions of the
distance, since each projected Gaussian of the shrunken function is just a translation
of the projected Gaussian from the original function.
When $p=2$, the Volterra distance is also a smooth but more concave function.
The Lebesgue distances for all $p$, by contrast, are more irregular functions
of the shrinkage parameter.

\begin{figure}
\center
\includegraphics[scale=.3]{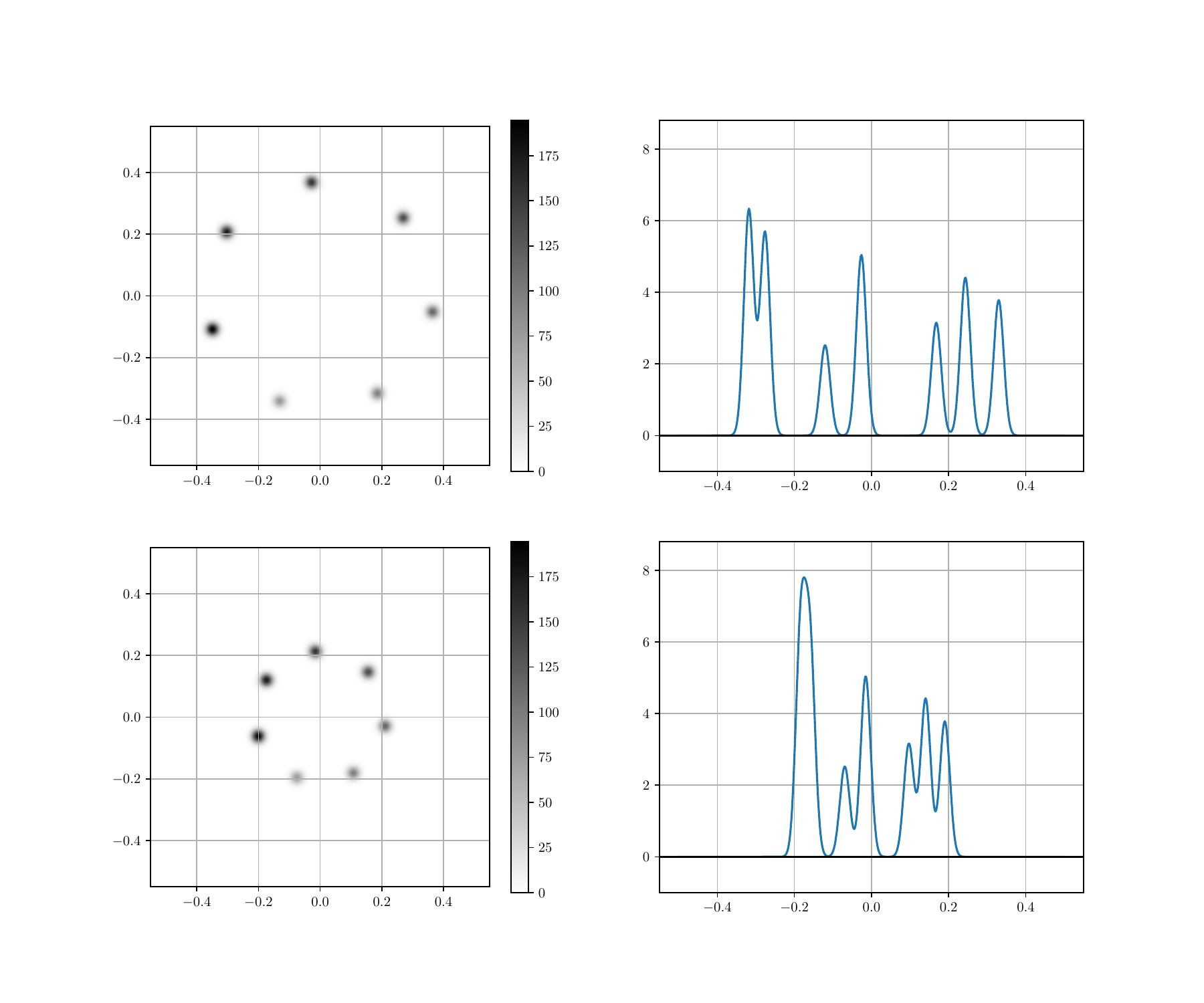}
\caption{The function $F$ from Section \ref{sec:shrink} (top left)
and a shrunken function (bottom left),
with their respective projections onto the $x$-axis.
}
\label{fig:shrink_plots}
\end{figure}

\begin{figure}
\center
\includegraphics[scale=.4]{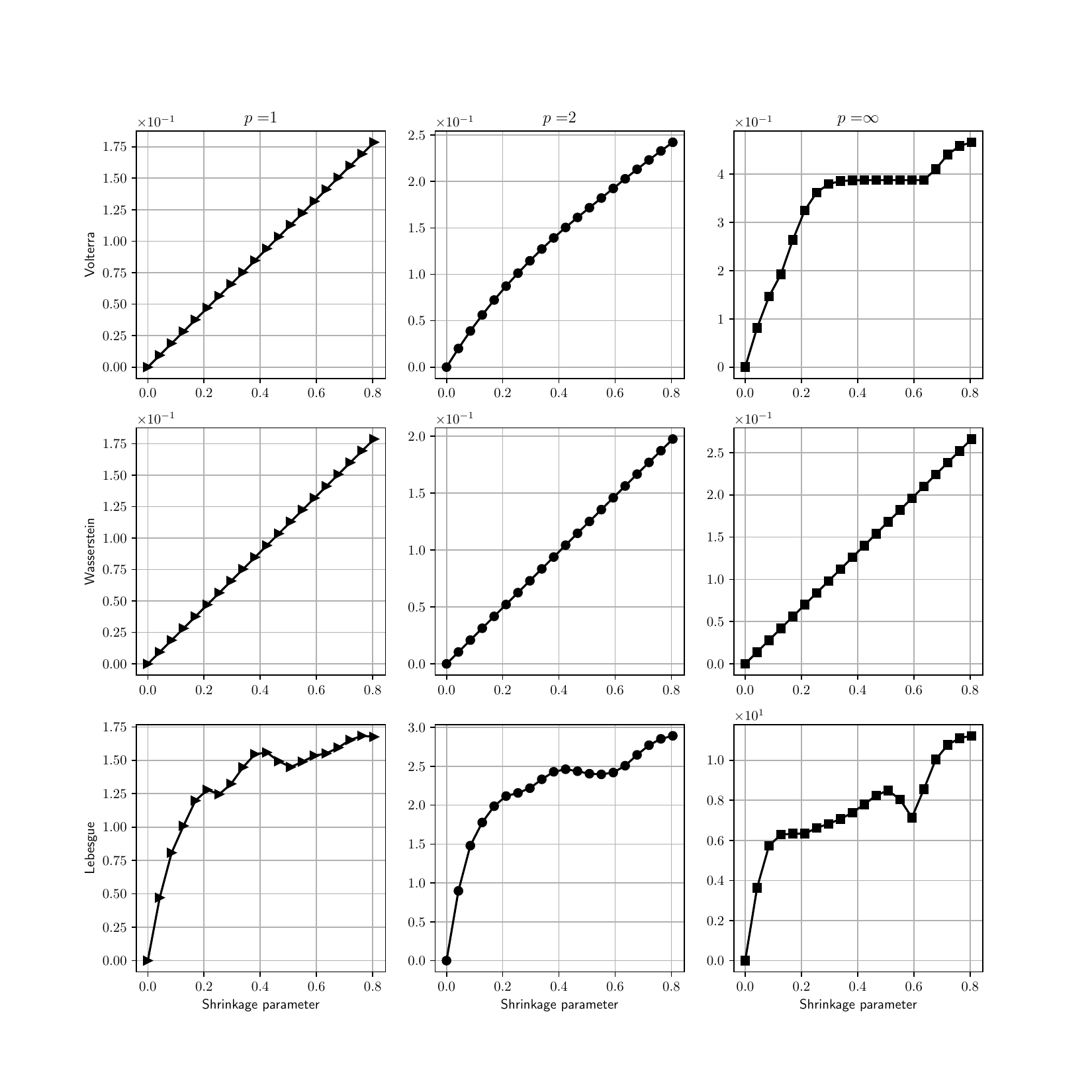}
\caption{The first row shows the approximated Volterra distances
(based on $n=500$ subintervals) between the projections of
$F$ and its shrunken versions, as a function of the shrinkage
(proportional to the distance between the Gaussian centers
of the original function and the shrunken function).
The second row shows the approximated Wasserstein distances,
and the third row shows the approximated Lebesgue distances.
The values of $p$ (from left to right) are $p=1,2, \infty$.
See Section \ref{sec:shrink} for details.
}
\label{fig:shrink_all}
\end{figure}

\subsection{Distances under domain squashing}
\label{sec:squash}

We next illustrate the behavior of the distances
between projections on the function $F$ displayed in
Figure \ref{fig:squash_plots},
given by
\begin{align}
F(\x) = \frac{h}{2 \pi \sigma} \sum_{k=0}^{19} e^{-|\x - \c_k|^2/\sigma}
\end{align}
where $\sigma=1/5000$, $h=1/20$,
and $\c_k = (\cos(\theta_k),\sin(\theta_k))$
with $\theta_k = k\pi/10 + (\sqrt{2} + \sqrt{5} + \sqrt{3}) \pi$,
$0 \le k \le 19$.

The function $F$ is transformed by squashing
the ring, mapping each center $(x,y)$
to $(\lambda x, y/\lambda)$, and then rotating
the result by $\pi/4$.
Figure \ref{fig:squash_plots} shows a heatmap of $F$ and a squashed version,
along with their corresponding projections.
Figure \ref{fig:squash_plots} shows the estimated Volterra $p$-distances (top row),
$p$-Wasserstein (middle row), and Lebesgue $p$-distances (bottom row)
based on samples from $n = 500$ subintervals, plotted as functions of the
squashing parameter $\lambda$.
When $p=2$, the Volterra and Wasserstein distances
have very similar behavior, and both vary smoothly with respect to the
distortion
(when $p=1$, Volterra and Wasserstein are the same, as always).
In this example, unlike previous examples, the $\infty$-Volterra distance
appears to vary more smoothly (for $\lambda$ close to $1$) than the
Wasserstein $\infty$-distance.
As in the other examples, the Lebesgue distances for all $p$ are irregular functions
of the distortion.

\begin{figure}
\center
\includegraphics[scale=.3]{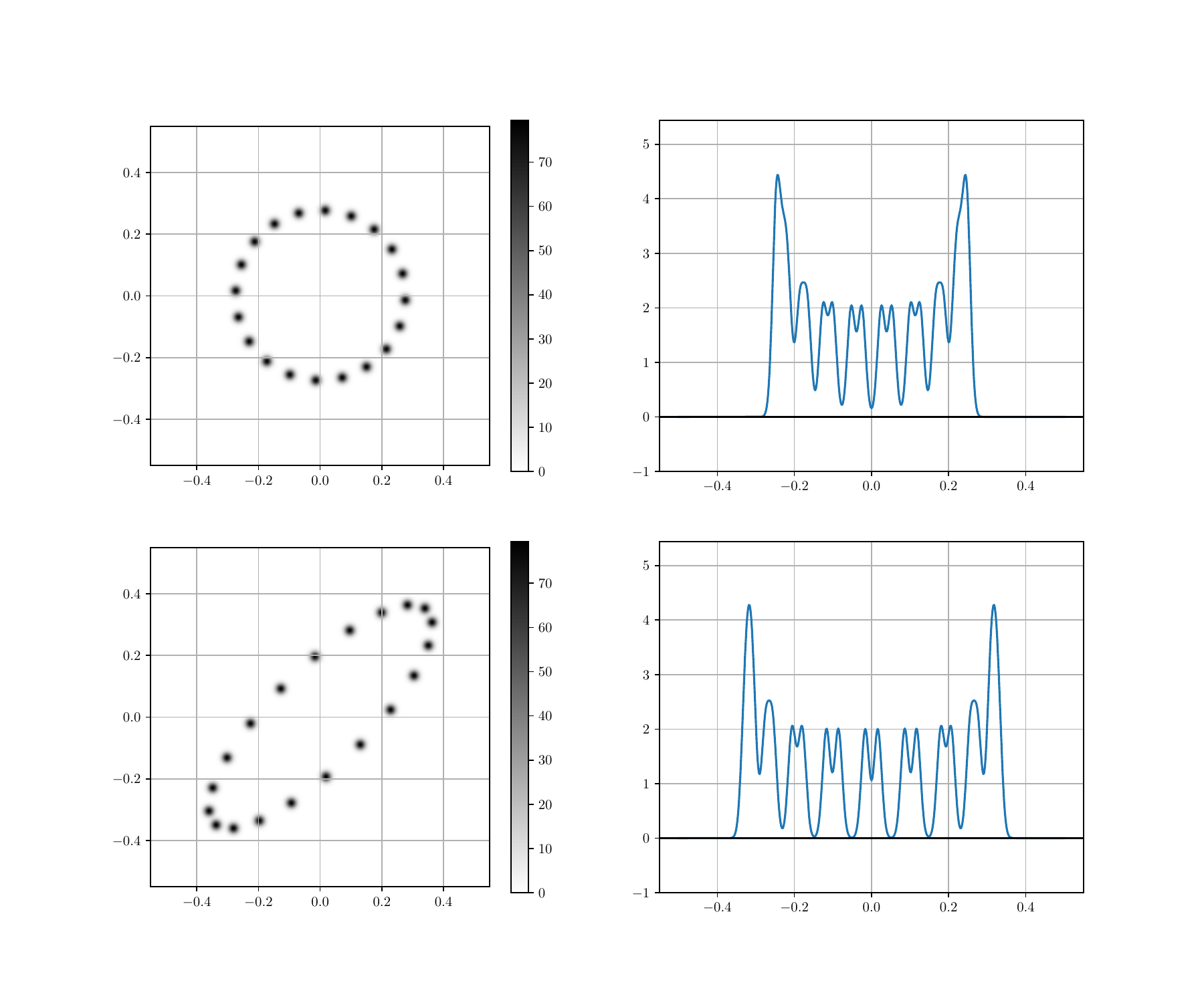}
\caption{The function $F$ from Section \ref{sec:squash} (top left)
and a squashed function (bottom left),
with their respective projections onto the $x$-axis.
}
\label{fig:squash_plots}
\end{figure}

\begin{figure}
\center
\includegraphics[scale=.4]{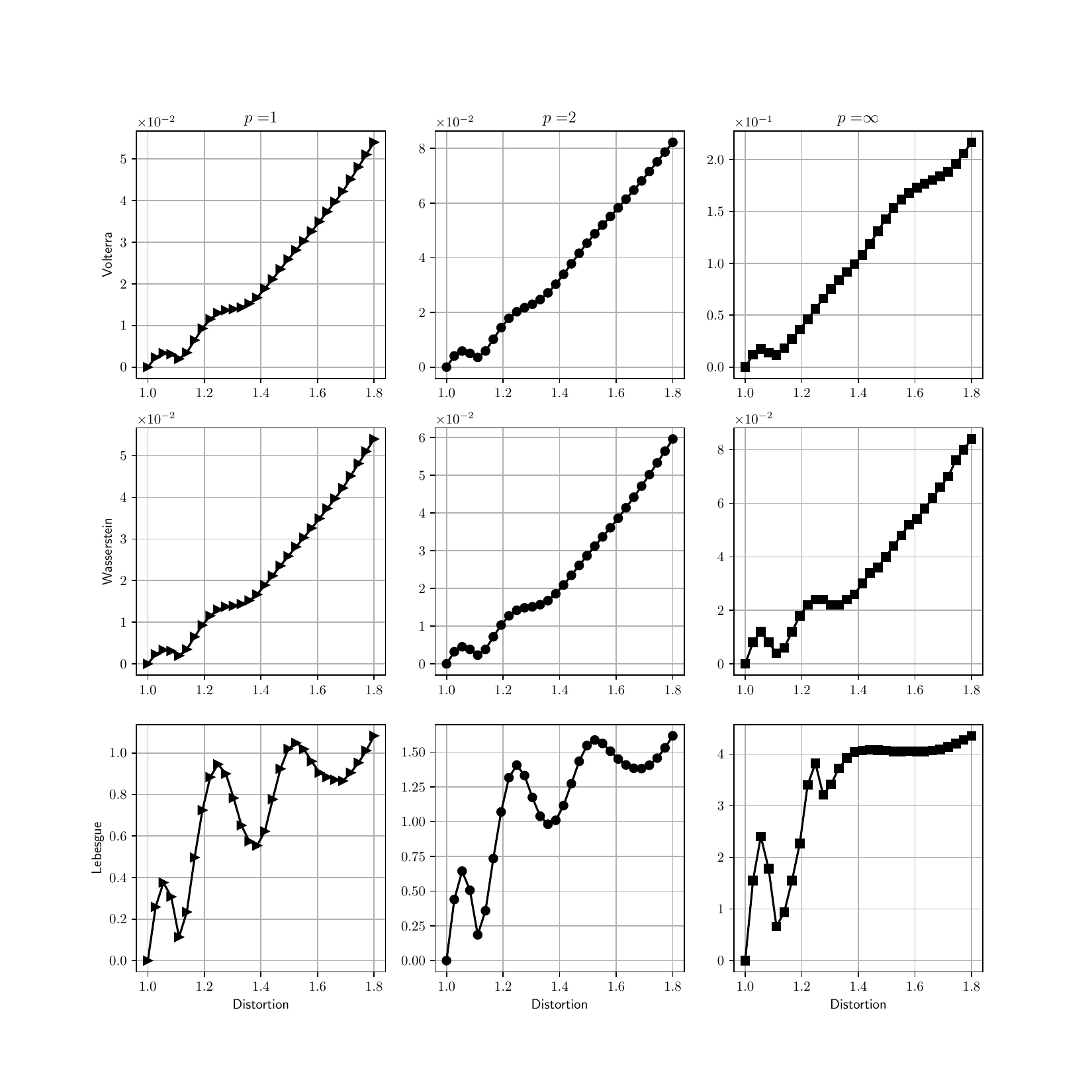}
\caption{The first row shows the approximated Volterra distances
(based on $n=500$ subintervals) between the projections of
$F$ and its squashed versions, as a function of the distortion.
The second row shows the approximated Wasserstein distances,
and the third row shows the approximated Lebesgue distances.
The values of $p$ (from left to right) are $p=1,2, \infty$.
See Section \ref{sec:squash} for details.
}
\label{fig:squash_all}
\end{figure}

\subsection{Robustness to noise I}
\label{sec:noise}

To demonstrate the robustness of the Volterra norms under noise described
by Corollary \ref{cor:convergence},
we run the following experiment.
For different values of $n$,
we take a vector $\bf$ of $n+1$ equispaced samples from the function $f$ on $[-1,1]$ defined by
\begin{align}
f(x) = x e^{-x^2 / 4};
\end{align}
A vector $Z$ of iid Gaussian noise with variance $.01$
is then added to each sample;
let $Y = \bf + Z$.
A plot of a realization of $Y$, when $n=512$,
is shown in the left panel of Figure \ref{fig:noise}.

\begin{figure}
\center
\includegraphics[scale=.5]{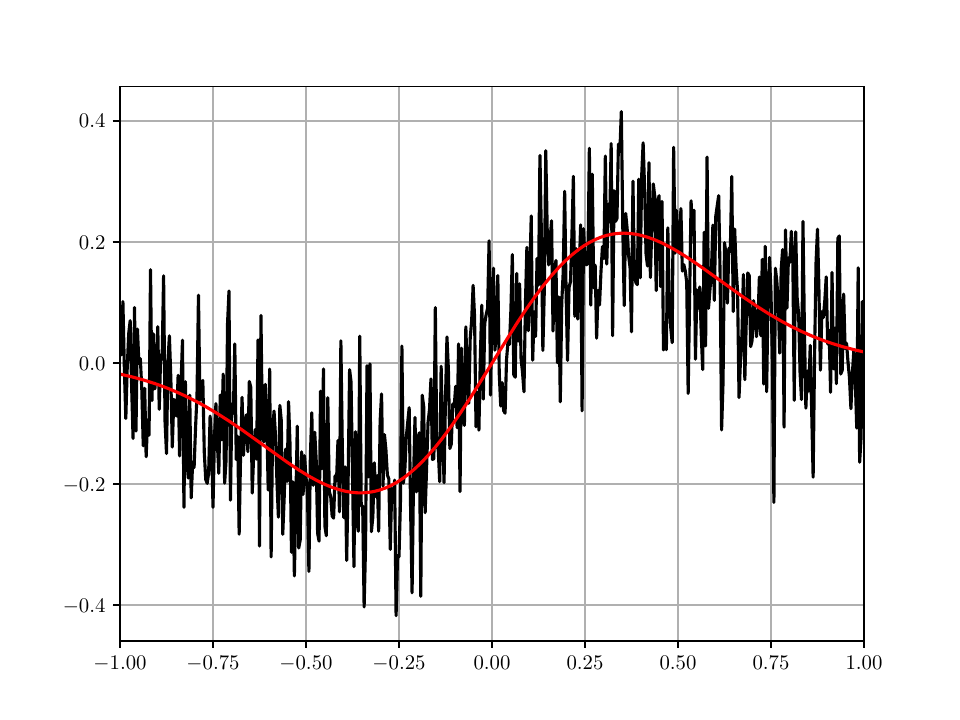}
\includegraphics[scale=.5]{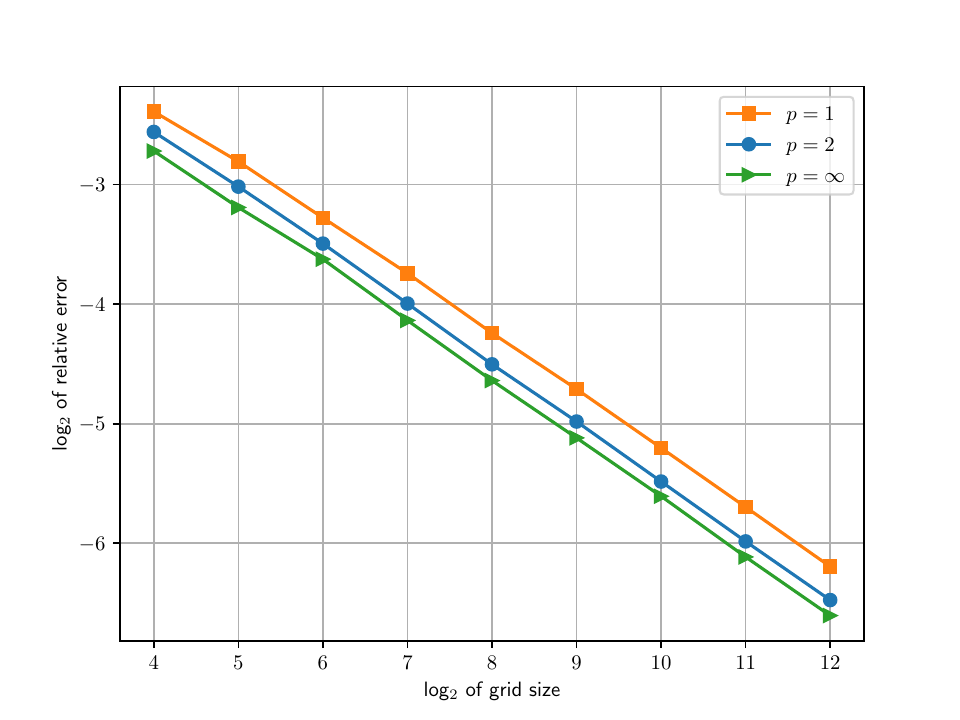}
\caption{The first panel shows a realization of the noisy draws when $n=512$,
with the noiseless curve graphed in red.
The second panel plots $\log_2(\err_{n,p})$ against $\log_2(n)$,
for $p=1,2,\infty$, where the number of draws is $5000$.
The slope of each curve is approximately $-1/2$,
consistent with the error rate predicted by Corollary \ref{cor:convergence}.}
\label{fig:noise}
\end{figure}

For $p=1,2,\infty$, we evaluate the norms $\|Y\|_{\nu_p}$.
For each value of $n$, the experiment is repeated $M = 5000$ times.
Denoting the $M$ random signal-plus-noise vectors by $Y_1, \dots, Y_M$,
we record the average absolute error:
\begin{align}
\err_{n,p} 
= \frac{1}{M} \sum_{k=1}^{M} \frac{\left |\|Y_k\|_{\nu_p} - \|f\|_{V^p} \right|}{\|f\|_{V^p}},
\end{align}
The right panel of Figure \ref{fig:noise} plots $\log_2(\err_{n,p})$
as a function of $\log_2(n)$.
The average error decays like $O(1/\sqrt{n})$ as $n$ increases,
consistent with Corollary \ref{cor:convergence}.

\subsection{Robustness to  noise II}
\label{sec:noise_translation}

Next, we examine the simultaneous effects of
additive noise and deformation on the distance.
We consider the function $f$ defined by
\begin{align}
\label{eq:f_noisytr}
f(x) = \sin(2 \pi x) \chi_{[0,1/2]}(x),
\end{align}
and its shifts $f_\epsilon(x) = f(x - \epsilon)$.
We sample the functions on a grid with $n = 1000$ subintervals,
denoting by $\x$ the vector of samples of $f$
and by $\x_\epsilon$ the vector of samples of $f_\epsilon$.
For different values of $\sigma \ge 0$, we draw a random vector
$\z$ with iid entries $z_j \sim N(0,1)$ and set $\y_{\epsilon,\sigma} = \x_\epsilon + \sigma \z$.
We then compute the distances $\|\x - \y_{\epsilon,\sigma}\|_{\nu_p}$
and $\|\x - \y_{\epsilon,\sigma}\|_{\ell_p}$,
for $p=1,2,\infty$.
The distances are averaged over $5000$ independent realizations
of the noise vector $\z$.
In Figure \ref{fig:noisytr_dists},
we plot these average distances as a function of $\epsilon$,
for different noise levels $\sigma$.
From these plots, we see that the Volterra distances are more
robust to noise than the Lebesgue distances,
though the effect of the noise is larger for smaller values of $p$.

\begin{figure}
\center
\includegraphics[scale=.4]{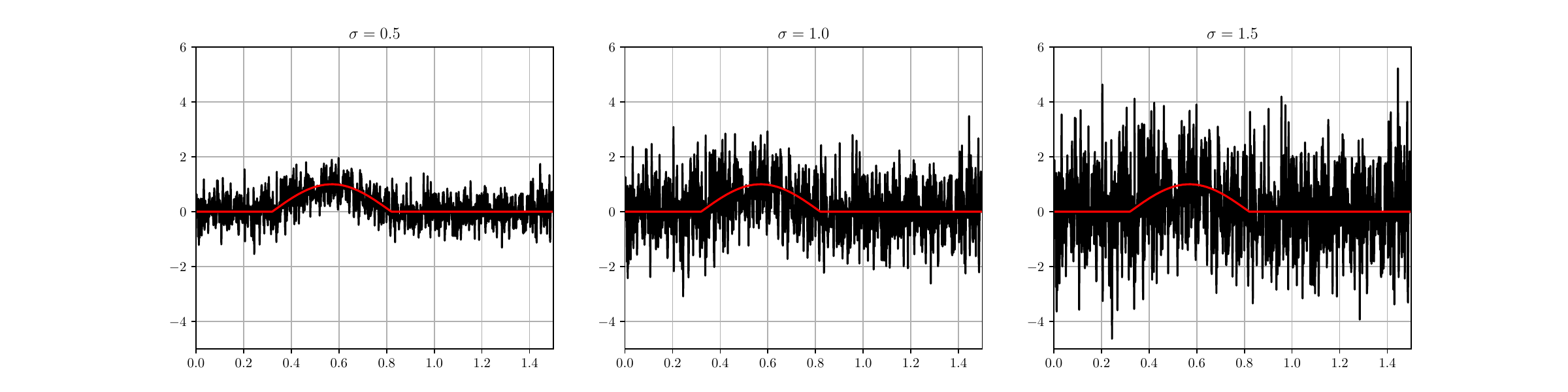}
\caption{Each panel shows shows a realization of the noisy curve (in black)
from Section \ref{sec:noise_translation} with the noise levels $0.5$, $1.0$,
and $1.5$. The noiseless curve is graphed in red.
}
\label{fig:noisytr_curves}
\end{figure}

\begin{figure}
\center
\includegraphics[scale=.5]{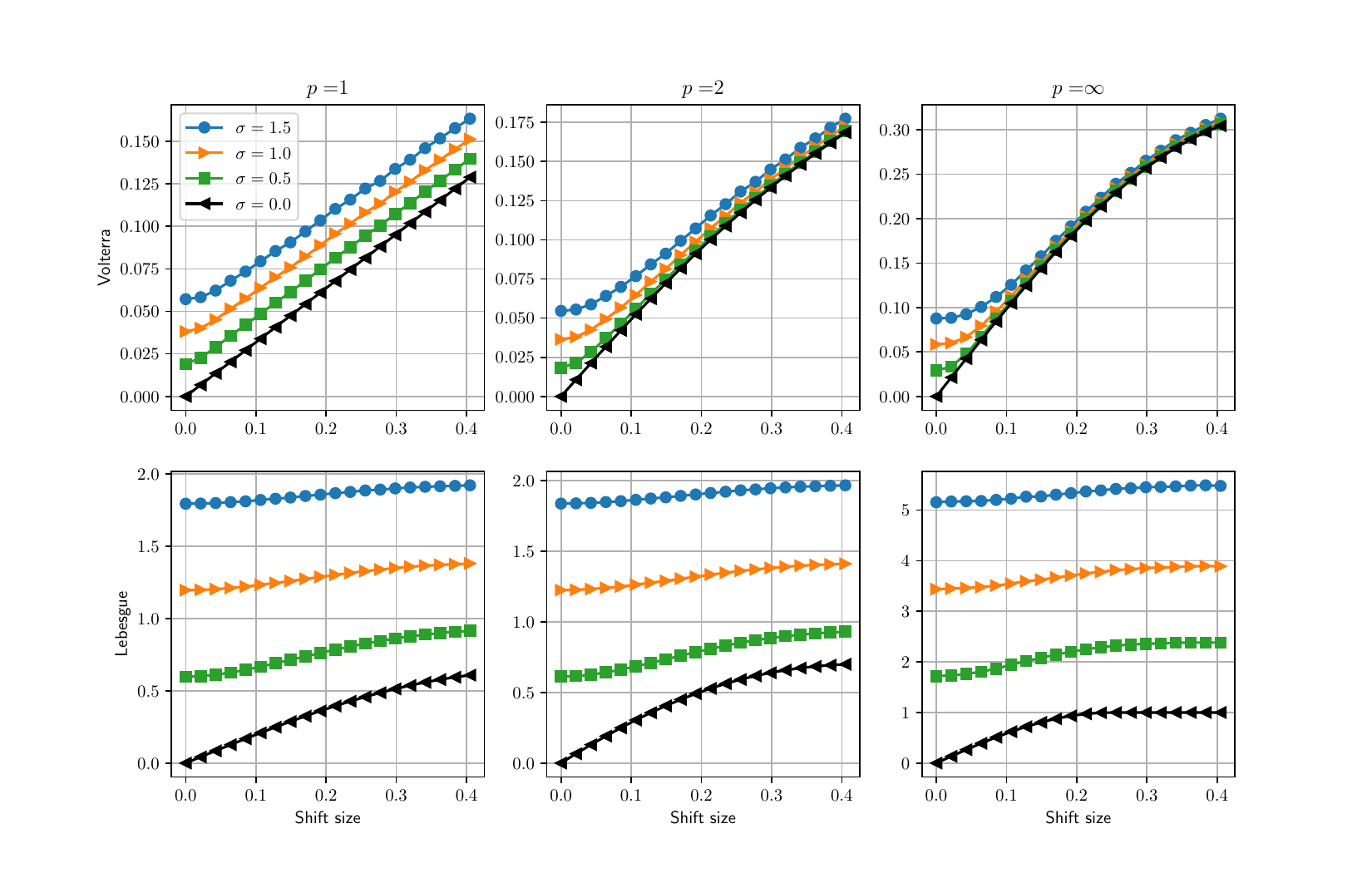}
\caption{The first row shows the approximated Volterra distances
between the function \eqref{eq:f_noisytr} 
and its noisy shifts, as a function of the shift size,
for different noise levels.
The second row shows the approximated Lebesgue distances.
The values of $p$ (from left to right) are $p=1,2, \infty$.
See Section \ref{sec:noise_translation} for details.
}
\label{fig:noisytr_dists}
\end{figure}

\section{Proofs from Section \ref{section:properties}}
\label{section:proofs_properties}

We now turn to the proofs of the main theorems from Section \ref{section:properties}.
Theorems \ref{thm:projpert}, \ref{thm:rotations2D}, and \ref{thm:deformation}
may all be derived as corollaries of the following
general result:

\begin{thm}
\label{thm:general}
Let $A$ and $B$ be non-empty, bounded, open sets in $\R^d$,
with $r = \diam(A \cup B)$.
Let $F: \R^d \to \R$ be in $L^p$ and supported on $A$,
$\Phi : B \to A$ be an $\epsilon$-deformation,
and $F_\Phi = \Phi^{-1}_\sharp F$, i.e.
\begin{align}
F_\Phi(\x) = F(\Phi(\x)) |\det(\nabla \Phi(\x))|
\end{align}
on $B$, and $0$ elsewhere.
Then for any $\u \in \S^{d-1}$,
\begin{align}
\|\calP_\u F - \calP_\u F_\Phi\|_{V^p}
\le \epsilon^{1/p} \cdot \min\left\{C(\Psi,r)^{1-1/p} \cdot \|F\|_{L^p}, \ \|F\|_{L^1}\right\},
\end{align}
where
\begin{align}
C(\Psi,r) = \max_{|t| \le r/2} |\{ \x \in A : \la \x , \u \ra \le t \le \la \Psi(\x),\u \ra
                                  \ \mathrm{ or } \ \la \Psi(\x),\u \ra \le t \le \la \x , \u \ra\}|.
\end{align}

\end{thm}

Theorem \ref{thm:general} is proved in Section \ref{proof:general}.
Sections \ref{proof:projpert}, \ref{proof:rotations2D},
and \ref{proof:deformation} then contain the proofs for
Theorems \ref{thm:projpert}, \ref{thm:rotations2D},
and \ref{thm:deformation}, respectively.

\subsection{Proof of Theorem \ref{thm:general}}
\label{proof:general}

Without loss of generality, suppose $\u = \e_1 = (1,0,\dots,0)$,
and for brevity, if $G : \R^d \to \R$ is a function of $d$ variables,
let
\begin{math}
\calP G = \calP_{\e_1} G.
\end{math}
That is,
\begin{align}
(\calP G)(t) = \int_{\R^{d-1}} G(t,x_1,\dots,x_{d-1}) dx_1 \cdots dx_{d-1}
= \int_{\R^{d-1}} G(t,\x) d\x.
\end{align}

Let $R=r/2$. Also without loss of generality, suppose
that $A$ and $B$ are contained in $\BB = \BB_R$,
the closed ball of radius $R$ centered at the origin.

First, suppose $p=1$. We will show that
\begin{align}
\label{eq:general_one}
\|\calP F - \calP F_\Phi\|_{V^1}
\le \epsilon \cdot \|F\|_{L^1}.
\end{align}

By definition,
\begin{align}
(\calP F)(x) = \int_{\R^{d-1}} F(x,\y) \, d\y,
\end{align}
and
\begin{align}
(\calP F_\Phi)(x) &= \int_{\R} F_\Phi(x,\y) \, d\y
\nonumber \\
&= \int_{\y:(x,\y) \in B} F(\Phi(x,\y)) |\det(\nabla \Phi(x,\y))| \, d\y.
\end{align}
Let $G$ on $[-R,R]$ be absolutely continuous,
with derivative $g = G'$ satisfying $\|g\|_{L^\infty} \le 1$.

Performing the change of variables $\u = \Phi(x,\y)$ gives
\begin{align}
\label{eq:401021-0}
\int_{-R}^{R} G(x) (\calP F_\Phi)(x) \, dx
&= \int_{-R}^{R} G(x) \int_{y:(x,\y) \in B}
        F(\Phi(x,\y)) |\det(\nabla\Phi(x,\y))| \,d\y \,dx
\nonumber \\
&= \int_{B} G(x)  F(\Phi(x,\y)) |\det(\nabla\Phi(x,\y))| \,d\y \,dx
\nonumber \\
&= \int_{A} G(\psi_1(\u)) F(\u)  \,d\u,
\end{align}
and similarly,
\begin{align}
\label{eq:401021-1}
\int_{-R}^{R} G(x) (\calP F)(x) \, dx
&= \int_{A} G(x_1) F(\x)  \,d \x.
\end{align}

We then have
\begin{align}
\label{eq:401021}
\int_{-R}^{R} G(x) ((\calP F)(x) - (\calP F_\Phi)(x)) \, dx
&= \int_{A} G(x_1) F(\x)  \,d\x - \int_{A} G(\psi_1(\x)) F(\x)  \,d\x
\nonumber \\
&= \int_{A} (G(x_1) - G(\psi_1(\x)))  F(\x) \,d\x
\nonumber \\
&\le \|F\|_{L^1} \max_{\x \in A}|G(x_1) - G(\psi_1(\x))|.
\end{align}

Now, because $g = G'$ satisfies $\|g\|_{L^\infty} \le 1$,
we have
\begin{align}
|G(x_1) - G(\psi_1(\x))|
&= \left| \int_{x_1}^{\psi_1(\x)} g(t) \, dt \right|
\nonumber \\
&\le \|g\|_{L^\infty} |x_1 - \psi_1(\x)|
\nonumber \\
&\le | \x - \Psi(\x) |
\nonumber \\
&\le \epsilon,
\end{align}
and therefore, taking the supremum over all such $G$
and using Proposition \ref{prop:variational} shows that
\begin{align}
\label{eq:59910301-0}
\| \calP F - \calP F_\Phi\|_{V^1} \le \epsilon \cdot \|F\|_{L^1}.
\end{align}
This completes the proof when $p=1$.

We will now prove that
\begin{align}
\label{eq:general_infty}
\|\calP F - \calP F_\Phi\|_{V^\infty} \le C(\Psi,r) \|F\|_{L^\infty}.
\end{align}

Let $I_{\x}$ be the interval $[x_1,\psi_1(\x)]$ when $x_1 \le \psi_1(\x)$,
and $[\psi_1(\x),x_1]$ when $x_1 > \psi_1(\x)$;
and let $\chi(\x,t)$ be $1$ if $t \in I_{\x}$, and $0$ otherwise.

Now, take an absolutely continuous $G$ on $[-R,R]$ whose derivative
$g = G'$ satisfies $\|g\|_{L^1} = 1$.
Using \eqref{eq:401021-0} and \eqref{eq:401021-1}
as before, we have
\begin{align}
\int_{-R}^{R} G(x) ((\calP F)(x) - (\calP F_\Phi)(x)) \, dx
&= \int_{A} G(x_1) F(\x)  \,d\x - \int_{A} G(\psi_1(\x)) F(\x)  \,d\x
\nonumber \\
&= \int_{A} (G(x_1) - G(\psi_1(\x)))  F(\x) \,d\x
\nonumber \\
&\le \|F\|_{L^\infty} \int_{A} |G(x_1) - G(\psi_1(\x))| \,d\x.
\end{align}
Then, using $g = G'$, we have
\begin{align}
\int_{A} |G(x_1) - G(\psi_1(\x))| \,d\x
&= \int_{A} \left| \int_{I_{\x}} g(t) dt \right| \,d\x
\nonumber \\
&= \int_{A} \left| \int_{-R}^{-R} g(t) \chi(\x,t) \,dt \right| \,d\x
\nonumber \\
&\le \int_{A} \int_{-R}^{R} | g(t) | \chi(\x,t) \,dt  \,d\x
\nonumber \\
&= \int_{-R}^{R} | g(t) | \int_{A}  \chi(\x,t) \,d\x \,dt
\nonumber \\
&\le \|g\|_{L^1} \sup_{|t| \le R}\int_{A} \chi(\x,t) \,d\x
\nonumber \\
&=\sup_{|t| \le R}\int_{A} \chi(\x,t) \,d \x
\nonumber \\
&= C(\Psi,r),
\end{align}
which completes the proof of \eqref{eq:general_infty}.

Finally, since the mapping $F \mapsto \int_{-R}^x (\calP F - \calP F_\Phi)$ is
linear, we may now combine the bounds \eqref{eq:general_one}
and \eqref{eq:general_infty} using
the Riesz-Thorin Interpolation Theorem
(see, e.g. Theorem~6.27 in \cite{folland1999real})
to complete the proof that
\begin{align}
\|\calP_\u F - \calP_\u F_\Phi\|_{V^p}
\le \epsilon^{1/p} \cdot C(\Psi,r)^{1-1/p} \cdot \|F\|_{L^p}.
\end{align}

Next, we will show that
\begin{align}
\|\calP_\u F - \calP_\u F_\Phi\|_{V^p}
\le \epsilon^{1/p} \cdot \|F\|_{L^1}.
\end{align}
Let $G$ be absolutely
continuous on $[-R,R]$, with derivative $g = G'$ satisfying
$\|g\|_{L^q} \le 1$.
From  \eqref{eq:401021},
\begin{align}
\int_{-R}^{R} G(x) ((\calP F)(x) - (\calP F_\Phi)(x)) \, dx
\le \|F\|_{L^1} \max_{\x \in A}|G(x_1) - G(\psi_1(\x))|,
\end{align}
and so we must show that for all $\x \in A$,
\begin{align}
|G(x_1) - G(\psi_1(\x))| \le \epsilon^{1/p}.
\end{align}

As before, let $I_{\x}$ be the interval $[x_1,\psi_1(\x)]$ if
$x_1 \le \psi_1(\x)$, and $[\psi_1(\x),x_1]$ if $\psi_1(\x) \le x_1$,
and let $\chi(\x,t)$ be $1$ if and only if $t \in I_{\x}$,
and $0$ otherwise. Note that
\begin{align}
\label{eq:69503010}
\int_{-R}^{R} \chi(\x,t) \, dt
\le |x_1 - \psi_1(\x)|
\le | \x - \Psi(\x) |
\le \epsilon.
\end{align}

Using that $G(y) = -\int_{y}^{b} g(t)dt$, we may write,
for any $\x \in A$,
\begin{align}
|G(x_1) - G(\psi_1(\x))|
= \left| \int_{I_\x} g(t) \, dt \right|,
\end{align}
and H\"{o}lder's inequality yields
\begin{align}
\left| \int_{I_\x} g(t) \, dt \right|
&= \left| \int_{-R}^{R} g(t) \chi(\x,t) dt \right|
\nonumber \\
&\le \|g\|_{L^q} \left(\int_{-R}^{R} \chi(\x,t)^p \, dt \right)^{1/p}
\nonumber \\
&\le \left(\int_{-R}^{R} \chi(\x,t) \, dt \right)^{1/p}
\nonumber \\
&\le \epsilon^{1/p},
\end{align}
where the last inequality follows from \eqref{eq:69503010}.
This completes the proof.

\subsection{Proof of Theorem \ref{thm:projpert}}
\label{proof:projpert}

Without loss of generality, suppose that $A$ and $B$ are contained in $\BB = \BB_{r/2}$,
and that $\u = \e_1 = (1,0,\dots,0)$; and let $\calP = \calP_{\e_1}$.

We will first show that
for all $|t| \le r/2$,
\begin{align}
C(\Psi,r) \le 2 r^{d-1} \epsilon.
\end{align}

Let $\SS_1 = \{\x \in A \,:\, x_1 \le t \le \psi_1(\x)\}$,
and let $\SS_2 = \{\x \in A \,:\, \psi_1(\x) \le t \le x_1\}$. Then
\begin{align}
C(\Psi,r) = |\SS_1 \cup \SS_2|.
\end{align}

To bound the area of $\SS_1$, observe first that any $\x$
contained in $\SS_1$ must satisfy $t-\epsilon \le x_1 \le t$.
Indeed, since, by assumption, $\psi_1(\x) - x_1 \le \epsilon$, we have
\begin{align}
x_1 \ge \psi_1(\x) - \epsilon \ge t - \epsilon,
\end{align}
as claimed. Consequently, since $A \subset \BB_{r/2}$,
\begin{align}
|\SS_1|
\le |\{\x \in \BB_{r/2}\ \,:\, t-\epsilon \le x_1 \le t\}|
\le r^{d-1} \epsilon.
\end{align}
Similarly, $|\SS_2| \le r^{d-1} \epsilon$, and hence
\begin{align}
C(\Psi,r)
= |\SS_1 \cup \SS_2|
\le 2 r^{d-1} \epsilon,
\end{align}
as claimed.

Consequently, Theorem \ref{thm:general} states that
\begin{align}
\label{eq:5940201-0}
\|\calP F - \calP F_\Phi\|_{V^p}
&\le \epsilon^{1/p} \cdot \min\left\{(2r^{d-1}\epsilon)^{1-1/p} \cdot \|F\|_{L^p}, \
                                    \|F\|_{L^1}\right\}
\nonumber \\
&= \min\left\{\epsilon \cdot (2r^{d-1})^{1-1/p} \cdot \|F\|_{L^p},\ 
              \epsilon^{1/p} \cdot \|F\|_{L^1}\right\}.
\end{align}

Switching the roles of $\Psi$ and $\Phi$, and using that
$(F_\Phi)_\Psi = F$ and $\|F\|_{L^1} = \|F_\Phi\|_{L^1}$, shows the bound 
\begin{align}
\label{eq:5940201-1}
\|\calP F - \calP F_\Phi\|_{V^p}
= \|\calP (F_\Phi)_\Psi - \calP F_\Phi\|_{V^p}
\le \min\left\{\epsilon \cdot (2r^{d-1})^{1-1/p} \cdot \|F_\Phi\|_{L^p},\ 
              \epsilon^{1/p} \cdot \|F\|_{L^1}\right\}.
\end{align}
Combining \eqref{eq:5940201-0} and \eqref{eq:5940201-1}
completes the proof.

\subsection{Proof of Theorem \ref{thm:rotations2D} }
\label{proof:rotations2D}

Without loss of generality, we assume that $\varphi = 0$.
Let $c = \cos(\theta)$ and $s = \sin(\theta)$, and
let $f = f_\varphi = f_0$.
We have $\Phi(x,y) = (c x + sy, cy - sx)$,
and $\Psi(x,y) = (c x - sy, cy + sx)$.
With $\epsilon = \max_{(x,y) \in \BB_R} |\Phi(x,y) - (x,y) |$,
we will prove
\begin{align}
\label{eq:rotation_epsilon}
\epsilon \le 2 R \sin(\theta/2),
\end{align}
and
\begin{align}
\label{eq:rotation_constant}
C(\Psi,2R) \le R^2 \theta.
\end{align}

Assuming these bounds, Theorem \ref{thm:general}
states that
\begin{align}
\|f - f_\theta\|_{V^p}
&\le \epsilon^{1/p} \cdot  \min\left\{C(\Psi,r)^{1-1/p} \cdot \|F\|_{L^p}, \
        \|F\|_{L^1}  \right\}
\nonumber \\
&\le (2 R \sin(\theta/2))^{1/p} \cdot \min\left\{ (R^2 \theta)^{1-1/p} \cdot \|F\|_{L^p}, \
        \|F\|_{L^1}  \right\}
\nonumber \\
&= (2 \sin(\theta/2))^{1/p} \cdot
        \min\left\{\theta^{1-1/p}\cdot \|F\|_{L^p} \cdot R^{2-1/p}  ,\
        \|F\|_{L^1} \cdot R^{1/p} \right\},
\end{align}
which is the desired result.

To prove \eqref{eq:rotation_epsilon},
suppose $x^2 + y^2 \le R^2$. Then from the double angle formula
$c = \cos(\theta) = \cos^2(\theta/2) - \sin^2(\theta/2)$,
or equivalently $1-c = 2 \sin^2(\theta/2)$;
therefore,
\begin{align}
|\Phi(x,y) - (x,y) |^2
&= (x - cx - sy)^2 + (y - cy + sx)^2
\nonumber \\
&= (1-c)^2 x^2 + s^2 y^2 - 2(1-c)sxy + (1-c)^2y^2 + s^2 x^2 + 2(1-c)sxy
\nonumber \\
&\le (1-c)^2 R^2 + s^2 R^2
\nonumber \\
&\le (1-c)^2 R^2 + (1-c^2) R^2
\nonumber \\
&= 2(1-c) R^2
\nonumber \\
&= 4 R^2\sin^2(\theta/2),
\end{align}
and hence
\begin{align}
\epsilon \equiv\max_{(x,y) \in \BB_R} |\Phi(x,y) - (x,y)|
\le 2 R \sin(\theta/2),
\end{align}
as desired.

Next, to prove \eqref{eq:rotation_constant},
let $I_{x,y}$ be the interval $[x,cx+sy]$
when $x \le cx + sy$, and the interval $[cx + sy, x]$ when
$cx + sy \le x$; and let $\chi(x,y,t)$ be $1$ if $t \in I_{x,y}$ and $0$ otherwise.
Then $C(\Psi,2R) = \max_{|t| \le R} \int_{\BB_R} \chi(x,y,t) dx dy$, and
so we need to show that for all $t \in [-R,R]$,
\begin{align}
\label{eq:bound44}
\int_{\BB_R} \chi(x,y,t) \, dx \, dy \le R^2 \theta.
\end{align}
It is enough to show this for $R=1$, since $C(\Psi,2R) = R^2 C(\Psi,2)$.

To that end, let $\BB = \BB_1$, and
observe that for all $|t| \le 1$,
\begin{align}
\int_{\BB} \chi(x,y,t) \,dx \,dy = 2 |\SS_{t,(1,0),(c,s)} \cup \SS_{t,(c,s),(1,0)}|,
\end{align}
where, for unit vectors $\v$ and $\w$, $\SS_{t,\v,\w}$ is the region defined by
\begin{align}
\SS_{t,\v,\w} = \{\u \in \BB: \langle \u, \v \rangle \le t \le \langle \u, \w \rangle \}.
\end{align}
By rotational symmetry, the following lemma is immediate:
\begin{lem}
If $\a$ and $\b$ are any unit vectors with angle $\theta$, then
$|\SS_{t,(1,0),(c,y)}| = |\SS_{t,\a,\b}|$.
Furthermore, $|\SS_{t,\a,\b}| = |\SS_{-t,\a,\b}|$,
and $|\SS_{t,\a,\b} \cap \SS_{-t,\a,\b}| = 0$.
\end{lem}

By this lemma, it follows that
\begin{align}
\int_{\BB} \chi(x,y,t) dxdy = 2 |\SS_{t,\v,\w}|
= 2\left| \left\{ \u \in \BB: \langle \u,\v\rangle \le t \le \langle \u,\w\rangle \right\}\right|,
\end{align}
where $\w = (\cos(\theta/2),\sin(\theta/2))$ and
$\v = (\cos(\theta/2),-\sin(\theta/2))$.
Furthermore, we can restrict to $t \ge 0$.

It will be convenient to refer to Figure \ref{fig:tikz_circle},
where $\w$ corresponds to the point labeled $B$,
and $\v$ corresponds to the point labeled $E$. 
In the figure, the line $AD$ is perpendicular to $OB$,
and intersects $OB$ at distance $t$ from the origin;
consequently, the set of all vectors $\u$ in $\BB$ with
$\langle \u , \w \rangle \ge t$ is the circular segment
through the points $A$, $B$ and $D$.
Similarly; the line $CF$ is perpendicular to $OE$,
and intersects $OE$ at distance $t$ from the origin;
consequently, the set of all vectors $\u$ in $\BB$ with
$\langle \u , \v \rangle \le t$ is the circular segment
through the points $C$, $A$ and $F$.
Denote by $G$ the point at the intersection of the lines $AD$
and $CF$; then when $G$ lies within the circle,
the intersection of these two circular segments is the region bounded
by $A$, $C$ and $G$. (When $G$ is outside the circle, then
the two circular segments are disjoint.)

To evaluate the area of this region, we will first find the area of the full circular segment through $A$, $B$ and $D$, and then subtract off the area of the region bounded by $C$, $G$ and $D$.

\begin{lem}
\label{lem:circular_segment}
The area of the circular segment through $A$, $B$ and $D$ is
\begin{align}
\arccos(t) - t \sqrt{1 - t^2},
\end{align}
where $\arccos$ takes values in $[0,\pi]$.
\end{lem}

\begin{proof}
This is immediate from the well-known formula for the
area of a circular segment, and the fact that the line segment
from $O$ to $H$ has length $t$.
\end{proof}

The next lemma is also elementary, and likely known already;
however, since we could not find the exact identity in the literature,
we provide a self-contained proof.

\begin{lem}
\label{lem:area_intersection}
When $t \le \cos(\theta/2)$, the intersection between
the circular segment bounded by $A$, $B$ and $D$ and the
circular segment bounded by $C$, $E$ and $F$ has area
\begin{align}
\arcsin\left( \sqrt{1-t^2} \cos(\theta/2) - t \sin(\theta/2)\right) - t\sqrt{1 - t^2} + t^2 \tan(\theta/2),
\end{align}
where $\arcsin$ denotes the inverse of $\sin$ on the interval $[0,\pi/2]$
(and hence takes values in this interval).
When $t > \cos(\theta/2)$, the two circular segments are disjoint.
\end{lem}

\begin{proof}
We begin by showing the second part, namely that when $t > \cos(\theta/2)$,
the circular segments are disjoint, or equivalently that the point $G$
lies outside of the circle.
Indeed, it is straightforward to show that $G$ is located at the point
$(t/\cos(\theta/2),0)$; hence, $G$ is inside the circle so long as
$t / \cos(\theta/2) \le 1$,
or equivalently,
$t \le \cos(\theta/2)$, as desired.

Let us now suppose that $t \le \cos(\theta/2)$, and evaluate the area of the region bounded by $C$, $G$, and $D$. The line segment from $G$ to $D$ has arc-length parameterization
\begin{align}
\alpha(s) = t ( \cos(\theta/2) ,\sin(\theta/2)) + s ( \sin(\theta/2),- \cos(\theta/2)),
\end{align}
and the line segment from $C$ to $G$ has arc-length parameterization
\begin{align}
\beta(s) = t(\cos(\theta/2) , -\sin \theta/2) + (\sqrt{1-t^2}+t\cdot \tan(\theta/2) - s)
    (\sin(\theta/2),\cos(\theta/2)),
\end{align}
where
\begin{align}
t \cdot \tan(\theta/2) \le s \le \sqrt{1-t^2}.
\end{align}
The counterclockwise arc from $D$ to $C$ has arc-length parameterization
\begin{align}
\gamma(\varphi) = (\cos(\varphi), \sin(\varphi)),
\end{align}
where
\begin{align}
-\arcsin\left(\sqrt{1-t^2} \cos(\theta/2) - t \cdot \sin(\theta/2) \right)
\le \varphi
\le \arcsin\left(\sqrt{1-t^2} \cos(\theta/2) - t \cdot \sin(\theta/2) \right).
\end{align}
When $t \le \cos(\theta/2)$, we will evaluate the area using Green's Theorem, by computing
\begin{math}
\frac{1}{2} \oint (xdy -  ydx)
\end{math}
over each curve. For $\alpha$, we have
\begin{align}
\frac{1}{2} \oint_\alpha xdy
&= \frac{1}{2} \int_{t \cdot \tan(\theta/2)}^{\sqrt{1-t^2}} \left[(t \cdot \cos(\theta/2) + s\cdot \sin(\theta/2)) (-\cos(\theta/2)) \right] ds
\nonumber \\
&= \frac{t \cdot \cos^2(\theta/2)}{2} \left(t \cdot \tan(\theta/2) - \sqrt{1-t^2} \right)
    + \frac{\sin(\theta/2)\cos(\theta/2)}{4}\left(t^2 \cdot \tan^2(\theta/2) - 1 + t^2\right),
\end{align}
and 
\begin{align}
\frac{1}{2} \oint_\alpha y dx
&= \frac{1}{2} \int_{t \cdot \tan(\theta/2)}^{\sqrt{1-t^2}} \left[(t \cdot \sin(\theta/2) - s\cdot \cos(\theta/2)) (\sin(\theta/2)) \right] ds
\nonumber \\
&= \frac{t \cdot \sin^2(\theta/2)}{2} \left(\sqrt{1-t^2} - t \cdot \tan(\theta/2) \right)
    + \frac{\sin(\theta/2)\cos(\theta/2)}{4}(t^2 \cdot \tan^2(\theta/2) - 1 + t^2),
\end{align}
and hence
\begin{align}
\frac{1}{2} \oint_\alpha (xdy-ydx)
= \frac{t}{2} \cdot \left(t \cdot \tan(\theta/2) - \sqrt{1-t^2} \right).
\end{align}
Similarly,
\begin{align}
\frac{1}{2} \oint_\beta (xdy-ydx)
= \frac{t}{2} \cdot \left(t \cdot \tan(\theta/2) - \sqrt{1-t^2} \right).
\end{align}
Finally, it is straightforward to check that
\begin{align}
\frac{1}{2} \oint_\gamma (xdy-ydx) = \arcsin\left(\sqrt{1-t^2} \cos(\theta/2) - t \cdot \sin(\theta/2) \right).
\end{align}
Adding all three integrals together, we find that the area of the region is
\begin{align}
\frac{1}{2} \oint_\gamma (xdy-ydx) = \arcsin\left(\sqrt{1-t^2} \cos(\theta/2) - t \cdot \sin(\theta/2) \right) - t\sqrt{1-t^2} +  t^2 \tan(\theta/2),
\end{align}
as claimed.
\end{proof}

\begin{figure}
\centering

\tikzmath{
    \radius = 3;
    \angl = pi/4.6 * (180/pi);
    \angl = pi/3.6 * (180/pi);
    \hh = cos(\angl);
    \xcoord = \radius*\hh;
    \ycoord = sqrt(\radius*\radius-\xcoord*\xcoord);
    \tt=.4*\radius;
    \ss=sqrt(\radius*\radius-\tt*\tt);
    \xorth=\tt*cos(\angl) + \ss*sin(\angl);
    \yorth=-\tt*sin(\angl) + \ss*cos(\angl);
    \ss2=-sqrt(\radius*\radius-\tt*\tt);
    \xort=\tt*cos(\angl) + \ss2*sin(\angl);
    \yort=-\tt*sin(\angl) + \ss2*cos(\angl);
%
%
    \xinter = \tt*cos(\angl) + \tt*sin(\angl)*tan(\angl);
%
%
    \offset=.3;
}

\begin{tikzpicture}

\draw[thick] (0,0) circle (\radius);

\draw[thick] (0,0) -- (\xcoord,\ycoord);
\draw[thick] (0,0) -- (\xcoord,-\ycoord);

\draw[thick] (\xorth,\yorth) -- (\xort,\yort);
\draw[thick] (\xorth,-\yorth) -- (\xort,-\yort);

%
%
\tikzmath{
\fff=.2;
}

\draw (\radius*\fff,0) arc (0:\angl:\radius*\fff);
\draw (\radius*\fff,0) arc (0:-\angl:\radius*\fff);

\node(draw) at (\offset,0) {$\theta$};

%
%

%
%
\filldraw (\xcoord,\ycoord) circle (2pt);
\node(draw) at (\xcoord+\offset,\ycoord+\offset) {$B$};

\filldraw (\xcoord,-\ycoord) circle (2pt);
\node(draw) at (\xcoord+\offset,-\ycoord-\offset) {$E$};

%
%
\filldraw (\xorth,\yorth) circle (2pt);
\node(draw) at (\xorth+\offset,\yorth+\offset) {$C$};

\filldraw (\xorth,-\yorth) circle (2pt);
\node(draw) at (\xorth+\offset,-\yorth-\offset) {$D$};

\filldraw (\xort,-\yort) circle (2pt);
\node(draw) at (\xort-\offset,-\yort+\offset) {$A$};

\filldraw (\xort,\yort) circle (2pt);
\node(draw) at (\xort-\offset,\yort-\offset) {$F$};

%
%
\filldraw (0,0) circle (2pt);
\node(draw) at (-\offset,0) {$O$};

%
%
\filldraw (\xinter,0) circle (2pt);
\node(draw) at ({\xinter+sqrt(2)*\offset},0) {$G$};

%
%

\tikzmath{
    \xmeet = \tt*cos(\angl);
    \ymeet = \tt*sin(\angl);
    \sss=\offset*.7;
    \xmeetp = \xmeet + \sss*sin(\angl);
    \ymeetp = \ymeet - \sss*cos(\angl);
    \xmeetpp = \xmeetp - \sss*cos(\angl);
    \ymeetpp = \ymeetp - \sss*sin(\angl);
    \xmeetm = \xmeet - \sss*cos(\angl);
    \ymeetm = \ymeet - \sss*sin(\angl);
}

\filldraw (\xmeet,\ymeet) circle (2pt);
\node(draw) at ({\xmeet-sqrt(2)*\offset},\ymeet) {$H$};

\filldraw (\xmeet,-\ymeet) circle (2pt);
\node(draw) at ({\xmeet-sqrt(2)*\offset},-\ymeet) {$I$};

%
%
\draw (\xmeetp,\ymeetp) -- (\xmeetpp,\ymeetpp);
\draw (\xmeetm,\ymeetm) -- (\xmeetpp,\ymeetpp);
\draw (\xmeetp,-\ymeetp) -- (\xmeetpp,-\ymeetpp);
\draw (\xmeetm,-\ymeetm) -- (\xmeetpp,-\ymeetpp);

\end{tikzpicture}

\caption{Diagram for the proof of \eqref{eq:bound44}. 
The points labeled $B$ and $C$ are located at
$(\cos(\theta/2),\sin(\theta/2))$ and
$(\cos(\theta/2),-\sin(\theta/2))$, respectively.
The point labeled $O$ is the origin, $(0,0)$.
The line segment $OB$ is orthogonal to the line $AD$,
and the line segment $OE$ is orthogonal to the line $FC$.
The line segments $OH$ and $OI$ each have length $t$.}
\label{fig:tikz_circle}
\end{figure}

From Lemmas \ref{lem:circular_segment} and \ref{lem:area_intersection}, we find
\begin{align}
\label{eq:area_difference}
&\frac{1}{2}\int_{\BB} \chi(x,y,t) dxdy
\nonumber \\
=\,&
\begin{cases}
\arccos(t) - t \sqrt{1 - t^2} &\text{ if } t > \cos(\theta/2); \\
\arccos(t) -\arcsin\left(\sqrt{1-t^2} \cos(\theta/2) - t \cdot \sin(\theta/2) \right)  -  t^2 \tan(\theta/2), & \text{ if } t \le \cos(\theta/2).
\end{cases}
\end{align}

To conclude the proof, we must show that this expression is bounded
above by $\theta/2$ for all values of $t$ between $0$ and $1$.
In fact, we will show that \eqref{eq:area_difference} is a decreasing
function of $t$, and hence is maximized at $t=0$.
It is immediately apparent that the expression is
decreasing in $t$ when $t > \cos(\theta/2)$,
since this is the area of the circular segment
with chord at distance $t$ from the origin.
When $t \le \cos(\theta)$, we first observe that
\begin{align}
&\frac{d}{dt}\arcsin\left(\sqrt{1-t^2} \cos(\theta/2) - t \cdot \sin(\theta/2) \right)
\nonumber \\
=\,& \frac{\frac{d}{dt}\left[ \sqrt{1-t^2} \cos(\theta/2) - t \cdot \sin(\theta/2)\right]}
        {\sqrt{1 - (\sqrt{1-t^2} \cos(\theta/2) - t \cdot \sin(\theta/2))^2}}
\nonumber \\
=\,& \frac{-t(1-t^2)^{-1/2} \cos(\theta/2) - \sin(\theta/2)}
        {\sqrt{1 - \left(\sqrt{1-t^2} \cos(\theta/2) - t \cdot \sin(\theta/2)\right)^2}},
\end{align}
and the square of the denominator may be written more simply as
\begin{align}
& 1 - (\sqrt{1-t^2} \cos(\theta/2) - t \cdot \sin(\theta/2))^2
\nonumber \\
=\,& 1 - (1-t^2)\cos^2(\theta/2) - t^2 \sin^2(\theta/2) + 2 t \sqrt{1-t^2}\cos(\theta/2)\sin(\theta/2)
\nonumber \\
=\,& 1 - \cos^2(\theta/2) + t^2 \cos^2(\theta/2)- t^2 \sin^2(\theta/2)
    + 2 t \sqrt{1-t^2}\cos(\theta/2)\sin(\theta/2)
\nonumber \\
=\,& \sin^2(\theta/2) + t^2 \cos^2(\theta/2)- t^2 \sin^2(\theta/2)
    + 2 t \sqrt{1-t^2}\cos(\theta/2)\sin(\theta/2)
\nonumber \\
=\,& t^2 \cos^2(\theta/2) + (1-t^2) \sin^2(\theta/2)
    + 2 t \sqrt{1-t^2}\cos(\theta/2)\sin(\theta/2)
\nonumber \\
=\,& \left(t \cdot \cos(\theta/2) + \sqrt{1-t^2} \sin(\theta/2) \right)^2;
\end{align}
consequently,
\begin{align}
\frac{d}{dt}\arcsin\left(\sqrt{1-t^2} \cos(\theta/2) - t \cdot \sin(\theta/2) \right)
&= \frac{-t(1-t^2)^{-1/2} \cos(\theta/2) - \sin(\theta/2)}
        {t \cdot \cos(\theta/2) + \sqrt{1-t^2} \sin(\theta/2)}
\nonumber \\
&= \frac{-1}{\sqrt{1-t^2}}.
\end{align}

Therefore,
\begin{align}
&\frac{d}{dt} \left[ \arccos(t) -\arcsin\left(\sqrt{1-t^2} \cos(\theta/2) - t \cdot \sin(\theta/2) \right)  -  t^2 \tan(\theta/2) \right]
\nonumber \\
=\,& \frac{-1}{\sqrt{1-t^2}} 
    + \frac{1}{\sqrt{1-t^2}} 
    - 2 t \tan(\theta/2)
\nonumber \\
=\,& - 2 t \tan(\theta/2),
\end{align}
which is negative. Therefore, the maximum value of $\int_{\BB} \chi(x,y,t)dxdy$ occurs when $t=0$, where the value is
\begin{align}
2\arccos(0) - 2\arcsin\left(\cos(\theta/2) \right)
&= \pi - 2 \arcsin\left( \sin(\pi/2 + \theta/2) \right)
\nonumber \\
&= \pi - 2 \arcsin\left( \sin(\pi/2 - \theta/2)\right)
\nonumber \\
&= \pi - 2 \left( \frac{\pi}{2} - \frac{\theta}{2}\right)
\nonumber \\
&= \theta;
\end{align}
where we have used that $\sin$ is even around $\pi/2$. Note that
$\pi/2 - \theta/2$ lies between $0$ and $\pi/2$ (since $\theta$ is between $0$ and $\pi$),
and hence $\arcsin\left( \sin(\pi/2 - \theta/2)\right) = \pi/2 - \theta/2$, as claimed.
This completes the proof.

\subsection{Proof of Theorem \ref{thm:deformation}}
\label{proof:deformation} 

We will first prove that
\begin{align}
C(\Psi,r) \le \epsilon.
\end{align}

Let $I_x$ be the interval $[x,\Psi(x)]$ if $x \le \Psi(x)$, and
$[\Psi(x),x]$ if $\Psi(x) \le x$.
Let $\chi(x,t)$ be $1$ if $t \in I_x$, and $0$ otherwise;
then
\begin{align}
C(\Psi,r) = \max_{|t| \le r/2} \int_{\R} \chi(x,t)\, dt.
\end{align}

Take  $|t| \le r/2$. Suppose that there is some
$x \le t$ with $t \in I_x$; note that for such $x$, $I_x = [x,\Psi(x)]$,
and so $x \le \Psi(x)$. Let $x^*$ be the smallest such $x$.
Then $x^* \le t \le \Psi(x^*)$. We claim that
for all $x > t$, $t \notin I_x$.
Indeed, since $\Psi$ is increasing and $x > t \ge x^*$,
we have $\Psi(x) > \Psi(x^*) \ge t$.
Since both $x > t$ and $\Psi(x) > t$, $t$ does not lie in $I_x$, as claimed.

Consequently, all $x$ for which $t$ lies in $I_x$
are contained inside the interval $[x^*,t]$.
Since $x^* \le t \le \Psi(x^*)$ and $|x^* - \Psi(x^*)| \le \epsilon$,
it follows that $|t - x^*| \le \epsilon$ too. 
Furthermore, if $x > t$, then $\chi(x,t) = 0$ since $t \notin I_x$;
and since $x^*$ is the smallest $x$ for which $t \in I_x$,
if $x < x^*$ then $t \notin I_x$, hence $\chi(x,t) = 0$.
Therefore,
\begin{align}
\int_I \chi(x,t) dx \le \int_{x^*}^{t} 1 dx = |t-x^*| \le \epsilon,
\end{align}
and so $C(\Psi,r) \le \epsilon$.
Analogous reasoning yields the same bound in the case that
there exists $x \ge t$ with $t \in I_x$.

Therefore, Theorem \ref{thm:general} states that
\begin{align}
\label{eq:194920-0}
\|F - F_\Phi\|_{V^p}
&\le \epsilon^{1/p} \cdot \min\left\{\epsilon^{1-1/p} \cdot \|F\|_{L^p}, \
                                    \|F\|_{L^1}\right\}
\nonumber \\
&= \min\left\{\epsilon \cdot  \|F\|_{L^p},\ 
              \epsilon^{1/p} \cdot \|F\|_{L^1}\right\}.
\end{align}

Switching the roles of $\Psi$ and $\Phi$, and using that
$(F_\Phi)_\Psi = F$ and $\|F\|_{L^1} = \|F_\Phi\|_{L^1}$, shows the bound
\begin{align}
\label{eq:194920-1}
\|F - F_\Phi\|_{V^p}
= \|(F_\Phi)_\Psi - F_\Phi\|_{V^p}
\le \min\left\{\epsilon \cdot \|F_\Phi\|_{L^p},\ 
              \epsilon^{1/p} \cdot \|F\|_{L^1}\right\}.
\end{align}
Combining \eqref{eq:194920-0} and \eqref{eq:194920-1}
completes the proof.

\section{Proofs for Section \ref{section:discrete}}
\label{section:proofs_discrete}

First, we introduce some notation that will be useful for the proofs
in this section. For a function $f$ on $[a,b]$, define the vector
$\mV f$ in $\R^{n+1}$ with entries $(\mV f)[k] = (\V f)(a_k)$,
$0 \le k \le n$.

We also state the following simple lemma:

\begin{lem}
\label{lem:volterra_ones}
Suppose $1 \le p \le \infty$.
Let $\mathbf{1}$ be the all $1$'s vector
in $\R^{n+1}$. 
Then
\begin{math}
\|\mathbf{1}\|_{\nu_p} \le (b-a)^{1+1/p}.
\end{math}
\end{lem}

\begin{proof}
For $1 \le k \le n$,
\begin{align}
(\Vtrap \mathbf{1})[k]
= \frac{b-a}{2n} \sum_{j=0}^{k-1} (1 + 1)
= \frac{k}{n}(b-a),
\end{align}
and therefore
\begin{align}
\|\mathbf{1}\|_{\nu_p}
= \left( \frac{b-a}{2n} \sum_{k=0}^{n-1}
    \left[\left(\frac{k}{n}(b-a) \right)^p + \left(\frac{k+1}{n}(b-a) \right)^p\right]  \right)^{1/p}
\le (b-a)^{1+1/p},
\end{align}
as claimed.
\end{proof}

\subsection{Proof of Theorem \ref{thm:convergence_lipschitz}}
\label{proof:convergence_lipschitz}

To begin the proof of Theorem \ref{thm:convergence_lipschitz},
suppose $1 \le p < \infty$.

\begin{lem}
\label{lem:lipschitz}
Suppose $f$ satisfies the assumptions of Theorem \ref{thm:convergence_lipschitz}.
Let $1 \le m \le n$. Then
\begin{align}
\left|T_m(f,a,a_m) - (\V f)(a_m)  \right|
\le \frac{L(b-a)^2}{2n} + \frac{4 r \|f\|_{L^\infty} (b-a)}{n}.
\end{align}
\end{lem}

\begin{proof}
Because $f$ is bounded by $\|f\|_{L^\infty}$, for any  $0 \le j \le m$ we have the bound
\begin{align}
\left|\frac{b-a}{2n}\left(f(a_j) + f(a_{j+1})\right) - \int_{a_{j}}^{a_{j+1}} f(t) dt \right|
\le \frac{2\|f\|_{L^\infty}(b-a)}{n}.
\end{align}
On the other hand, if $j$ is such that
there are no $c_\ell$ in $[a_{j}, a_{j+1}]$,
then $f$ is $L$-Lipschitz on $[a_{j},a_{j+1}]$, and so
\begin{align}
\label{eq:401039401}
\left|\frac{b-a}{2n}\left(f(a_j) + f(a_{j+1})\right) - \int_{a_{j}}^{a_{j+1}} f(t) dt \right|
&= \frac{1}{2}\left|\int_{a_{j}}^{a_{j+1}} (f(a_j) - f(t) + f(a_{j+1}) - f(t)) dt \right|
\nonumber \\
&\le \frac{L}{2}\int_{a_{j}}^{a_{j+1}} (t-a_j + a_{j+1}-t) dt
\nonumber \\
&= \frac{L}{2}\int_{a_{j}}^{a_{j+1}} \frac{(b-a)}{n} dt
\nonumber \\
&= \frac{L(b-a)^2}{2n^2}.
\end{align}

Since there are at most $2 r$ intervals $[a_{j},a_{j+1}]$
containing a value from among $c_0,\dots,c_r$
(since each $c_1,\dots, c_r$ can be contained in at most $2$ such intervals),
and there are $m \le n$ subintervals in total, we have
\begin{align}
\label{eq:401039402}
\left|T_m(f,a,a_m) - (\V f)(a_m)  \right|
&= \left|\frac{b-a}{2n}\sum_{j=0}^{m-1} (f(a_j)+f(a_{j+1})) - \int_{a}^{a_{m}} f(t) dt \right|
\nonumber \\
&\le \sum_{j=0}^{m-1} \left|\frac{b-a}{2n}\left(f(a_j) + f(a_{j+1})\right)
    - \int_{a_{j}}^{a_{j+1}} f(t) dt \right|
\nonumber \\
&\le \frac{mL(b-a)^2}{2n^2} + \frac{4 r \|f\|_{L^\infty} (b-a)}{n},
\nonumber \\
&\le \frac{L(b-a)^2}{2n} + \frac{4 r \|f\|_{L^\infty} (b-a)}{n},
\end{align}
as claimed.
\end{proof}

By Lemma \ref{lem:lipschitz},
\begin{align}
\left| (\bV \bf)[m] - (\mV f)[m] \right|
\le \frac{L(b-a)^2}{2n} + \frac{4 r \|f\|_{L^\infty} (b-a)}{n}
\end{align}
Consequently,
\begin{align}
\label{eq:bound1}
\left| \|\bV \bf \|_{\tau_p} - \| \mV f \|_{\tau_p} \right|
&\le \|\bV \bf - \mV f \|_{\tau_p}
\nonumber \\
&\le (b-a)^{1/p}\|\bV \bf - \mV f \|_{\ell_\infty}
\nonumber \\
&\le \frac{L(b-a)^{2+1/p}}{2n} + \frac{4 r \|f\|_{L^\infty} (b-a)^{1+1/p}}{n}.
\end{align}

Next, we will show that
\begin{align}
\label{eq:bound2}
\left| \|f\|_{V^p}
    - \|\mV f\|_{\tau_p} \right|
\le \frac{(b-a)^{1+1/p}}{n} \|f\|_{L^\infty},
\end{align}
Combined with \eqref{eq:bound1}, this will conclude the proof.
To that end, we have the following lemma:

\begin{lem}
\label{lem:trap_pnorm}
Suppose $g$ has Lipschitz constant bounded by $A$ on $[a,b]$,
and let $\bg[k] = g(a_k)$,
where
\begin{align}
a_k = a + \frac{k}{n}(b-a), \quad 0 \le k \le n.
\end{align}
Then for any $1 \le p \le \infty$,
\begin{align}
\left| \|\bg\|_{\tau_p} - \|g\|_{L^p} \right|
\le \frac{(b-a)^{1+1/p}}{n} A.
\end{align}
\end{lem}

\begin{proof}
First, suppose $1 \le p < \infty$.
For each $0 \le m \le n$, let
\begin{align}
S_m = \left[ \left(\frac{b-a}{n} \right)
                \left(\frac{|g(a_m)|^p + |g(a_{m+1})|^p}{2} \right)\right]^{1/p},
\end{align}
and let
\begin{align}
R_m = \left( \int_{a_m}^{a_{m+1}} |g(x)|^p \,dx \right)^{1/p}.
\end{align}
Then
\begin{align}
\|\bg\|_{\tau_p} = \left(\sum_{m=0}^{n-1} |S_m|^p \right)^{1/p}
\end{align}
and
\begin{align}
\|g\|_{L^p} = \left(\sum_{m=0}^{n-1} |R_m|^p \right)^{1/p}.
\end{align}
The Mean Value Theorem
ensures that there is some $t_m$ in the interval $[a_m,a_{m+1}]$
satisfying
\begin{align}
R_m = \left(\frac{b-a}{n} \right)^{1/p}|g(t_m)|.
\end{align}

Then
\begin{align}
|S_m - R_m|
&=\left| S_m - \left(\frac{b-a}{n} \right)^{1/p} |g(t_m)| \right|
\nonumber \\
&= \left(\frac{b-a}{n} \right)^{1/p}
    \left| \left[\left(\frac{|g(a_m)|^p + |g(a_{m+1})|^p}{2} \right)\right]^{1/p}
    - \left[\left(\frac{|g(t_m)|^p + |g(t_{m})|^p}{2} \right)\right]^{1/p} \right|
\nonumber \\
&\le \left(\frac{b-a}{n} \right)^{1/p}\left( \frac{|g(a_m) - g(t_m)|^p
    + |g(a_{m+1}) - g(t_m)|^p}{2}\right)^{1/p}
\nonumber \\
&\le \left(\frac{b-a}{n} \right)^{1/p}\left( \frac{A^p |a_m-t_m|^p
    + A^p |a_{m+1} - t_m|}{2}\right)^{1/p}
\nonumber \\
&\le \left(\frac{b-a}{n} \right)^{1/p} A |a_{m+1} - a_m|
\nonumber \\
&= A \left(\frac{b-a}{n} \right)^{1+1/p}.
\end{align}

Consequently,
\begin{align}
\left|\|\bg\|_{\tau_p} - \|g\|_{L^p} \right|
&= \left|\left(\sum_{m=0}^{n-1} |S_m|^p \right)^{1/p}
    -  \left(\sum_{m=0}^{n-1} |R_m|^p \right)^{1/p}\right|
\nonumber \\
&\le \left(\sum_{m=0}^{n-1} |S_m - R_m|^p \right)^{1/p}
\nonumber \\
&\le \left(\sum_{m=0}^{n-1} A^p \left(\frac{b-a}{n} \right)^{p+1} \right)^{1/p}
\nonumber \\
&= \frac{(b-a)^{1+1/p}}{n} A.
\end{align}
This completes the proof when $p$ is finite.
The proof for $p=\infty$
follows by taking the limit $p \to \infty$
and using the convergence of the $p$-norm to the $\infty$-norm.
\end{proof}

Now, $\V f$ has Lipschitz constant bounded by $\|f\|_{L^\infty}$:
\begin{align}
\left| (\V f)(x) - (\V f)(y) \right|
= \left| \int_{a}^{x} f(t) dt  -  \int_{a}^{y} f(t) dt \right|
= \left| \int_{x}^{y} f(t) dt \right|
\le L |x-y|.
\end{align}
We may therefore apply Lemma \ref{lem:trap_pnorm} with $g = \V f$
and $A = \|f\|_{L^{\infty}}$
to show \eqref{eq:bound2}, thereby completing the proof of
Theorem \ref{thm:convergence_lipschitz} for $f$ and $\bf$.

To prove the result for $\fcen$ and $\bfcen$,
from Lemma \ref{lem:lipschitz} we have
\begin{align}
|\mu(f) - \mean(\bf)|
= \left| \frac{1}{b-a}(\V f)(b) - \frac{1}{b-a} T_n(f,a,b) \right|
\le \frac{L(b-a)}{2n} + \frac{4r \|f\|_{L^\infty}}{n}.
\end{align}
Letting $\wtilde \bf$ be the vector in $\R^{n+1}$ with entries
$\wtilde \bf[k] = \fcen(a_k) = f(a_k) - \mu(f)$, $0 \le k \le n$,
applying  the result we have already shown to $\fcen$ in place of $f$ gives
\begin{align}
\label{eq:8640}
\left| \|\wtilde \bf \|_{\nu_p} - \| \fcen \|_{V^p} \right|
\le C \frac{(b-a)^{1+1/p}}{n} \left( L (b-a) + r\|\fcen\|_{L^\infty}\right).
\end{align}
Furthermore, for all $0 \le k \le n$,
\begin{align}
\bfcen[k] - \wtilde \bf[k]
= \mu(f) - \mean(\bf),
\end{align}
and so, by Lemma \ref{lem:volterra_ones},
\begin{align}
\label{eq:8650}
\left| \|  \bfcen \|_{\nu_p} - \|\wtilde \bf \|_{\nu_p} \right|
&\le \| \bfcen - \wtilde \bf \|_{\nu_p}
\nonumber \\
&= \| (\mu(f) - \mean(\bf)) \mathbf{1}  \|_{\nu_p}
\nonumber \\
&= |\mu(f) - \mean(\bf)| \|\mathbf{1}\|_{\nu_p}
\nonumber \\
&\le |\mu(f) - \mean(\bf)| (b-a)^{1+1/p}
\nonumber \\
&\le C\frac{(b-a)^{1+1/p}}{n}
    \left(L(b-a) + r \|\fcen\|_{L^\infty} \right).
\end{align}
The result now follows by combining \eqref{eq:8640} and \eqref{eq:8650}.

\subsection{Proof of Theorem \ref{thm:convergence_smooth}}
\label{proof:convergence_smooth}

We will first prove a general result on approximating
the $L^p$ norm:

\begin{prop}
\label{prop:pnorm_trap}
Suppose $G$ is a $C^3$ function on $[a,b]$ with $G(a) = 0$.
Suppose $a = c_0 < c_1 < \dots < c_{r} = b$ are points such that,
for $0 \le j \le r-1$,
$G(c_j) = 0$ and $\sign(G)$ is constant on $(c_j,c_{j+1})$.
Let $\bg[k] = G(a_k)$, $0 \le k \le n+1$.

Let $1 \le p < \infty$.
Then for all $n$ sufficiently large,
\begin{align}
\left| \|\bg\|_{\tau_p} - \|G\|_{L^p} \right|
\le C\frac{(b-a)^2}{n^2} \left((b-a)^{1/p} \|G''\|_{L^\infty}
        + \frac{|G(b)|^{p-1}}{\|G\|_{L^p}^{p-1}}|G'(b)| \right),
\end{align}
where $C>0$ is a universal constant;
and
\begin{align}
\left| \|\bg\|_{\ell_\infty} - \|G\|_{L^\infty} \right|
\le C \frac{(b-a)^2}{n^2}  \|G''\|_{L^\infty},
\end{align}
for all  $n$ sufficiently large.
\end{prop}

\begin{proof}
Since the result is trivial if $G \equiv 0$, suppose that $G$
is not constantly zero.
First suppose $1 < p < \infty$, and let $H(x) = |G(x)|^p$
for $a \le x \le b$.
Then $H$ is continuous on $[a,b]$,
and, since the sign of $G$ is constant on
each interval $(c_j,c_{j+1})$, $0 \le j \le r-1$,
$H$ is $C^3$ on $(c_j,c_{j+1})$.

From (a slightly different form of) the Euler-Maclaurin
formula found in
\cite{berrut2006circular} and \cite{berrut2023extrapolation},
we may write the error of the trapezoidal rule approximation $T_n(H,a,b)$
to $\int_{a}^{b} H(x) dx$ as
\begin{align}
T_n(H,a,b) - \int_{a}^{b} H(x) dx
= \frac{\delta^2}{2} \sum_{k=0}^{r-1} P_2(t_k) \left[ H'(c_k^-) - H'(c_k^+)  \right]
    - \frac{\delta^2}{2}\int_{a}^{b} H''(x) P_2\left( \frac{x-a}{\delta}\right) dx,
\end{align}
where $\delta = (b-a)/n$;
$P_2$ is the second Bernoulli polynomial defined on $[0,1]$ and extended
$1$-periodically (that is, $P_2(x) = (x-1/2)^2 - 1/12$ on $[0,1]$,
and $P_2(x+k) = P_2(x)$ for all integers $k$);
$t_k = (c_k - a) / \delta$; and where $H'(c_0^-)$ is understood to denote $H'(b^-)$.

Because $p > 1$, $H$ is differentiable and
\begin{align}
H'(x) = p |G(x)|^{p-1} \sign(G(x)) G'(x).
\end{align}
Since $G(c_j) = 0$ when $0 \le j \le r-1$,
$H'(c_j^+) = \pm p |G(c_j)|^{p-1} G'(c_j^+) = 0$,
and for $1 \le j \le r-1$, $H'(c_j^-) = 0$ as well.
When $j=0$, $H'(c_0^-) = H'(b^-) = \pm p |G(b)|^{p-1} G'(b)$.
Therefore,
\begin{align}
\left|\frac{\delta^2}{2}\sum_{k=0}^{r-1} P_2(t_k) \left[ H'(c_k^-) - H'(c_k^+)  \right] \right|
= \frac{(b-a)^2}{2 n^2}p |P_2(1)|  |G(b)|^{p-1} |G'(b)|
= \frac{(b-a)^2 p  |G(b)|^{p-1} |G'(b)|}{12 n^2}.
\end{align}

Next, we will bound the term
\begin{align}
\frac{\delta^2}{2} \int_{a}^{b} H''(x) P_2\left( \frac{x-a}{\delta}\right) dx.
\end{align}

For $x$ in each open interval $(c_k,c_{k+1})$,
\begin{align}
H''(x) = p(p-1) |G(x)|^{p-2} G'(x)^2
            + p |G(x)|^{p-1} \sign(G(x)) G''(x).
\end{align}
Suppose first that $G > 0$ on $(c_j,c_{j+1})$, so that $H(x) = G(x)^p$.
Let $D(x) = G(x)^{p-1}$; then $D'(x) = (p-1) G(x)^{p-2} G'(x)$,
and $D(c_j) = D(c_{j+1}) = 0$.
Then, using integration by parts, and $\|P_2\|_{L^\infty} = 1/6$,
\begin{align}
\int_{c_j}^{c_{j+1}} p(p-1) G(x)^{p-2}  G'(x)^2 P_2\left( \frac{x-a}{\delta}\right) dx
&\le \frac{p}{6}\int_{c_j}^{c_{j+1}} (p-1) G(x)^{p-2}  G'(x)^2 dx
\nonumber \\
&= \frac{p}{6} \int_{c_j}^{c_{j+1}} D'(x)  G'(x) dx
\nonumber \\
&= \frac{p}{6}(D(c_{j+1}) G'(c_{j+1}) - D(c_{j}) G'(c_{j}))
    - \frac{p}{6} \int_{c_j}^{c_{j+1}} D(x) G''(x) dx
\nonumber \\
&= - \frac{p}{6} \int_{c_j}^{c_{j+1}} G(x)^{p-1} G''(x) dx.
\end{align}
Consequently, we have the upper bound
\begin{align}
\left| \int_{c_j}^{c_{j+1}} p(p-1) G(x)^{p-2}  G'(x)^2 P_2\left( \frac{x-a}{\delta}\right) dx \right|
&\le \frac{p\|G''\|_{L^\infty}}{6}  \int_{c_j}^{c_{j+1}} |G(x)|^{p-1}  dx;
\end{align}
Furthermore, we also have the bound
\begin{align}
\left| \int_{c_j}^{c_{j+1}} p G(x)^{p-1}  G''(x) P_2\left( \frac{x-a}{\delta}\right) dx \right|
&\le \frac{p\|G''\|_{L^\infty}}{6} \int_{c_j}^{c_{j+1}} |G(x)|^{p-1} dx.
\end{align}
Putting these together shows
\begin{align}
\left| \int_{c_j}^{c_{j+1}} H''(x) P_2\left( \frac{x-a}{\delta}\right) dx \right|
\le \frac{p \|G''\|_{L^\infty}}{3} \int_{c_j}^{c_{j+1}} |G(x)|^{p-1} dx.
\end{align}
The same bound may also be shown if $G < 0$ on $(c_j,c_{j+1})$
(and it is obvious if $G = 0$ on all of $(c_j,c_{j+1})$).

Consequently,
\begin{align}
\left|\frac{\delta^2}{2} \int_{a}^{b} H''(x) P_2\left( \frac{x-a}{\delta}\right) dx \right|
&= \left|\frac{\delta^2}{2} \sum_{j=0}^{r-1}
        \int_{c_j}^{c_{j+1}} H''(x) P_2\left( \frac{x-a}{\delta}\right) dx \right|
\nonumber \\
&\le \frac{\delta^2 p \|G''\|_{L^\infty}}{6}   \int_{a}^{b} |G(x)|^{p-1} dx
\nonumber \\
&\le \frac{\delta^2 p \|G''\|_{L^\infty}}{6} (b-a)^{1/p}
         \left(\int_{a}^{b} |G(x)|^{p} dx\right)^{1-1/p}
\nonumber \\
&= \frac{(b-a)^{2+1/p}}{6 n^2} p  \|G''\|_{L^\infty} \|G\|_{L^p}^{p-1}.
\end{align}

Applying the Euler-Maclaurin formula then gives
\begin{align}
\left| \|\bg\|_{\tau_p}^p - \int_{a}^{b} |G(x)|^p dx \right|
\le \frac{(b-a)^2}{12 n^2} p |G(b)|^{p-1} |G'(b)|
        + \frac{(b-a)^{2+1/p}}{6 n^2} p  \|G''\|_{L^\infty} \|G\|_{L^p}^{p-1}.
\end{align}
It follows that for all $n$ sufficiently large,
\begin{align}
\|\bg\|_{\tau_p}^p \ge \frac{1}{2} \|G\|_{L^p}^p.
\end{align}

Now, the function $y \mapsto y^{1/p}$ has derivative $(1/p)y^{1/p-1}$;
this is a  decreasing function, and hence its maximum value between $\|\bg\|_{\tau_p}^p$
and $\|G\|_{L^p}^p$ is bounded above by
\begin{align}
\frac{1}{p} \left(\frac{1}{2}\|G\|_{L^p}^p \right)^{1/p-1}
= \frac{2^{1-1/p}}{p} \|G\|_{L^p}^{1-p}.
\end{align}

By the mean value theorem, therefore,
\begin{align}
\left| \|\bg\|_{\tau_p} - \|G\|_{L^p} \right|
&\le \left(\frac{(b-a)^2}{12 n^2} p |G(b)|^{p-1} |G'(b)|
        + \frac{(b-a)^{2+1/p}}{6 n^2} p  \|G''\|_{L^\infty} \|G\|_{L^p}^{p-1} \right)
        \cdot \frac{2^{1-1/p}}{p} \|G\|_{L^p}^{1-p}
\nonumber \\
&\le C \frac{(b-a)^2}{n^2} \left( (b-a)^{1/p}\|G''\|_{L^\infty}
        + |G'(b)| \frac{|G(b)|^{p-1}}{\|G\|_{L^p}^{p-1}} \right),
\end{align}
where $C > 0$ is universal.
This completes the proof when $1 < p < \infty$.
The result for $p=1$  follows by taking the limit $p \to 1^+$.

When $p=\infty$, let $x^*$ satisfy $\|G\|_{L^\infty} = |G(x^*)|$.
If $x^* = b$, then, since $\bg[n] = G(b)$,
$\|\bg\|_{\ell_\infty} = |G(b)| = \|G\|_{L^\infty}$;
and since $G(a) = 0$, $x^*$ cannot equal $a$ unless $G \equiv 0$,
in which case the result is trivial.

Now suppose $a < x^* < b$. Then $G'(x^*) = 0$, and so
a second-order Taylor expansion gives
\begin{align}
|G(x) - G(x^*)| \le C\|G''\|_{L^\infty} |x - x^*|^2,
\end{align}
where $C > 0$ is universal.
Consequently, since there is a grid point $a_{k^*}$
within $(b-a)/n$ of $x^*$, we have
\begin{align}
|\bg[k^*] - G(x^*)| = |G(a_{k^*}) - G(x^*)| \le C \frac{(b-a)^2}{n^2} \|G''\|_{L^\infty},
\end{align}
and therefore,
\begin{align}
\left| \|\bg\|_{\ell_\infty} - \|G\|_{L^\infty} \right|
&= |G(x^*)| - \|\bg\|_{\ell_\infty}
\nonumber \\
&\le |G(x^*)| - |\bg[k^*]|
\nonumber \\
&\le |G(x^*) - \bg[k^*]|
\nonumber \\
&\le C\frac{(b-a)^2}{n^2} \|G''\|_{L^\infty},
\end{align}
which is the desired result.
\end{proof}

Applying Proposition \ref{prop:pnorm_trap} to $G(x) = (\V f)(x)$,
and using that $G'(x) = f(x)$ and $G(b) = \mu(f)$,
gives the bound
\begin{align}
\left| \|\mV f \|_{\tau_p} - \|f\|_{V^p} \right|
\le C \frac{(b-a)^2}{n^2} \left( (b-a)^{1/p}\|f'\|_{L^\infty}
        + |f(b)| \frac{|\mu(f)|^{p-1}}{\| f \|_{V^p}^{p-1}} \right)
\end{align}
when $1 \le p < \infty$ and for all $n$ sufficiently large,
\begin{align}
\left| \|\mV f \|_{\tau_p} - \|f\|_{V^\infty} \right|
\le C \frac{(b-a)^2}{n^2} \|f'\|_{L^\infty}
\end{align}
for all $n$.

To finish the proof, we will show that
\begin{align}
\left| \|\mV f \|_{\tau_p} - \|\bf\|_{\nu_p} \right|
\le C\|f''\|_{L^\infty} \frac{(b-a)^{3+1/p} }{n^2}.
\end{align}

\begin{lem}
\label{lem:midpoint_rule}
Let $0 \le m \le n$. Then
\begin{align}
\left| (\V f)(a_{m}) - (\bV \bf)[m] \right|
\le C\|f''\|_{L^\infty} \frac{ (b-a)^3 }{n^2},
\end{align}
where the constant $C > 0$ is universal.
\end{lem}

\begin{proof}
When $m=0$, the left side is 0. When $m \ge 1$,
using a standard error estimate for
the trapezoidal rule
(see Section \ref{sec:trap_rule})
and the fact that $a_m - a = m(b-a)/n$, we get
\begin{align}
\left| (\V f)(a_m) - (\bV \bf)[m] \right|
&=\left| \int_{a}^{a_m} f(t) dt -  T_m(f,a,a_m) \right|
\nonumber \\
&\le  C\|f''\|_{L^\infty} \frac{(a_m-a)^3 }{m^2},
\nonumber \\
&=  C\|f''\|_{L^\infty} \frac{m (b-a)^3 }{n^3}
\nonumber \\
&\le  C\|f''\|_{L^\infty} \frac{(b-a)^3 }{n^2},
\end{align}
as claimed.
\end{proof}

Using Lemma \ref{lem:midpoint_rule}, we have
\begin{align}
\label{eq:bound10201}
\left| \|\mV f \|_{\tau_p} - \|\bf\|_{\nu_p} \right|
&= \left| \| \mV f \|_{\tau_p}  - \|\bV \bf \|_{\tau_p}\right|
\nonumber \\
&\le \|\mV f - \bV \bf \|_{\tau_p}
\nonumber \\
&\le (b-a)^{1/p}\|\mV f - \bV \bf \|_{\ell_\infty}
\nonumber \\
&\le C\|f''\|_{L^\infty} \frac{(b-a)^{3+1/p} }{n^2}.
\end{align}
This completes the proof of Theorem \ref{thm:convergence_smooth}
for the uncentered function $f$.

To prove the result for $\fcen$ and $\bfcen$,
let $\wtilde \bf$ have entries
$\wtilde \bf[k] = \fcen(a_k) = f(a_k) - \mu(f)$.
Applying the result already shown to $\fcen$ in place of $f$,
and noting that $\mu(\fcen) = 0$, gives
\begin{align}
\label{eq:9640}
\left| \|\wtilde \bf \|_{\nu_p} - \| \fcen\|_{V^p} \right|
\le C\frac{(b-a)^{2+1/p}}{n^2}
        \left( (b-a)\|\fcen''\|_{L^\infty} + \|\fcen'\|_{L^\infty}\right).
\end{align}

Since $f$ is $C^2$,
\begin{align}
|\mu(f) - \mean(\bf)|
= \left| \frac{1}{b-a}\int_{a}^{b} f(x) dx - T_n(f,a,b) \right|
\le C\|f''\|_{L^\infty}\frac{(b-a)^2}{n^2}
= C\|\fcen''\|_{L^\infty}\frac{(b-a)^2}{n^2},
\end{align}
where $C > 0$ is universal.
Furthermore, for all $0 \le k \le n$,
\begin{align}
\bfcen[k] - \wtilde \bf[k]
= \mean(\bf) - \mu(f),
\end{align}
and so, by Lemma \ref{lem:volterra_ones},
\begin{align}
\label{eq:9650}
\left| \| \wtilde \bf \|_{\nu_p} - \| \bfcen \|_{\nu_p} \right|
&\le \| \wtilde \bf - \bfcen \|_{\nu_p}
\nonumber \\
&= \| (\mean(\bf) - \mu(f)) \mathbf{1}  \|_{\nu_p}
\nonumber \\
&= |\mean(\bf) - \mu(f)| \| \mathbf{1}  \|_{\nu_p}
\nonumber \\
&\le |\mean(\bf) - \mu(f)| (b-a)^{1+1/p}
\nonumber \\
&\le C\|f''\|_{L^\infty}\frac{(b-a)^{3+1/p}}{n^2}.
\end{align}
The result now follows by combining \eqref{eq:9640} and \eqref{eq:9650}.

\subsection{Proofs of Theorem \ref{thm:noise} and Corollary \ref{cor:convergence}}
\label{proof:noise}

Let $T[0] = 0$, and for $1 \le k \le n$, let
\begin{align}
T[k]
&= \frac{b-a}{n} \sum_{j=0}^{k-1} \frac{Z[j] + Z[j+1]}{2}
\nonumber \\
&= \frac{b-a}{2n} \sum_{j=0}^{k-1} Z[j] + \frac{b-a}{2n} \sum_{j=1}^{k} Z[j]
\nonumber \\
&= \frac{1}{2}\left(S_0[k-1] + S_1[k] \right),
\end{align}
where
\begin{align}
S_0[k] = \frac{b-a}{n}\sum_{j=0}^{k} Z[j], \quad 0 \le k \le n-1,
\end{align}
and
\begin{align}
S_1[k] = \frac{b-a}{n}\sum_{j=1}^{k} Z[j], \quad 1 \le k \le n.
\end{align}

\begin{lem}
For any $t > 0$,
\begin{align}
\Prob\left( \max_{0 \le k \le n-1} \left| S_0[k] \right| \ge t\right)
\le 2\exp(-n t^2 / 2 (b-a)^2 \sigma^2),
\end{align}
and
\begin{align}
\Prob\left( \max_{1 \le k \le n} \left| S_1[k] \right| \ge t\right)
\le 2\exp(-n t^2 / 2 (b-a)^2 \sigma^2).
\end{align}
\end{lem}

\begin{proof}
The method of proof is fairly standard; see, for instance, \cite{revuz2005continuous}.
Let $\lambda > 0$, and define
\begin{align}
X[k] = \exp(\lambda S_0[k]), \quad 0 \le k \le n-1.
\end{align}
Then $X$ is a submartingale, i.e.\
$\EE[X[k] \,|\, Z[0],\dots,Z[k-1]] \ge X[k-1]$ for $1 \le k \le n-1$.
Observe that $S_0[n-1]$ is normally distributed with mean zero
and with variance
\begin{align}
\overline \sigma^2 = \frac{(b-a)^2}{n^2} \sum_{j=0}^{n-1} \sigma_j^2
\le \frac{(b-a)^2}{n} \sigma^2.
\end{align}
Consequently, using the standard formula for the Gaussian moment generating function,
\begin{align}
\EE[X[n-1]] = e^{\lambda^2 \overline{\sigma}^2 / 2}
\le e^{\lambda^2 (b-a)^2 \sigma^2 / 2 n}.
\end{align}
By Doob's Inequality (e.g.\ see Theorem 5.4.2 in \cite{durrett2019probability}),
for any real number $t$,
\begin{align}
\Prob\left( \max_{0 \le k \le n-1} S_0[k] \ge t\right)
&= \Prob\left( \max_{0 \le k \le n-1} X[k] \ge \exp(\lambda t) \right) 
\nonumber \\
&\le \EE[X[n-1]] \exp(-\lambda t )
\nonumber \\
&\le e^{\lambda^2 (b-a)^2 \sigma^2 / 2 n-\lambda t}.
\end{align}
Taking $\lambda = t n / \sigma^2 (b-a)^2$ yields the bound
\begin{align}
\Prob\left( \max_{0 \le k \le n-1} S_0[k] \ge t\right)
\le \exp(-n t^2 / 2 (b-a)^2 \sigma^2).
\end{align}
Symmetry and the union bound immediately gives the bound
\begin{align}
\Prob\left( \max_{0 \le k \le n-1}|S_0[k]| \ge t\right)
\le 2\exp(-n t^2 / 2 (b-a)^2 \sigma^2).
\end{align}
An identical argument holds for $S_1$, completing the proof.
\end{proof}

Since $T[0]=0$ and $T[k] = (S_0[k-1] + S_1[k])/2$ when $k \ge 1$, and since 
\begin{align}
\|Z\|_{\nu_\infty} = \max_{0 \le k \le n} |T[k]|,
\end{align}
the union bound shows
\begin{align}
\Prob\left( \|Z\|_{\nu_\infty} \ge t\right)
= \Prob\left( \max_{0 \le k \le n}|T[k]| \ge t\right)
\le 4\exp(-n t^2 / 2 (b-a)^2 \sigma^2).
\end{align}
This establishes \eqref{eq:concentration_bound} for $p = \infty$;
the result then follows for all $p \ge 1$ since
$\|Z\|_{\nu_p} \le \|Z\|_{\nu_\infty}$.

To see that the rest of the theorem follows from
\eqref{eq:concentration_bound}, observe that
since the right side of \eqref{eq:concentration_bound}
is summable over $n$, it follows from the Borel-Cantelli Lemma \cite{chandra2012borel}
that
\begin{math}
\lim_{n \to \infty}\| Z \|_{\nu_p} = 0
\end{math}
almost surely, establishing \eqref{eq:as_limit}. Furthermore,
\begin{align}
\EE [\| Z \|_{\nu_p}] 
&= \int_{0}^{\infty} \Prob\left(\| Z \|_{\nu_p}  \ge t \right) dt
\nonumber \\
&\le 2 \int_{0}^{\infty} \exp(-n t^2 / 2 (b-a)^2\sigma^2) dt
\nonumber \\
&=  \frac{\sigma (b-a)}{\sqrt{n}} 2 \int_{0}^{\infty} \exp(- u^2 / 2) du,
\end{align}
which establishes \eqref{eq:expectation} and completes the proof of the theorem for $Z$. The corresponding results for $Z_\cen$ may be deduced from those of $Z$
and the standard concentration bound
for $\mean(Z) \sim N(0, \sigma^2/n)$:
\begin{align}
\Prob \{ |\mean(Z)| > t \} \le 2 e^{- t^2 n / 2 \sigma^2}.
\end{align}
(See, e.g., Chapter 2 in \cite{wainwright2019high}.)
This completes the proof of Theorem \ref{thm:noise}.

To prove \eqref{eq:concentration2}, recall that
Theorem \ref{thm:convergence_lipschitz} gives the bound
\begin{align}
\left| \|f\|_{V^p} - \|\bf\|_{\nu_p} \right| \le \frac{C}{n},
\end{align}
where $C$ is a constant not depending on $p$, $t$ or $n$.
From the triangle inequality we have
\begin{align}
\label{eq:93100-1}
\| Y \|_{\nu_p} - \|f\|_{V^p}
&= \| \bf + Z \|_{\nu_p} - \|f\|_{V^p}
\nonumber \\
&\le \| \bf  \|_{\nu_p} + \| Z \|_{\nu_p} - \|f\|_{V^p}
\nonumber \\
& \le \frac{C}{n} + \|Z\|_{\nu_p},
\end{align}
and similarly, since $\|\bf\|_{\nu_p} - \|Z\|_{\nu_p} \le \|Y\|_{\nu_p}$,
\begin{align}
\label{eq:93100-2}
\|f\|_{V^p} - \|Y\|_{\nu_p}
&\le \|f\|_{V^p} - \| \bf\|_{\nu_p} + \|Z\|_{\nu_p}
\nonumber \\
&\le \frac{C}{n} + \|Z\|_{\nu_p}.
\end{align}
Combining \eqref{eq:93100-1} and \eqref{eq:93100-2} shows
\begin{align}
\label{eq:530520-1}
\left| \| Y \|_{\nu_p} - \|f\|_{V^p}\right| \le \frac{C}{n} + \|Z\|_{\nu_p},
\end{align}

If $t - C/n \ge t/2$, which holds for all $n$
sufficiently large,
then from Theorem \ref{thm:noise},
\begin{align}
\Prob\left\{\left| \|f\|_{V^p} - \|Y\|_{\nu_p} \right| \ge t \right\}
&\le \Prob\{\|Z\|_{\nu_p} \ge t - C/n\}
\nonumber \\
&\le \Prob\{\|Z\|_{\nu_p} \ge t/2\}
\nonumber \\
&\le A e^{-B n (t/2)^2 / \sigma^2},
\end{align}
which is a bound of the desired form.
This completes the proof of \eqref{eq:concentration2}.
The limit \eqref{eq:as_limit2} follows immediately from \eqref{eq:530520-1}
and the fact that $\|Z\|_{\nu_p} \to 0$ almost surely.
To prove \eqref{eq:expectation2}, 
take expectations of each side of \eqref{eq:530520-1} and apply
\eqref{eq:expectation}. The proofs of the corresponding results
for $Y_\cen$ and $\fcen$ are nearly identical.

\section{Conclusion}
\label{section:conclusion}

This paper has proven a number of robustness properties of the Volterra distances
for functions of a single variable, and the sliced Volterra distances for
functions of multiple variables. 
These results extend previous results known for Wasserstein distances.
Our results indicate that the favorable properties of Wasserstein distances may be
shared by a wider class of metrics, which may be better suited for certain applications;
for instance, the Volterra metrics are defined between functions
with negative values and unequal integrals, and are less sensitive to
large deformations of the data.

The Volterra metrics are extremely simple: one merely applies
a smoothing filter to the input functions, and then computes their
Lebesgue distance. It seems likely that
one could prove similar results for other
families of filters.
As such, the present work suggests that
constructing metrics with a desired set of robustness properties
may only require applying an appropriate collection of filters
to the input data, where the filters are designed to
smooth the functions with respect to the underlying geometry
of their domain.
In fact, many metrics that have been considered
in recent years are of exactly this type \cite{mishne2016hierarchical,
mishne2017datadriven,
leeb2016holder, leeb2018mixed, jacobs2008approximate, mishne2019comanifold}.

It is also natural to explore applications
of the Volterra distances, sliced Volterra
distances, and related metrics to problems where
Wasserstein and sliced Wasserstein distances
have been used previously.
One such area is analysis of data from cryo-electron
microscopy (cryo-EM), in which one observes two-variable projections
of a three-variable volume (a molecule),
at unknown viewing directions, from which the volume
is to be determined \cite{singer2020computational, bendory2020single,
doerr2016single}.
Wasserstein metrics have been proposed for clustering
images and parameterizing volumes in cryo-EM
\cite{singer2023, rao2020wasserstein, zelesko2020earthmover}.
Due to their robustness to deformations,
Volterra metrics, or other metrics with similar
properties, may also be appropriate
for heterogeneity analysis in cryo-EM
\cite{zhong2021cryodrgn, frank2016continuous, liao2015efficient,
aizenbud2019maxcut, lederman2020hyper, scheres2016processing, toader2023methods,
lederman2020representation},
as has been proposed for Wasserstein distances \cite{rao2020wasserstein}.
More generally, it is of interest to explore applications
to clustering, nearest
neighbor regression, and other metric-based tasks.
Questions along these lines will be pursued in future work.

\subsection*{Acknowledgements}

I thank Joe Kileel, Amit Moscovich, Rohan Rao, and Amit Singer
for stimulating discussions on the papers
\cite{rao2020wasserstein}, \cite{kileel2020manifold}, and \cite{zelesko2020earthmover}.
I give additional thanks to Amit Singer
for helpful feedback and suggestions on earlier versions of the manuscript.
I acknowledge support from  NSF BIGDATA award IIS-1837992, BSF award 2018230,
and NSF CAREER award DMS-2238821.

\bibliographystyle{plain}
\bibliography{refs_norms}

\end{document}